\documentclass[a4paper, twoside, 11pt]{amsart}
\usepackage[margin=1 in]{geometry}                
\usepackage{mathtools}
\mathtoolsset{showonlyrefs,showmanualtags}
\usepackage[colorlinks,linkcolor=blue,citecolor=blue,urlcolor=blue]{hyperref}
\usepackage{graphicx}
\usepackage{amssymb, amsmath}
\usepackage{soul}
\usepackage{mathrsfs}
\usepackage{enumitem}
\usepackage{multirow}
\usepackage[usenames,dvipsnames]{xcolor}
\usepackage{bbm}
\usepackage{tikz-cd}
\usetikzlibrary{3d,calc,tikzmark,decorations.markings,arrows,shapes.misc}
\usepackage{yfonts}
\usepackage{etoolbox}
\usepackage{stackengine}[2013-10-15]
\usepackage{mdframed}
\usepackage{bm}
\usepackage[square,numbers]{natbib}
\usepackage{mathdots}
\usepackage{arydshln}
\usetikzlibrary{3d,calc}
\usepackage[toc,page]{appendix}

\numberwithin{equation}{section}

\newcommand{\downmapsto}{\rotatebox[origin=c]{-90}{$\scriptstyle\mapsto$}\mkern2mu}

\newcommand{\Z}{\mathbb Z}

\newcommand{\R}{\mathbb R}
\newcommand{\C}{\mathbb C}

\newcommand{\Hom}{\text{Hom}}

\newcommand{\ad}{\text{\textnormal{ad}}}

\newtheorem{theorem}{Theorem}[section]
\newtheorem{lemma}[theorem]{Lemma}

\newtheorem{definition}[theorem]{Definition}

\newtheorem{proposition}[theorem]{Proposition}
\newtheorem{corollary}[theorem]{Corollary}
\newtheorem{remark}[theorem]{Remark}

\theoremstyle{definition}
\newtheorem{example}{Example}[section]

\usepackage[textwidth=1.8cm]{todonotes}  
\usepackage{marginnote}
\usepackage{zref-savepos}

\newcount\todocount

\newcommand{\checkxpos}[3][]{%
  \ifdim \zposx{#2}sp < 20000000sp%
    \mynote[#1]{#3}%
  \else%
    \note[#1]{#3}%
  \fi%
}

\newcommand{\mytodo}[2][]{%
  \zsaveposx{todo\the\todocount}%
  \checkxpos[#1]{todo\the\todocount}{#2}%
  \global\advance\todocount1\relax
}

\newcommand{\mynote}[2][]{{%
  \let\marginpar\marginnote
  \reversemarginpar
  \renewcommand{\baselinestretch}{0.8}%
  \todo[#1]{#2}}}

\newcommand{\note}[2][]{\renewcommand{\baselinestretch}{0.8}\todo[#1]{#2}}

\begin{document}
\title[On geometry of $2$-nondegenerate CR structures]{On geometry of $2$-nondegenerate CR structures of hypersurface type and flag structures on leaf spaces of Levi foliations}

\date{}
\thanks{I.\ Zelenko is supported by Simons Foundation Collaboration Grant for Mathematicians 524213.}

\author{David Sykes}
\address{David Sykes,
	Department of Mathematics and Statistics 
	Masaryk University,
	Kotl\'{a}\v{r}sk\'{a} 2,
	611 37 Brno,
	Czech Republic}\email{ sykes@math.muni.cz}
\urladdr{\url{http://www.math.muni.cz/~sykes}}


\author{Igor Zelenko}
\address{Igor Zelenko, Department of Mathematics
	Texas A\&M University
	College Station
	Texas, 77843
	USA}\email{ zelenko@math.tamu.edu}
\urladdr{\url{http://www.math.tamu.edu/~zelenko}}


\begin{abstract}
We construct canonical absolute parallelisms over real-analytic manifolds equipped with $2$-nondegenerate, hypersurface-type CR structures of arbitrary odd dimension not less than $7$ whose Levi kernel has constant rank belonging to a broad subclass of CR structures that we label as \emph{recoverable}. For this we develop a new approach based on a reduction to a special flag structure, called the \emph{dynamical Legendrian contact structure}, on the leaf space of the CR structure's associated Levi foliation. This extends the results of \cite{porter2017absolute} from the case of regular CR symbols constituting  a discrete set in the set of all CR symbols to the case of the arbitrary CR symbols for which the original CR structure can be uniquely recovered from its corresponding dynamical Legendrian contact structure. We find an explicit criterion for this recoverability. In particular, if the rank of the Levi kernel is 1 and the dimension of the CR manifold is not less than 7, then for each given signature of the reduced Levi form  in the space of all CR symbols (which depend on continuous parameters) there are no more than 2 symbols for which the aforementioned recoverability fails, and while the present method is applicable for all but those 2 cases, they can be treated separately by the method of \cite{porter2017absolute}.  Our method clarifies the relationship between the bigraded Tanaka prolongation of regular symbols developed in \cite{porter2017absolute} and their usual Tanaka prolongation, providing a geometric interpretation of conditions under which they are equal. Motivated by the search for homogeneous models with given nonregular symbols, we also describe a process of reduction from the original natural frame bundle, which is inevitable for structures with nonregular CR symbols. We demonstrate this reduction procedure for examples whose underlying manifolds have dimension $7$ and $9$. 
 
We  show  that, for any fixed rank $r>1$, in the set of all CR symbols associated with 2-nondegenerate, hypersurface-type CR manifolds of odd dimension greater than $4r+1$ with rank $r$ Levi kernel, the CR symbols not associated with any homogeneous model are generic, and, for $r=1$, the same result holds if the CR structure is pseudoconvex.
\end{abstract}

\maketitle\tableofcontents

\section{Introduction}\label{introduction}

A \emph{Cauchy--Riemann  structure of hypersurface type} on a $(2n+1)$-dimensional real manifold $M$ is an integrable,  totally real, complex rank $n$ distribution $H$ contained in the complexified tangent bundle $\C TM$, that is,
\begin{align}\label{levi form}
[H,H]\subset H\quad\mbox{ and } \quad H\cap\overline{H}=0.
\end{align}
Here $\C T_pM=T_pM\otimes \mathbb C$. 

Fixing some notation, if $E$ is a fiber bundle over a base space $B$ then we let $\Gamma(E)$ denote the space of smooth sections of  $E$, and, for $p\in B$, we let $E_p$ denote the fiber of $E$ over $p$. For a given CR structure $H$, a Hermitian form $\mathcal{L}$, called the \emph{Levi form} of the CR structure $H$ on $M$, is defined on fibers of $H$ by the formula
\begin{equation}
\label{Leviformdef}
\mathcal{L}(X_p,Y_p):=\frac{i}{2}\left[X,\overline{Y}\right]_p\mod H_p\oplus \overline{H}_p\quad\quad\forall\, X,Y\in \Gamma(H)\mbox{ and }p\in M,
\end{equation}
where $\mathcal{L}$ takes values in the quotient spaces $\C T_pM/\left(H_p\oplus \overline{H}_p\right)$. Note that the coset represented by $\frac{i}{2}[X,\overline{Y}]_p$ in  $\C T_pM/\left(H_p\oplus \overline{H}_p\right)$ depends only on the values of $X$ and $Y$ at $p$ rather than the values of $X$ and $Y$ in a neighborhood of $p$.
Letting $K$ denote the kernel of the Levi form,
a CR structure with $K_p=0$ is called \emph{Levi-nondegenerate} at $p$, and it is called degenerate at $p$ if $K_p\neq 0$.

The equivalence problem for Levi-nondegenerate CR structures of hypersurface type is classical. \'{E}. Cartan solved it for hypersurfaces in $\mathbb C^2$ \cite{cartanCR}, and then Tanaka \cite{tanakaCR} and Chern and Moser \cite{chernmoserCR} generalized the solution to hypersurfaces in $\mathbb{C}^{n+1}$ for $n\geq1$. This case is well understood in the general framework of parabolic geometries \cite{tanaka2, capshichl, capslovak}.

In the present paper we are interested in the construction of absolute parallelisms and finding upper bounds for the dimension of groups of local symmetries for CR structures that are uniformly (i.e., at every point) Levi-degenerate, but satisfy further nondegeneracy conditions called $2$-nondegeneracy. Assuming that $K$ is a distribution, that is, $\dim K_p$ is independent of $p\in M$,  $2$-nondegeneracy can be defined as follows:  We say that the CR structure $H$ on $M$ is \emph{2-nondegenerate at a point $p\in M$} if $K_p\neq0$ and, for any $Y\in \Gamma(K)$ with $Y_p\neq 0$, there exists $X\in \Gamma(H)$ such that $\left[X,\overline{Y}\right]_p\not\in K_p\oplus \overline{K}_p \oplus \overline {H}_p$.  The structure $H$ is  \emph{2-nondegenerate} if it is 2-nondegenerate at every point in $M$. 

Equivalently, if for $v\in K_p$ and $y\in\overline{H}_p/\overline{K}_p$, we take $V\in\Gamma(K)$ and $Y\in\Gamma(\overline{H})$ such that $V(p)=v$ and $ Y(p)\equiv y\mod\overline{K}$, and define a linear map
\begin{equation}\label{adv}
\begin{aligned}
 \ad_v&:\overline{H}_p/\overline{K}_p\to H_p/K_p\\
 &\hspace{7mm}y\hspace{6.5mm}\mapsto [V,Y]|_p \,\, \mathrm{mod}\,\, K_p\oplus \overline {H}_p,
\end{aligned}
\end{equation}
and similarly define a linear map $\ad_{v}:H_p/K_p\to \overline{H}_p/\overline{K}_p$ for $v\in \overline K_p$ (or simply take complex conjugates), then 
a Levi-degenerate CR structure is $2$-nondegenerate at $p$ if and only if there is no nonzero $v\in K_p$ (equivalently, no nonzero $v\in \overline{K}_p$) such that  $\ad_v=0$.

The generalization of this definition to arbitrary $k\geq 1$ via the \emph{Freeman sequence} under analogous constant rank assumptions was given in \cite{freeman1977local}. A more general definition,  without the assumption that $K$ is a distribution and for arbitrary $k\geq 1$ can be found in the monograph \cite[chapter XI]{BER99}.   Note that our  definition of $2$-nondegeneracy (the Freeman definition of $k$-nondegeneracy) and the definition in \cite{BER99} under the assumption that $K$ is a distribution (respectively, of constancy of ranks in the Freeman sequence) are equivalent (see \cite[Appendix]{kaupzaitsev}).

As a direct consequence of the Jacobi identity for every $v\in \overline K_p$ the antilinear operator $\overline {\ad_v}: H_p/K_p\to H_p/K_p$, defned by  $\overline {\ad_v} (x):= \overline {\ad _v (x)}$, is a self-adjoint antilinear operator with respect to the Hermitian form $\ell$ induced on $H_p/ K_p$ by the Levi form $\mathcal L$ on $H_p$, that is,
$$\ell(\overline{\ad _v} x, y)= \ell(\overline{\ad _v} y,x), \quad \forall x, y\in  H_p/ K_p.$$

If $H$ is a hypersurface-type CR structure, $n=\mathrm{rank}\, H$, and $r=\mathrm{rank}\, K$, then the assumption of 2-nondegeneracy implies that 
\begin{equation}
\label{estim1}
\dbinom{n-r+1}{2}\geq r,
\end{equation}
as the left side is exactly the dimension of the space of self-adjoint antilinear operators on $H/K$ (equal to the dimension of $(n-r+1)\times(n-r+1)$ symmetric matrices) and this must not be less than $\mathrm{rank}\, K$.  

This implies in particular that among hypersurface-type CR manifolds, the lowest dimension in which $2$-nondegeneracy can occur is $\dim M=5$ (i.e., with $n=2$ and $r=1$).  The structure of absolute parallelism in this case  was constructed only recently  and independently in the following three papers (preceded by the work \cite{ebenfelt} for a more restricted class of structures): Isaev and Zaitsev \cite{isaev},  Medori and Spiro \cite{medori}, and Merker and  Pocchiola \cite{pocchiola}.

The most general results on a canonical  absolute parallelism for $2$-nondegenerate, hypersurface-type CR structures of dimension higher than $5$  (and without an assumption of semisimplicity of the symmetry group of homogeneous models as in \cite{greg1, greg2} and implicitly in \cite{porter}) were obtained by \cite{porter2017absolute}, where under specific algebraic conditions a bigraded (i.e., $\mathbb Z\times \mathbb Z$-graded) analogue of Tanaka's prolongation procedure to construct a canonical absolute parallelism for these CR structures in arbitrary (odd) dimension with Levi kernel of arbitrary admissible dimension was developed. The starting point of these constructions  was the introduction of the notion of a bigraded Tanaka symbol of a CR structure at a point, playing the role of the Tanaka symbol in the standard Tanaka theory  \cite{tanaka1, zelenko2009tanaka}, which is not applicable here.  As with the usual Tanaka symbol, the bigraded Tanaka symbol contains the information about brackets of sections adapted to a filtration (determined by the CR structure) that remains after a passage from the filtered structure to a natural bigraded structure at a point (see Section \ref{Symbols of 2-nondegenerate CR structures section} for more detail), but in contrast to the standard theory the  bigraded symbol is not a Lie algebra in general, but rather a bigraded vector space.

The algebraic assumption of \cite{porter2017absolute} under which the bigraded Tanaka prolongation approach works is that the bigraded CR symbol is a Lie algebra. Such a symbol is called regular. 
Yet, for fixed $n=\mathrm{rank}\,H$ and $r=\mathrm{rank}\, K$ satisfying \eqref{estim1}, apart from the case in which the equality in \eqref{estim1} holds, that is, when $r=\tbinom{l+1}{2}$ and $n=\tfrac{l(l+3)}{2}$ for some positive integer $l$, which was treated in \cite{greg1, greg2}, the nonregular symbols constitute a generic subset in the set of all symbols (see Lemma \ref{gennonreglem} for the proof), and the goal of the present paper is to treat structures exhibiting nonregular CR symbols. For this, in the real-analytic category, we develop an alternative approach based on a (local) reduction of  the original CR structure to a sort of flag structure in the spirit of \cite{dz13}  (see Definition \ref{flagdef} below)  or, equivalently,  to families of Legendrian contact structures (following the terminology of \cite{doubrov2019homogeneous}, see Remark \ref{Legendre} below) on the space of leaves of the Levi foliation (or shortly the Levi leaf space) of the complexified manifold. We call these flag structures \emph{dynamical Legendrian contact structures}.  Specifically, the Levi leaf space is endowed with a contact distribution, and within the Lagrangian Grassmannian of each fiber of this contact distribution (defined with respect to its canonical conformal symplectic form) the original CR structure induces a pair of  submanifolds with complex dimension equal to the rank of the Levi kernel (see Section \ref{sec2} for more detail).  
In particular, if the Levi kernel is one-dimensional, then at each fiber of the contact distribution one has a pair of curves of Lagrangian subspaces.

In Section \ref{sec2}, we give criteria (Proposition \ref{prolK}) for when by passing from the CR structure to the corresponding dynamical Legendrian contact structure  we do not lose any information, that is, the  former can be uniquely recovered from the latter. In particular, we show that in the case of $\mathrm{rank}\, K=1$ the CR structure is recoverable if and only if, for a generator $v$ of $K$, the operator $\mathrm{ad}_v$ has rank greater than $1$ and, consequently, in this case every CR structure with nonregular CR symbol is recoverable.  Moreover, for fixed $\mathrm{rank}\,H >1$ and given signature of the reduced Levi form (i.e., the Hermitian form induced on $H/K$ from the Levi form),  among all regular CR symbols, if the reduced Levi form is sign-indefinite (Figure \ref{fig1}) then there are exactly two symbols for which the operator $\mathrm{ad}_v$ has rank $1$ and are consequently non-recoverable, whereas if the reduced Levi form is sign-definite then there is exactly one such symbol. The two non-recoverable symbols arising in the sign-indefinite case are distinguished by whether the antilinear operator $\overline{\ad_v}$  is  nilpotent  or not. The former is not possible in the sign-definite case.

\begin{figure}

\begin{align}\label{region diagram}
\parbox{\textwidth}{
\begin{center}
\begin{tikzpicture}
\path[shade] (0,0)  to[out=30,in=-90] (5,0)  to[out=90,in=-135] (5,3) to[out=60,in=0] (2.5,3) to[out=180,in=210]  (0,0) -- cycle;
\begin{scope}[thick,decoration={
    markings,
    mark=at position 0.5 with {\arrow{>}}}
    ] 
\draw[] (0,0)  to[out=30,in=-90] (5,0);
\draw[] (5,0)  to[out=90,in=-135] (5,3);
\draw[] (2.5,3)  to[out=180,in=210]  (0,0);
\draw[] (5,3)  to[out=60,in=0] (2.5,3);
\end{scope}
\draw 
    (-2.5,1) node{$\left.\parbox{3cm}{Finitely many re-\\gular and recoverable symbols}\right\}$};
\draw 
    (6.8,1) node{$\left\{\parbox{2.6cm}{At most 2 non-recoverable symbols}\right.$};
\draw[thick]
    (4.3,2.9) circle(2pt);
\draw[thick]
    (2.8,2.4) circle(2pt);
\draw[fill]
    (1,1.6) circle(1.5pt);
\draw[fill]
    (1.7,1.1) circle(1.5pt);
\draw[fill]
    (2.9,0.4) circle(1.5pt);
\draw[fill]
    (3.7,0.2) circle(1.5pt);
\draw[fill]
    (2.3,0.7) circle(0.5pt);
\draw[fill]
    (2.2,0.75) circle(0.5pt);
\draw[fill]
    (2.4,0.65) circle(0.5pt);
\begin{scope}[decoration={
    markings,
    mark=at position 0.999 with {\arrow{>}}}
    ] 
\draw[postaction={thick,decorate}] (-0.85,1)  to[out=0,in=-90]  (1,1.55);
\draw[postaction={thick,decorate}] (-0.85,1)  to[out=0,in=-160]  (1.66,1.07) ;
\draw[postaction={thick,decorate}] (-0.85,1)  to[out=0,in=190]  (2.85,0.4);
\draw[postaction={thick,decorate}] (-0.85,1)  to[out=0,in=200] (3.64,0.18) ;
\draw[postaction={thick,decorate}] (5.2,1)  to[out=180,in=-90] (4.3,2.8)  ;
\draw[postaction={thick,decorate}] (5.2,1)  to[out=180,in=-10]  (2.9,2.4)  ;
\end{scope}
    \end{tikzpicture}
    \end{center}
    }
\end{align}
\caption{Moduli space of CR symbols for $\mathrm{rank}\, K=1$, fixed $\dim M>5$, and fixed signature of $\mathcal{L}$}
    \label{fig1}
\end{figure}
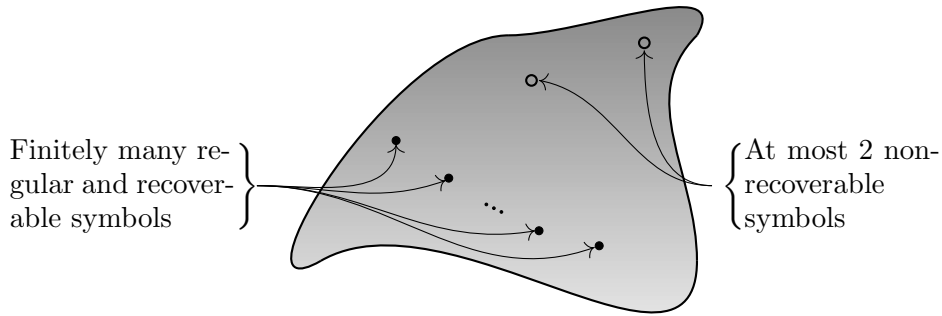

Inspired by the theory developed in \cite{doubrov2012geometry} for geometry of a single submanifold in flag varieties, we apply a description of the local differential geometry of pairs of submanifolds in Lagrangian Grassmannians to assign to our original structure a sort of Tanaka structure, in general of nonconstant type and with the symbol at every point different from the original CR symbol of \cite{porter2017absolute} (therefore called the \emph{modified  symbol}), for which  both the Tanaka-like prolongation procedure for the construction of canonical moving frames and the upper bounds for the pseudogroup of local symmetries can be established (see Theorem \ref{maintheorm}).  A nonstandard aspect in this theorem is that the modified symbol of our structure is not necessarily a Lie algebra and it varies from point to point; so to prove it (see Section \ref{Absolute parallelisms}) we needed to make certain modifications to the standard Tanaka prolongation in the spirit of \cite{zelenko2009tanaka}, obtaining a microlocal version of the standard construction.

In Section \ref{CR structures with constant modified symbols},  we prove that structures with nonregular CR symbols cannot have constant modified symbols (Theorem \ref{constant mod symbols implies regular}), a notion introduced in Section \ref{bundles on curves in Lagrangian Grassmannians}, which motivates the reduction procedure of Section \ref{reduction_and_nongenericity_section}. In particular, Section \ref{reduction_and_nongenericity_section} introduces another theorem on absolute parallelisms (Theorem \ref{maintheorconst}), which gives more precise upper bounds for the dimension of algebras of infinitesimal symmetries than Theorem \ref{maintheorm} does in certain cases. As a consequence, we get that if the CR symbol $\mathfrak g^0$ is regular and recoverable then its usual Tanaka prolongation and the bigraded Tanaka prolongation defined in \cite[section 3]{porter2017absolute} coincide.

Although to every regular CR symbol one can assign a homogeneous model, the existence of homogeneous models exhibiting a given nonregular CR symbol turned out to be a subtle question.  In Section \ref{nongenericity section}, we  show  (Theorem \ref{gen_nonreg}) that for any fixed rank $r>1$, in the set of all CR symbols associated with 2-nondegenerate, hypersurface-type CR manifolds of odd dimension greater than $4r+1$ with rank $r$ Levi kernel, the CR symbols not associated with any homogeneous model are generic, and for $r=1$ the same results holds if the reduced Levi form is sign-definite, that is, when the CR structure is pseudoconvex.

Despite these non-existence results for generic symbols,  such homogeneous models do exist for specific nonregular symbols. In  Section \ref{examples} and Section \ref{Examples of the canonical parallelism construction},  we demonstrate our constructions with four examples. All three examples of Section  \ref{examples} are actually homogeneous CR manifolds exhibiting the maximally symmetric structures described in Theorem \ref{maintheorconst}, and they illustrate novel applications of this paper's main results. 
Example \ref{7d nonregular} has a nonregular CR symbol, and, as such, the method developed in this paper is the only known way to build an absolute parallelism over such CR manifolds such that the parallelism's automorphisms are all induced by its underlying CR manifold's symmetries. The other two examples in Section \ref{examples}  have the same regular CR symbol in the sense of \cite[Definition 2.2]{porter2017absolute} but different modified CR symbols, and, as such, while the construction of an absolute parallelism given in \cite{porter2017absolute} is the same for both examples, the construction given here varies, resulting in parallelisms of different dimensions for each example and whose dimension matches that of the underlying CR manifold's symmetry group. In Section \ref{Examples of the canonical parallelism construction} we analyze a class of  $2$-nondegenerate CR structures of hypersurface type with non-regular symbol on $11$-dimensional manifolds, that are in general nonhomogeneous, and for this class we implement our prolongation procedure for construction of absolute parallelisms, finding the structures' fundamental set of invariants and classifying all homogeneous models. 

The classification of  CR symbols that admit homogeneous models is a nontrivial problem even for small dimensions and  is treated fully in all dimensions for regular symbols in \cite{porter2017absolute} and partially for nonregular symbols in the paper \cite{SZ2}, where the full classification is given in dimension $7$ and a general method for construction of such symbols in arbitrary dimension is developed. 
The natural question arises:  does nonexistence of homogeneous models with a given symbol imply nonexistence of any structure having this symbol at every point? This question  turns to be  subtle even in small dimensions, as it leads to the study of certain rather nontrivial overdetermined systems of PDEs, and it is still an open question. Nevertheless, the construction of absolute parallelisms in the present paper also gives a natural framework for answering this question reframing it as one on existence of solutions of the system of overdetermined PDE's given by the Bianchi identity corresponding to the canonical absolute parallelism, which can be studied with Cartan-K\"ahler theory. This is still a work in progress, and, in the light of this yet unsolved issue, in  Appendix \ref{Non-constant symbol structures}, we also sketch how our absolute parallelism construction can be extended to structures with non-constant CR symbol.

\section {The Levi leaf space and its flag structure}
\label{sec2}
From now on we assume that $K$ is a distribution of rank $r$, that is, $\dim K_p=r$ for all $p\in M$. Note that directly from the definition \eqref{Leviformdef} it follows that $K$ is an involutive distribution. In this section we introduce an important geometric object, the space of leaves of the foliation by integral submanifolds of the distribution $K\oplus\overline{K}$, called the Levi leaf space.   Since $K$ and $\overline{K}$ are subbundles in the complexified tangent bundle $\C TM$, to define such leaves we must ``complexify" the manifold $M$, at least locally. For this to work we have to assume that all considered objects, namely the manifold $M$ and the CR-structure given by $H$, are real-analytic. Under the real-analytic assumption, locally (i.e., in some neighborhood of any point in $M$) we can consider a complex manifold $\mathbb C M$, a \emph{complexification of $M$}, by extending the transition maps between charts, which are real-analytic by definition, to analytic functions of complex variables. We can then extend locally the real-analytic distributions $H$ and $K$ to the holomorphic  distributions on $\C M$ which, for simplicity, will be denoted by the same letters $H$ and $K$. The conjugation in local charts of $\C M$ defines an involution $\tau$ on $\C M$ such that $M$ is the set of its fixed points. Using this involution we can extend $\overline{H}$ and $\overline{K}$ by the formulas
\begin{equation}
\label{antihol} 
\overline{H}:=\tau_*(H) \quad\mbox{ and }\quad \overline{K}:=\tau_*(K),
\end{equation}
so $\overline H$ and $\overline K$ are antiholomorphic extensions of the corresponding distributions from $M$ to $\C M$.

Furthermore,  distributions $H\oplus\overline H$ and $K\oplus \overline K$ in $\C M$ are holomorphic as they are holomorphic extensions of the  real parts of the corresponding distributions on $M$. Also note that the constructed extended distribution $K$ on $\C M$ is involutive as is $K\oplus \overline K$. So $\C M$ is foliated by the maximal integral (complex) submanifolds of $K\oplus \overline K$. This foliation is called the \emph{Levi foliation} and will be denoted by $\mathrm{Fol}(K\oplus\overline K)$ and, after an appropriate shrinking of $\C M$, which always can be done as our considerations are local, we can assume that the space of leaves of this foliation
$$\mathcal N=\C M / \mathrm{Fol}(K\oplus\overline K)$$
has a natural structure of a (complex) manifold. The manifold $\mathcal N$ is called the \emph{Levi leaf space}  of the original CR structure.

Let $\pi: \C M\rightarrow \mathcal{N}$  be the natural projection, sending a point $p\in \C M$ to the leaf  of the Levi foliation passing through $p$. Since, by construction, for every vector field $X\in \Gamma(K\oplus\overline K)$, we have 
\begin{equation}
\label{Cauchy}
[X, H\oplus \overline H]\subset \Gamma(H\oplus\overline H).
\end{equation}
The set
\begin{equation}
\label{contactD}
\mathcal{D}:=
\pi_* \left(H\oplus\overline{H}\right) 
\end{equation}
is a well defined (complex) corank 1 distribution on $\mathcal N$.

Moreover, since $X\in \Gamma(H\oplus\overline H)$ satisfies \eqref{Cauchy} if and only if  $X\in \Gamma(K\oplus\overline K)$,  the distribution $\mathcal D$ is contact, that is, if $\alpha$ is a $1$-form on $\mathcal N$  annihilating $\mathcal D$, then $d\alpha|_{D_\gamma}$ is nondegenrate at every point $\gamma\in \mathcal N$. The form  
\begin{equation}
\label{sigma}
\omega_\gamma: =d\alpha|_{D_\gamma}
\end{equation}
is, up to a multiplication by a nonzero constant, a well defined symplectic form on $\mathcal D_\gamma$, that is, it  defines a \emph{canonical conformal symplectic structure} on $\mathcal D_\gamma$.

For every $\gamma\in \mathcal N$  and every $p\in \pi^{-1}(\gamma)$,  considered as the leaf of the foliation $\mathrm{Fol}(K\oplus\overline{K})$ in $\mathbb C M$, set
$$\widehat J_\gamma^-(p):=\pi_* H_p, \quad \ \widehat J_\gamma^+(p):=\pi_* \overline H_p$$
From the involutivity of the distributions $H$ and $\overline H$, it follows that  
$J_\gamma^-(p)$ and  $J_\gamma^+(p)$ are Lagrangian subspaces of $\mathcal D_\gamma$ with respect to the symplectic form $\omega_\gamma$, that is, they are elements of the Lagrangian Grassmannian $\mathrm{LG}\bigl(\mathcal{D}_\gamma\bigr)$.

Finally, the distributions $K$ and $\overline{K}$ are involutive and define foliations $\mathrm{Fol}(K)$ and $\mathrm{Fol}(\overline{K})$, respectively. Obviously the leaves of $\mathrm{Fol}(K)$ (and of $\mathrm{Fol}(\overline K)$) foliate the leaves of $\mathrm{Fol}(K\oplus \overline K)$. Since $[K, H]\subset H$, the space $J_\gamma^-(p)$ is the same for every $p$ in the same leaf of $\mathrm{Fol}(K)$ in $\pi^{-1}(\gamma)$ for $\gamma\in \mathcal N$. Hence, we can define the map 
 \begin{equation}
 \label{Jgamma1} 
 J_\gamma^-:\pi^{-1}(\gamma)/\mathrm{Fol}(K)\rightarrow \mathrm{LG}\bigl(\mathcal D_\gamma\bigr)
 \end{equation}
 such that,  given  $ p\in \pi^{-1}(\gamma)/\mathrm{Fol}(K)$, we have $\widehat{J}_\gamma^{-}(p):=\widehat J_\gamma^{-} (\hat p)$ for some (and therefore any) $\widehat p\in \C M$ lying on the leaf containing $p$ of the foliation $\mathrm{Fol} (K)$.
 
 \begin{remark}
 \label{tangentid}
 Recall that the tangent space $T_\Lambda \mathrm{LG}(\mathcal D_\gamma)$ to the Lagrangian Grassmannian $\mathrm{LG}(\mathcal D_\gamma)$ at the point $\Lambda$ is identified with an appropriate subspace in $\mathrm {Hom}(\Lambda, \mathcal D_\gamma/\Lambda)$. Also, for every $p\in \pi^{-1}(\gamma)$, the map $(\pi_*)_p$ identifies $H_p/K_p$ with $J_\gamma^{-}(p)$ and $\overline H_p/\overline K_p$ with $J_\gamma^{+}(p)$. Using these identifications and basic properties of Lie derivatives, for every $v\in K_p$ (or $v\in \overline K_p$)  we can identify the operator  $\mathrm{ad}_v$ defined by \eqref{adv} with the operator $\left(J_\gamma^+\right)_*v $ (or respectively $\left(J_\gamma^-\right)_*v $).
 \end{remark}

 By the identification of the previous remark,  the  $2$-nondegeneracy condition implies that, after an appropriate  shrinking of $\C M$, the map $J_\gamma^-$ from \eqref{Jgamma1} is a well defined injective immersion, that is, its image is a submanifold of $LG(\mathcal D_\gamma)$ of complex dimension equal to $\mathrm{rank}\, K$. Similarly, the map 
  \begin{equation}
 \label{Jgamma2}
 J_\gamma^+:\pi^{-1}(\gamma)/\mathrm{Fol}(\overline{K})\rightarrow \mathrm{LG}\bigl(\mathcal D_\gamma\bigr).
 \end{equation}
 is a well defined injective immersion and its image is a submanifold of $\mathrm{LG}(\mathcal D_\gamma)$ of complex dimension equal to $\mathrm{rank}\, K$ as well. In the sequel, by $J_\gamma^-$ and $J_\gamma^+$  we will denote the images of the maps in \eqref{Jgamma1} and \eqref{Jgamma2}, respectively, rather than the maps themselves.

\begin{remark}
\label{Legendre} Note that, by construction, if $\Lambda^-\in J_\gamma^-$ and  $\Lambda^+\in J_\gamma^+$ then $\Lambda^-$ and $\Lambda^+$ are transversal as subspaces of $\mathcal D_\gamma$, that is, 
$\mathcal D_\gamma =\Lambda^-\oplus\Lambda^+$.
Recall (\cite{doubrov2019homogeneous}) that  a Legendrian contact structure on an odd dimensional distribution is a contact distribution $\Delta$ together with the fixed splitting of each fiber $\Delta_x$ by a pair  of transversal  Lagrangian subspaces smoothly depending on $x$. Any section $s$ of the bundle $\pi:\C M\rightarrow \mathcal N$, defines the Legendrian contact structure on $\mathcal N$ given by the distribution $\mathcal D$ and the splitting of $\mathcal D_\gamma$, given by $J_\gamma^-\bigl(s(\gamma) \bigr)$ and  $J_\gamma^+\bigl(s(\gamma) \bigr)$.
\end{remark}

Motivated by the previous constructions and Remark \ref{Legendre} we introduce the following definition.
\begin{definition}
\label{dyndef}
A dynamical Legendrian contact structure (with involution) on an odd-dimensional complex manifold $\mathcal M$  is a contact distribution 
$\Delta$ together with an involution $\sigma$  on $\mathcal M$ and  a fixed pair of $k$-dimensional submanifods $ \Lambda_x^-$ and $ \Lambda_x^+$ in the Lagrangian Grassmannian $\mathrm {LG}(\Delta_x)$ of each fiber such that the following conditions hold:
\begin{enumerate}
\item the submanifolds $ \Lambda_x^-$ and $ \Lambda_x^+$ are smoothly dependent on $x$ and any point of $ \Lambda_x^-$, considered as a Lagrangian subspace of $\Delta_x$, is transversal to any point  of $ \Lambda _x^-$, considered as a Lagrangian subspace of $\Delta_x$. 
\item $\Lambda_x^-=\sigma_*\Lambda_\sigma(x)^+$.
\end{enumerate}
Such dynamical Legendrian contact structures with involution will be denoted just by the triple  $(\Delta, \Lambda^-, \Lambda^+)$ when the involution is determined by context or by the triple $(\Delta, \Lambda^-,\tau)$.
\end{definition}

\begin{definition}
\label{flagdef}
Letting $H$ be a 2-nondegenerate hypersurface-type CR structure on the manifold $M$, the dynamical Legendrian contact structure $\{\mathcal D, J^{-}, J^{+}\}$ on the Levi leaf space $\mathcal{N}$, with $\mathcal D$ , $J ^-$, $J^+$, and involution $\tau$ defined by \eqref{contactD}, \eqref{Jgamma1}, \eqref{Jgamma2}, and the sentence before \eqref{antihol}, respectively, is called \emph{the dynamical Legendrian contact structure associated with the germ (at some point in $M$) of the CR structure $H$.} 
\end{definition}
The reason that, in Definition \ref{flagdef}, we say that dynamical Legendrian contact structures are associated with germs of CR structures rather than with the whole CR manifold, is that the construction of $\{\mathcal D, J^{-}, J^{+}\}$ is well defined only after an appropriate shrinking of $M$ (i.e., after possibly replacing $M$ by a neighborhood of any given point in $M$).
\begin{remark}
\label{recover_rem}
If $\{\mathcal D, J^{-}, J^{+}\}$ is the dynamical Legendrian contact structure associated with the germ of the 2-nondegenerate CR structure $H$ of hypersurface type on the manifold $M$ then $\C M$ is locally canonically diffeomorphic to the bundle $\Pi: J^-\times J^+\to \mathcal{N}$ with the fiber over the point $\gamma\in\mathcal{N}$ equal to $J_\gamma^{-}\times J_\gamma^{+}$, where here $J^-\times J^+:=\displaystyle{\bigcup_{\gamma\in\mathcal{N}}}\left(J^{-}_\gamma\times J^{+}_\gamma\right)$. This canonical diffeomorphism, denoted by $F$, is given by 
\begin{equation}
\label{canF}
F(p):=(J_\gamma^-(p), J_\gamma^+(p)).
\end{equation}
Moreover, each fiber $J_\gamma^{-}\times J_\gamma^{+}$ is foliated by two foliations with leaves $\{J_\gamma^{-}\times J_\gamma^{+}(p)\}_{p\in \pi^{-1} (\gamma)}$ and $\{J_\gamma^{-}(p)\times J_\gamma^{+}\}_{p\in \pi^{-1} (\gamma)}$, respectively.  Denote by $V_1$ and $V_2$ the distribution tangent to this foliation. Then it is clear that \begin{equation}
\label{recov_eq}
\begin{split}
~& V_1=F_*K, \quad V_2=F_*\overline K,\\
~&F_*(H_p+\overline K_p)=\left\{y\in T_{F(p)} (J^{-}\times J^{+})\,\left|\, \Pi_*y \in J_{\pi(p)}^-(p)\right.\right\},\quad\mbox{ and}\\
~&F_*(\overline H_p+K_p)=\left\{y\in T_{F(p)} (J^{-}\times J^{+})\,\left|\, \Pi_*y \in J_{\pi(p)}^+(p)\right.\right\}.
\end{split}
\end{equation}
 Finally the distribution $H$ is an involutive subdistribution of $H+\overline K$.
 \end{remark}

The main idea of the present paper is to study the local equivalence problem for the dynamical Legendrian contact structures associated with CR structures instead of the CR structures themselves. Before doing this, we have to understand the conditions under which passing from the CR structure to the corresponding dynamical Legendrian contact structure does not lose any information, that is, under which the former can be uniquely reconstructed from the latter. Such CR structures will be called \emph{recoverable}.

To describe the conditions for recoverability, recall (\cite{stern}) that, given two vector spaces $V$ and $W$ and a subpsace $Z$ in $\mathrm{Hom}(V, W)$, the \emph{anti-symmetrization}  (Spencer) operator $\partial: \mathrm{Hom}(V, Z) \rightarrow \mathrm{Hom}(V\wedge V, W)$ is defined by
\begin{equation}
\label{Spencer}
\partial (f)(v_1, v_2)=f(v_1)v_2-f(v_2)v_1, \quad v_1, v_2\in V, \,\,f\in \mathrm{Hom}(V, Z).
\end{equation}
The kernel of the operator $\partial$ is called the \emph{first prolongation} of the subspace $Z\subset \mathrm{Hom}(V, W)$ and is denoted by $Z_{(1)}$.

Now for $v\in \overline K_p$, take $\ad_{v}:H_p/K_p\to \overline{H}_p/\overline{K}_p$ to be as in the sentence after \eqref{adv}. From the assumption
of $2$-nondegeneracy it follows that the map $v\mapsto \ad_v$ identifies $\overline K_p$ with  a subspace in $\mathrm{Hom}(H_p/K_p, \overline{H}_p/\overline{K}_p)$, which is denoted by $\ad \, \overline K_p$.

\begin{proposition}
\label{prolK}
A $2$-nondegenerate hypersurface-type CR structure $H$ is recoverable in a neighborhood of a point $p$  if and only if the first prolongation $\bigl(\ad \overline K_p\bigr)_{(1)}$ of the space $\ad \overline K_p$ vanishes.
\end{proposition}
\begin{proof}
From Remark \ref{recover_rem} it follows that 
a CR structure  $H$ is recoverable 
if and only if
$H$ is the unique involutive subdistribution of $H+\overline K$ of rank $n=\mathrm{rank}\, H$, transversal to $\overline K$  and containing $K$, because the reconstruction can be obtained using formulas \eqref{recov_eq}.  Let $H^\prime$ be another complex complex rank $n$  involutive subdistrubution  of  $H+\overline K$ that is transversal  to $\overline K$ and containing $K$. Fix the point $p\in \C M$. For each $y\in \C M$, there exists a linear map $f_y:H_y/K_y\to \overline K_y$ such that $H^{\prime}$ is the graph of $f_y$ characterized by
\[
H^{\prime}_y=\left\{v+f_y(v)\,\left|\, v\in H_y\right.\right\}.
\]

In the sequel by $f$ we mean the field of linear maps $\{f_y\}_{y\in \C M}$.
For two vectors $y_1, y_2 \in H_p$, let $Y_1, Y_2 \in  \Gamma\left(H\right)$ be such that $Y_i(p)=y_i$ for $i\in\{1,2\}$, and let $Y_1^{\prime}$ and $Y_2^{\prime}$ be the associated vector fields in $\Gamma\left(H^{\prime}\right)$ such that 
\[
Y_i^{\prime}=Y_i+f\left(Y_i\right).
\]
We have
\begin{align}
\left[Y_1^{\prime},Y_2^{\prime}\right]_p&=\left[Y_1,Y_2\right]_p+\left[Y_1,f\left(Y_2\right)\right]_p+\left[f\left(Y_1\right),Y_2\right]+\left[f\left(Y_1\right),f\left(Y_2\right)\right]_p\label{4terms}\\
&\equiv \left[Y_1,f\left(Y_2\right)\right]_p+\left[f\left(Y_1\right),Y_2\right]_p\pmod{H_p+\overline K_p} \label{2terms}
\end{align}
because $\left[Y_1,Y_2\right]$ and $\left[f\left(Y_1\right),f\left(Y_2\right)\right]$
both belong to $H_p+\overline K_p$ due to the involutivity of $H$ and $\overline K$. Since $H^{\prime}$ is involutive,  it follows that the left side of \eqref{4terms} belongs to $H^\prime_p$. Hence, using \eqref{adv}, \eqref{2terms} can be written as
\begin{equation}\label{uniqueness of reconstruction lemma formula 1}
\mathrm{ad}_{f_p(y_1)} y_2-\mathrm{ad}_{f_p(y_2)} y_1=0.
\end{equation}
Since $y_1$ and $y_2$ are arbitrary elements of $H_p/K_p$, by \eqref{Spencer} and \eqref{uniqueness of reconstruction lemma formula 1}, we get that $f_p\in \bigl(\ad \overline K_p\bigr)_{(1)}$, so the vanishing of $\bigl(\ad \overline K_p\bigr)_{(1)}$ for generic $p$ is equivalent to $H^\prime=H$.
\end{proof}
Based on this proposition we obtain the following sufficient condition for the recoverability of the CR structure.

\begin{proposition}\label{uniqueness of holomorphic bundle} If for a generic point $p$ there is no  nonzero subspace  $L$ of $\overline K_p$ satisfying,
\begin{equation}
\label{inter_cond}
\dim\left(\bigcap_{v\in L} \ker \ad_v\right)\geq \mathrm{rank}\, H-
\mathrm{rank}\, K
-\dim  L,
\end{equation}
then the original CR structure given by $H$ is recoverable. 
\end{proposition}
\begin{proof}
By Proposition \ref{prolK} it is sufficient to prove that if \eqref{inter_cond} cannot be satisfied for some $L$ then $\bigl(\ad \overline K_p\bigr)_{(1)}=0$ for generic $p$.  To prove the contrapositive of this statement, assume that for a generic $p$ there is a nonzero 
$f \in \bigl(\ad \overline K_p\bigr)_{(1)}$. Set $L=\mathrm{Im} f$. Then
\begin{equation}
\label{kerf}
    \dim \ker f=\mathrm{rank}\, H-
    \mathrm{rank}\, K 
    -\dim L
\end{equation}
On the other hand, if $y\in \ker f$, then from the definition of the first prolongation for every $z\in H_p/K_p$
$$\ad_{f(z)}y=\ad_{f(z)}y-\ad_{f(y)}z=0,$$
because $f(y)=0$, which means that $y\in \ker \ad_v$ for every $v\in L$, that is,
$$\ker f\subset \left(\bigcap_{v\in L_p} \ker \ad_v\right). $$
This and  \eqref{kerf} implies  \eqref{inter_cond}, that is if $\bigl(\ad \overline K_p\bigr)_{(1)}\neq 0$ then there exists a space $L$ satisfying \eqref{inter_cond}, which is the contrapositive of what we needed to prove.
\end{proof}

\begin{corollary}
\label{rank1cor}
If $\mathrm{rank}\, K=1$ then  the original CR structure given by $H$ is recoverable 
if and only if $\ad \overline K_p$ is generated by an operator of rank greater than $1$ for a generic $p$.
\end{corollary}
\begin{proof}
First note that for  $\mathrm{rank}\, K=1$ the only nonzero subspace of $\overline K_p$ is $\overline K_p$ itself and the inequality \eqref{inter_cond} is equivalent to the condition that $\mathrm{ad} K_p$ is generated by an operator of rank $1$. This and Proposition \ref{uniqueness of holomorphic bundle} imply the ``if'' part  of the corollary.

To prove the ``only if'' part assume by contradiction that $\ad \overline  K_p$ is generated by rank 1 operator $\mathrm{ad}_{v}$, $v\in \overline K_p$ for generic $p$. Accordingly, $f\in \mathrm{Hom}(H_p/K_p),  \ad \overline  K_p)$ such that $\ker f=\ker \mathrm{ad}_v$ will be a nonzero element of $\ad \bigl(\ad \overline K_p\bigr)_{(1)}=0$. Hence, by Proposition \ref{prolK} the CR structure is not uniquely recoverable, which leads to the contradiction.
\end{proof}
\begin{remark}
\label{reformulations} 
Propositions \ref{prolK}, \ref{uniqueness of holomorphic bundle}, and Corollary \ref{rank1cor} can be reformulated in an obvious way to define CR structures in terms of their dynamical Legendrian contact structures, specifically, using the identifications in Remark \ref{tangentid} and the formulas in \eqref{recov_eq}. In all such reformulations we have to replace ``unique" recovery by ``at most one," because in general a dynamical Legendrian contact structure need not be associated with any CR structure as the distribution in the right side of the second line of \eqref{recov_eq} may not contain any involutive subdistribution of rank equal to the dimension of $J_\gamma^-(p)$. 
\end{remark}

\section{Symbols of 2-nondegenerate CR structures}\label{Symbols of 2-nondegenerate CR structures section}

Let us assume that $(\mathcal{D}, J^-,J^+)$ is a dynamical Legendrian contact structure with involution $\tau$ on $\mathcal{N}$ associated with a germ of some 2-nondegenerate, hypersurface-type CR structure $H$ on a manifold $M$.

For $\gamma\in\mathcal{N}$ and $p\in \pi^{-1}(\gamma)$, note that there is an identification $J^+_\gamma(p) \cong \big(J^-_\gamma(p)\big)^*$ determined by the symplectic form $\omega_\gamma)$. Using this identification, the tangent space $T_{J^{-}_\gamma(p)} J^{-}_{\gamma}$ can be canonically identified with a subspace of $\mathrm{Sym}^2 (J_\gamma^-(p))^*\big)\subset \mathrm{Hom} \big(J^{-}_\gamma(p),J^{+}_\gamma(p)\big)$, which will be denoted by $\delta^{-}(p)$. 

Note that any element $y$ of $\mathrm{Sym}^2 \big((J_\gamma^-(p))^*\big)$, considered as a self-adjoint operator from $J_\gamma^-(p)$ to $(J^-_\gamma(p))^*\cong J^+_\gamma(p)$ can be extended to an element $\tilde y$ of the conformal symplectic algebra $\mathfrak{csp}(\mathcal D_\gamma)$ by setting $\tilde y|_{J^-_\gamma(p)}=y$ and $\tilde y|_{J^+_\gamma(p)}=0$. 
In the sequel we will regard $\delta^{-}(p)$ as an element of $\mathfrak{csp}(\mathcal D_\gamma)$. 

The tangent space $T_{J^{+}_\gamma(p)} J^{+}_{\gamma}$ can be canonically identified with a subspace of $\mathrm{Hom} \big(J^{+}_\gamma(p),J^{-}_\gamma(p)\big)$ similarly, which will be denoted by $\delta^{+}(p)$ and will be also considered as a subspace of $\mathfrak{csp}(\mathcal D_\gamma)$. Note that $\delta^{+}(p)$ is obtained from $\delta^{-}(p)$ by the involution of $\mathfrak{csp}(\mathcal{D}_\gamma)$ induced by the differential of the involution $\tau$ whenever $p$ belongs to the fixed point set of $\tau$.

\begin{definition}
\label{symb_def}
The symbol of the 2-nondegenerate, hypersurface-type CR structure $H$ on a manifold $M$ at a point $p$ in $\C M$ is the orbit of the pair of subspaces $\big(\delta^{-}(p), \delta^{+}(p)\big)$ under the action of the conformal symplectic group $CSp(\mathcal D_\gamma)$ on $\mathfrak{csp}(\mathcal D_\gamma)\times \mathfrak{csp}(\mathcal D_\gamma)$ induced by the $\mathrm{Ad}$-action  of $CSp(\mathcal D_\gamma)$ on $\mathfrak{csp}(\mathcal D_\gamma)$. The 2-nondegenerate, hypersurface-type CR structure $H$ is said to have constant symbol (or be of constant type) if its  symbols at every two points are isomorphic via a conformal symplectic transformation between the corresponding fibers of $\mathcal D$ at these points. If $p$ belongs to $M$ then $\tau$ induces an involution on the $CR$ symbol, and, in this case, we refer to the symbol at $p$ as the CR symbol with involution.
\end{definition}

In the sequel we work with  2-nondegenerate, hypersurface-type CR structures with constant symbol.  Restricting to constancy of symbol still allows for rich variation in the possible geometry, presenting appreciable local equivalence problems as demonstrated in Section \ref{Examples of the canonical parallelism construction}, and in this category we can indeed fully solve the local equivalence problems. The category is notable as it contains, of particular interest, the homogeneous structures. Our constructions can in principle be extended to study structures of non-constant type, and the extent to which this paper's results generalize to that category is described in Appendix \ref{Non-constant symbol structures}.

Definition \ref{symb_def} is formally different from the notion of CR symbol as a certain bigraded (i.e., $(\Z\times \Z)$-graded) vector space introduced in \cite[Definition 2.2]{porter2017absolute}, but it contains equivalent information. To relate these definitions in the subsequent paragraph we use a certain bigrading to denote different spaces. The actual meaning of this bigrading is explained in detail in \cite{porter2017absolute}, but it is not crucial here.

The symbol in Definition \ref{symb_def} of the 2-nondegenerate, hypersurface-type CR structure $H$ with Levi kernel $K$  on a $(2n+1)$-dimensional manifold at a point $p$ in $M$ is encoded by
\begin{enumerate}
\item a $2n$-dimensional conformal symplectic space $\mathfrak g _{-1}$ (over $\C$)  with antilinear involution $\sigma$,
\item a splitting $\mathfrak g_{-1}=\mathfrak g_{-1,-1}\oplus \mathfrak g_{-1,1}$, where $\mathfrak g_{-1, \pm 1}$ are Lagrangian with $\mathfrak g_{-1,1}=\sigma(\mathfrak g_{-1,-1})$, and
\item two $(\mathrm{rank}\, K)$-dimensional subspaces $\mathfrak g_{0, 2}$ and $\mathfrak g_{0, -2}$  belonging to $\mathrm{Sym}(\mathfrak g_{-1,-1}^*)\subset \mathfrak{csp}(\mathfrak g_{-1})$ and 
$\mathrm{Sym}(\mathfrak g_{-1,1}^*)\subset \mathfrak{csp}(\mathfrak g_{-1})$ respectively, such that one is obtained from the other via the map induced by $\sigma$. 
\end{enumerate}
Specifically, 
we can identify $\sigma$ with the antilinear involution induced by the complex conjugation $\tau$ defined on $\C M$, and make the identifications
\begin{align}\label{CR symbol description identifications}
\mathfrak{g}_{-1}\cong\mathcal{D}_{\pi(p)},\quad  \mathfrak{g}_{-1,1}\cong J^+_{\pi(p)}(p),\quad \mathfrak{g}_{-1,-1}\cong J^-_{\pi(p)}(p),\quad\mbox{ and }\quad \mathfrak{g}_{0,\pm 2}\cong \delta^\pm(p).
\end{align}
\begin{remark}
The symbol in Definition \ref{symb_def} of the 2-nondegenerate, hypersurface-type CR structure at a point in $\C M$ not in $M$ is also encoded by the objects in the 3-item list above but excluding their properties referencing an involution.
\end{remark}

Further we can consider the $(2n+1)$-dimensional Heisenberg algebra $\mathfrak g_{-}:=\mathfrak g_{-1}\oplus \mathfrak g_{-2}$ characterized as being a central extension of $\mathfrak g_{-1}$ whose center is $\mathfrak{g}_{-2}$ and such that for any representative $\omega$ of the conformal symplectic structure on $\mathfrak{g}_{-1}$ there exists a nonzero $z\in g_{-2}$ for which $[x, y]:=\omega(x, y) z$ for every $x, y\in \mathfrak g_{-1}$. 

Finally, let $\mathfrak g_{0,0}$ be the maximal  subalgebra of $\mathfrak{csp}(\mathfrak g_{-1})$ such that
\begin{align}\label{g001}
[\mathfrak g_{0,0}, \mathfrak g_{-1, \pm 1}]\subset \mathfrak g_{-1, \pm 1}
\quad\mbox{ and }\quad 
[\mathfrak g_{0,0}, \mathfrak g_{0, \pm 2}]\subset \mathfrak g_{0, \pm 2},
\end{align}
where the brackets in the first line  are just the action of elements $\mathfrak g_{0,0}\subset \mathfrak{csp}(\mathfrak{g}_{-1})$ on $\mathfrak{g}_{-1}$ and the brackets in the second line are as in $\mathfrak{csp}(\mathfrak{g}_{-1})$.

The subspace 
\begin{equation}
\label{CRsymbol_2}
\mathfrak g^0=\mathfrak{g}_-\oplus\mathfrak{g}_{0,-2}\oplus \mathfrak{g}_{0,2}\oplus \mathfrak g_{0,0}\subset \mathfrak{g}_{-}\rtimes \mathfrak{csp}(\mathfrak{g}_{-1})   
\end{equation}
together with the involution induced on it by the involution $\sigma$ on $\mathfrak{g}_{-1}$ is the symbol introduced in \cite[Definition 2.2]{porter2017absolute}. So there is a bijective correspondence between the notion of CR symbols (at points in $M$) introduced here and in \cite{porter2017absolute}. 

Here in the semidirect sum we mean the natural action of $\mathfrak{csp}(\mathfrak{g}_{-1})$ on $\mathfrak g_-$ induced from the standard action on $\mathfrak g_{-1}$. This induced action actually identifies $\mathfrak{csp}(\mathfrak{g}_{-1})$ with the algebra of derivations of $\mathfrak{g}_{-}$, so in the sequel we freely regard elements of each space $\mathfrak{g}_{0,i}$ as endomorphisms of both $\mathfrak{g}_{-1}$ and $\mathfrak{g}_{-}$, letting context dictate which interpretation is being applied.  

\begin{definition}\label{definition of regular}
The CR symbol is called \emph{regular} if  $[\mathfrak{g}_{0, -2}, \mathfrak{g}_{0,2}]\subset \mathfrak g_{0,0}$, where these brackets are as in $\mathfrak{csp}(\mathfrak{g}_{-1})$ or, equivalently, $\mathfrak g^0$ is a Lie subalgebra of 
$\mathfrak{g}_{-}\rtimes \mathfrak{csp}(\mathfrak{g}_{-1})$. 
\end{definition}
\begin{remark}
\label{regcritrem}
By construction a CR symbol is regular if and only if 
$[\mathfrak g_{0,-2}, \mathfrak g_{0, 2}]\subset \mathfrak g_{0,0}$ or equivalently 
\begin{equation}
\label{regcriteq}
[\mathfrak g_{0,2},[\mathfrak g_{0,-2}, \mathfrak g_{0, 2}]]\subset \mathfrak g_{0,2}
\quad\mbox{ and }\quad [\mathfrak g_{0,-2},[\mathfrak g_{0,-2}, \mathfrak g_{0, 2}]]\subset \mathfrak g_{0,-2}.
\end{equation}

\end{remark}

A canonical absolute parallelism for CR structures with the regular symbols was constructed in \cite[Theorem 3.2] {porter2017absolute} using bigraded Tanaka prolongations. The set of regular CR symbols, at least for the case of $\mathrm{rank}\, K=1$, is a rather small discrete subset of the set of all symbols. To describe this in more detail, we need the following observation. 
\begin{remark}\label{symbol_as_pair} A CR  symbol can be also encoded by
\begin{enumerate}
    \item a real line of Hermitian forms spanned by a form $\ell:\mathfrak{g}_{-1,1}\times \mathfrak{g}_{-1,1}\to \C$ defined by $\ell(x,y)=i\omega\left(x,\sigma(y)\right)$, for some representative $\omega$ of the conformal symplectic form on $\mathfrak{g}_{-1}$, and
    \item a vector space of $\ell$-self-adjoint antilinear operators on $\mathfrak{g}_{-1,1}$ defined by $\{x\mapsto [v,\sigma(x)]\,|\,v\in \mathfrak{g}_{0,2}\}$
\end{enumerate}
because the symbol's Heisenberg algebra structure on $\mathfrak{g}_{-}$ can be recovered from the line of Hermitian forms spanned by $\ell$, and the subspaces $\mathfrak{g}_{0,\pm 2}$ can be recovered from the vector space of $\ell$-self-adjoint antilinear operators. In particular, if $\mathrm{rank}\, K=1$, then the vector space in item (2) is generated by an $\ell$-self-adjoint antilinear operator $A$, so  the  CR symbol is encoded in the pair  $(\R\ell,\C A)$. Canonical forms for these pairs were obtained in \cite{sykes2020canonical}.
\end{remark} 

This allows us to state the following proposition.
\begin{proposition}[{\cite[section 4] {porter2017absolute}, see also Remark \ref{regmatrixrem} for generalization  to an arbitrary $\mathrm{rank}\, K$}]\label{regularity condition for rank-1 K}
If $\mathrm{rank}\, K=1$, $\R\ell$ is the real line of Hermitian forms, and $\C A$ is the complex line of $\ell$-selfadjoint antilinear operators spanned by a representative $A:\mathfrak{g}_{-1,1}\to\mathfrak{g}_{-1,1}$ such that the CR structure's symbol is encoded in $(\R\ell,\C A)$ as described above, then the CR symbol is regular if and only if 
\begin{equation}
\label{cube}
A^3\in\C A.
\end{equation}

\end{proposition}

In particular, if $\mathrm{rank}\, K=1$ and $\ad \overline K$ is generated by a rank 1 operator, then the corresponding antilinear operator $A$ is of rank 1 and so it satisfies \eqref{cube}. Therefore the CR symbol in this case is always regular. This together with  Corollary \ref{rank1cor} immediately implies the following corollary.
\begin{corollary}
If the symbol of the CR structure $H$ at a point $p$ is not regular and $\mathrm{rank}\, K=1$ then the CR structure $H$ is recoverable  in some neighborhood of $p$, that is, $H$ is uniquely determined by the dynamical Legendrian contact structure $(\mathcal{D},J^+,\tau)$ associated with the germ of $H$ at $p$.
\end{corollary}

Moreover, based on the classification of regular symbols from \cite[section 4]{porter2017absolute}, for fixed $\mathrm{rank}\,H >1$ and given signature of $\ell$,  among all regular CR symbols,  there are exactly two symbols for which the operator $A$ has rank $1$ and is consequently non-recoverable if the reduced Levi form is sign-indefinite, and exactly one if the reduced Levi form is sign-definite. These two symbols in the sign-definite case are distinguished by whether the corresponding antilinear operator $A$ is nilpotent or not. The former is not possible in the sign-definite case. 

Note that the recoverability criteria of  Proposition \ref{prolK}  can be reformulated in  terms of the symbol of the CR structure as follows.

\begin{proposition}
\label{prolK1} A $2$-nondegenerate, hypersurface type CR structure with symbol $\mathfrak g^0 (p)$ at a point $p$ is recoverable in a neighborhood of $p$  if and only if the first prolongation $(\mathfrak{g}_{0, 2})_{(1)}$ of the space $\mathfrak g_{0,2}$ considered as the subspace of $\mathfrak{csp}(\mathfrak g_{-1})$ vanishes.
\end{proposition}

The CR symbols satisfying the condition of Proposition \eqref{prolK1} will be called recoverable.

\section{Modified CR symbols and the construction of absolute parallelisms}\label{bundles on curves in Lagrangian Grassmannians}

Assume that $H$ is a 2-nondegenerate hypersurface-type CR structure with constant symbol $\mathfrak g^0$ (having the involution $\sigma$) as in \eqref{CRsymbol_2}. Here $\mathfrak{g}^0$, which we will call the ``model space,'' is a fixed representative in the equivalence class of CR symbols. On the other hand, at every point $p\in \C M$ we have the representative $\mathfrak g^0(p)$ of the equivalence class of CR symbols, obtained by the identifications \eqref{CR symbol description identifications} and the corresponding bigraded components of it will be denoted by $\mathfrak{g}_{i, j}(p)$, including $\mathfrak{g}_{0,0}(p)$ defined accordingly as in \eqref{g001}.   

Note that $\mathfrak{g}_{-}$ and each $\mathfrak{g}_{i,j}$ is determined by the symbol of the CR structure at a point $p$, so we will write $\mathfrak{g}_{-}(p)$ or $\mathfrak{g}_{i,j}(p)$ instead of $\mathfrak{g}_{-}$ or $\mathfrak{g}_{i,j}$ whenever we need to emphasize the dependence of $\mathfrak{g}_{-}$ or $\mathfrak{g}_{i,j}$ on $p$. 

We now define two bundles, one over $\C M$ and another one over $\mathcal N$ with the same total space $P^0$. For the first bundle $\mathrm{pr}: P^0\to \C M$, its fiber $\mathrm {pr}^{-1} (p)$ over $p\in\C M$ is comprised of all \emph{adapted frames}, or bigraded Lie algebra isomorphisms, that is, 
\begin{align}
\label{fiberP0}
\mathrm{pr}^{-1}(p)=\left\{
\varphi:\mathfrak g_-\to \mathfrak g_-(p) \left|\ 
\parbox{7.9cm}{
$\varphi(\mathfrak g_{i,j})=\mathfrak g_{i,j}(p)\quad\forall\, (i,j)\in \{(-1, \pm 1),(-2, 0)\}$,\\
$\varphi^{-1}\circ \mathfrak g_{0,\pm2}(p)\circ\varphi=\mathfrak g_{0,\pm2}$, and \\
$\varphi([ y_1, y_2])=[\varphi( y_1),\varphi( y_2)] \quad\forall\, y_1, y_2\in\mathfrak g_-$
}
\right.\right\}.
\end{align}
Here $\mathfrak g_{-2, 0}:=\mathfrak g_{-2}$ and  $\mathfrak g_{-2, 0}(p):=\mathfrak g_{-2}(p)$. The second bundle is $\pi\circ \mathrm{pr}:P^0\rightarrow \mathcal N$. Furthermore, we define a bundle $\mathrm{pr}:\Re P^0\to M$ by
\begin{align}\label{involution invariant frames}
\Re P^0:=\{\varphi\in P^{0}\,|\,\mathrm{pr}(\varphi)\in M\mbox{ and } \tau_*\circ \varphi=\varphi\circ \sigma\},
\end{align}
where $\sigma$ denotes the fixed involution on the model space $\mathfrak{g}^0$ and $\tau_*$ denotes the involution on $J^+\times J^-$ induced by the map $\tau$ introduced just before \eqref{antihol}. Note that the set-inclusion conditions defining the set in \eqref{involution invariant frames} have some redundancy, but we have stated them as such for clarity.

\begin{remark}
Using the standard identification of the automorphism group of the conformal symplectic structure on $\mathfrak{g}_{-1}$ and the automorphism group of the graded Heisenberg algebra $\mathfrak{g}_{-}$, one can alternatively define $P^0$ to be an equivalent object formally described as a bundle of maps from $\mathfrak{g}_{-1}$ to $\mathfrak{g}_{-1}(p)$ in the obvious way. When calculating $P^0$ for specific examples, this alternate description is more succinct, and we use it in section \ref{examples}.
\end{remark}

For any $\psi\in P^0$ with $\gamma=\pi\circ\mathrm{pr}(\psi)$, the tangent space of the fiber $(P^0)_{\gamma}=(\pi\circ \mathrm{pr})^{-1}(\gamma)$ of the second bundle at $\psi$ can be identified with a subspace of $\mathfrak{csp}(\mathfrak g_-)$ by the map $\theta_0:T_\psi(P^0)_{\gamma}\to \mathfrak{csp}(\mathfrak g_-)$ given by
\begin{align}\label{degree zero soldering form section 4}
\theta_0\big(\psi^\prime(0)\big):=\big(\pi_*\circ\psi\big)^{-1}\circ \left.\frac{d}{dt}\right|_{t=0}\left(\pi_*\circ\psi(t)\right)
\end{align}
where $\psi:(-\epsilon,\epsilon)\to (P^0)_{\gamma}$  denotes an arbitrary curve in $(P^0)_{\gamma}$ with $\psi(0)=\psi$. 
\begin{remark}
For $\psi\in (P^0)_{\gamma}$ the map $\widehat{\psi}\mapsto \left(\pi_*\circ \psi\right)^{-1}\circ (\pi_*\circ \widehat{\psi})$ embeds the fiber $(P^0)_{\gamma}$ into the Lie group $CSp(\mathfrak{g}_{-1})$, and the map $\theta_0$ in \eqref{degree zero soldering form section 4} can be interpreted as the Maurer--Cartan form on $CSp(\mathfrak{g}_{-1})$ restricted to the tangent spaces of the embedded submanifold $\left(\pi_*\circ \psi\right)^{-1} (P^0)_{\gamma}$, as summarized in the following diagram.

\begin{center}
\begin{tikzcd}[every matrix/.append style={name=m}]
\left(P^{0}\right)_\gamma \arrow[ddr, "\Pi"] \arrow[dd,swap,"\mathrm{pr}" ]\arrow[rr,"\widehat{\psi}\mapsto(\pi_*\circ{\psi})^{-1}\circ(\pi_* \circ\widehat{\psi})" ] & & CSp(\mathfrak{g_{-1}}) \arrow[dd,swap,"\parbox{2cm}{\tiny \flushright Maurer--Cartan\\ form}" ]\arrow[dd,"\bgroup \def\arraystretch{1}\begin{matrix}T_{(\pi_*\circ{\psi})^{-1}\circ(\pi_* \circ\widehat{\psi})} \left(\left(\pi_*\circ \psi\right)^{-1} (P^0)_{\gamma}\right)\\ \downmapsto\\ \mathfrak{g}_0^{\mathrm{mod}}(\widehat{\psi}) \end{matrix}\egroup" ]\\
&     &   \\
\pi^{-1}(\gamma)\subset \mathbb{C}M\arrow[r, "\pi"]  &\gamma= \Pi(\psi)\in \mathcal{N} & \mathfrak{csp}(\mathfrak{g}_{-1})
\end{tikzcd}
 \end{center}
\end{remark}

Let 
\begin{equation}
\label{modsymeq}
\mathfrak g_0^{\mathrm{mod}}(\psi):=\theta_0(T_\psi(P^0)_{\gamma}).
\end{equation}
Here the superscript $\mathrm{mod}$ stands for modified in order to distinguish it from  the space $\mathfrak g_0:=\mathfrak g_{0, -2}\oplus \mathfrak g_{0,0}\oplus \mathfrak{g}_{0, 2}$, defined in \cite{porter2017absolute}, which is in general another subspace of $\mathfrak {csp}(\mathfrak g_-)$.

\begin{definition}
The space $\mathfrak g^{0, \mathrm{mod}}(\psi):=\mathfrak g_-\oplus \mathfrak g_0^{\mathrm{mod}}(\psi) $ is called the modified CR symbol of the CR structure at the point $\psi\in P^0$.  
\end{definition}

The bigrading $\mathfrak{g}_{-1}=\mathfrak{g}_{-1,-1}\oplus \mathfrak{g}_{-1,1}$ of $\mathfrak{g}_{-1}$ confers a bigrading on $\mathfrak{csp}(\mathfrak{g}_{-1})$ with weighted components given by
\[
\big(\mathfrak{csp}(\mathfrak{g}_{-1})\big)_{0,i}=\big\{\varphi\in\mathfrak{csp}(\mathfrak{g}_{-1})\,|\, \varphi(\mathfrak{g}_{-1,j})\subset \mathfrak{g}_{-1,i+j}\, \forall i\in\{-1,1\} \big\},
\]
yielding the decomposition 
\begin{align}\label{csp decomposition}
\mathfrak{csp}(\mathfrak{g}_{-1})=\big(\mathfrak{csp}(\mathfrak{g}_{-1})\big)_{0,-2}\oplus\big(\mathfrak{csp}(\mathfrak{g}_{-1})\big)_{0,0}\oplus \big(\mathfrak{csp}(\mathfrak{g}_{-1})\big)_{0,2}.
\end{align}
For an element $\varphi \in \mathfrak{csp}(\mathfrak{g}_{-1})$ (or subset $S \subset \mathfrak{csp}(\mathfrak{g}_{-1})$), we write $\varphi_{0,i}$ (or $S_{0,i}$) to denote the natural projection of $\varphi$ (or $S$) into $\big(\mathfrak{csp}(\mathfrak{g}_{-1})\big)_{0,i}$ with respect to the decomposition in \eqref{csp decomposition}.

By construction 
\begin{align}\label{theta on 0,2 horizontals}
\left(\theta_0\left(\mathrm{pr}_*^{-1}(K_p)\right) \right)_{0,2}=\mathfrak{g}_{0,2}\quad\quad\forall\, p\in M,
\end{align}
and
\begin{align}\label{theta on 0,-2 horizontals}
\left(\theta_0\left(\mathrm{pr}_*^{-1}\left(\overline{K}_p\right)\right) \right)_{0,-2}=\mathfrak{g}_{0,-2}\quad\quad\forall\, p\in M.
\end{align}
The bundle $\mathrm{pr}:P^0\to \C M$ is a principal bundle whose structure group, which we label $G_{0,0}$, has the Lie algebra $\mathfrak{g}_{0,0}$, and $\theta_0$ is the principal connection on this bundle, which means, in particular, that 
\begin{align}\label{theta on verticals}
\theta_0\left(\mathrm{pr}_*^{-1}(0)\right)=\mathfrak{g}_{0,0}.
\end{align}
Since
\[
T_\psi(P^0)_{\gamma}\cong \mathrm{pr}_*^{-1}(0)\oplus K_{\mathrm{pr}(\psi)} \oplus \overline{K}_{\mathrm{pr}(\psi)}
\]
and the subspace $\mathrm{pr}_*^{-1}(0)\subset T_\psi(P^0)_{\gamma}$ belongs to the kernel of the projections $\varphi\mapsto \varphi_{0,\pm 2}$, it follows from \eqref{theta on 0,2 horizontals}, \eqref{theta on 0,-2 horizontals}, and \eqref{theta on verticals} that 
\begin{align}
\dim \theta_0\left( T_\psi(P^0)_{\gamma}\right)&=\dim \theta_0\left(\mathrm{pr}_*^{-1}(0)\right) +\dim \left(\theta_0\left(\mathrm{pr}_*^{-1}(K_{\mathrm{pr}(\psi)})\right)\right)_{0,2} +\dim \left(\theta_0\left(\mathrm{pr}_*^{-1}\left(\overline{K}_{\mathrm{pr}(\psi)}\right)\right)\right)_{0,-2}
\\
&= \dim\mathfrak{g}_{0,0}+\dim K_{\mathrm{pr}(\psi)}+\dim \overline{K}_{\mathrm{pr}(\psi)}
=\dim T_\psi(P^0)_{\gamma},
\end{align}
and hence $\theta_0$ is injective on each tangent space $T_\psi(P^0)_{\gamma}$. Accordingly,
\[
\dim \mathfrak{g}^0\big(\mathrm{pr}(\psi)\big)=\dim \mathfrak{g}^{0,\mathrm{mod}}(\psi)\quad\quad\forall\, \psi\in P^0.
\]

Now we are going to define the universal Tanaka prolongation of the modified symbol $\mathfrak g^{0, \mathrm{mod}}(\psi)$ by analogy with the standard Tanaka theory \cite{tanaka1, zelenko2009tanaka}. A nonstandard feature here is that the modified CR symbol is not necessarily a Lie algebra. In the standard Tanaka theory the universal Tanaka prolongation of a graded Lie algebra $\mathfrak{g}^0=\mathfrak{g}_{-2}\oplus \mathfrak{g}_{-1}\oplus \mathfrak{g}_0$ can be shortly defined as the largest $\mathbb Z$-graded subalgebra containing $\mathfrak{g}^0$ as its nonnegative part satisfying the additional condition that nonnegatively graded elements have nontrivial brackets with $\mathfrak{g}_{-1}$ or, equivalently, one can define each positively graded component of this resulting algebra recursively.  Contrastingly, since $\mathfrak g^{0, \mathrm {mod}}(\psi)$ is not necessarily a Lie algebra, the recursive definition is the only available one to define this prolongation of $\mathfrak g^{0, \mathrm {mod}}(\psi)$  because the resulting universal prolongation is also not a Lie algebra. In more detail, we recursively define the subspaces $\mathfrak g_k^{\mathrm{mod}}(\psi)$
by 
\begin{equation}
\label{mgk 1}
\mathfrak g_1^{\mathrm{mod}} (\psi):=\left\{f\in {\rm Hom}(\mathfrak g_{-2},\mathfrak g_{-1})+\mathfrak {\rm Hom}\big(\mathfrak g_{-1},\mathfrak g_{0} (\psi)\big)\,\left|\,\parbox{5.8cm}{$f ([v_1,v_2])=[f (v_1),v_2]+[v_1, f( v_2)]$\\$\forall\, v_1, v_2 \in \mathfrak g_-$}\right.\right\}
\end{equation}
and
\begin{equation}
\label{mgk higher}
\mathfrak g_k^{\mathrm{mod}} (\psi):=\left\{f\in\bigoplus_{i<0}{\rm Hom}(\mathfrak  g_i,\mathfrak g_{i+k}(\psi))\left|\,\parbox{6cm}{$f ([v_1,v_2])=[f (v_1),v_2]+[v_1, f( v_2)]$\\$\forall\, v_1, v_2 \in \mathfrak g_-$}\right.\right\} \quad\quad\forall\, k>1,
\end{equation}
and the universal Tanaka prolongation of $\mathfrak g^{0, \mathrm {mod}}(\psi)$
is the space
\begin{equation}
\label{universal prolongation definition}
\mathfrak u(\psi):=\mathfrak g_-\oplus \bigoplus_{k\geq 0} \mathfrak g_k^{\mathrm{mod}}(\psi).
\end{equation}
Note that the antilinear involution on $\mathfrak g_-$ naturally induces an involution on $\mathfrak {csp}(\mathfrak g_-)$, which in turn induces a natural involution on each modified symbol $\mathfrak g^{0, \mathrm{mod}}(\psi)$ whenever $\psi\in \Re P^0$.

To state our main result we introduce the following definitions.
\begin{definition}
\label{abstractmod}
Fix a CR symbol $\mathfrak g^0=\mathfrak{g}_{-}\oplus \mathfrak{g}_{0,-2}\oplus \mathfrak{g}_{0,0}\oplus \mathfrak{g}_{0,2}$ of the form in \eqref{CRsymbol_2}. A subspace $\mathfrak g^{0, \mathrm{mod}}= \mathfrak g_-\oplus \mathfrak g_0^{ \mathrm {mod}}$ of $\mathfrak g_- \rtimes\mathfrak {csp}(\mathfrak g_{-1}) $ is called an abstract modified CR symbol (with involution) of type $\mathfrak g^0$ if the following properties hold
\begin{enumerate}
    \item $\dim \mathfrak{g}^{\mathrm{mod}}_{0}=\dim \mathfrak{g}_{0}$;
    \item $\mathfrak g_0^{\mathrm{mod}}\cap \left(\mathfrak{csp}(\mathfrak{g}_{-1})\right)_{0,0}=\mathfrak g_{0,0} $;
    \item $\left(\mathfrak g_0^{ \mathrm {mod}}\right)_{0,\pm2}=\mathfrak{g}_{0,\pm 2}$, where $\left(\mathfrak g_0^{ \mathrm {mod}}\right)_{0,\pm2}$ stands for the image of the projection to the subspace $\mathfrak {csp}(\mathfrak g_{-1})$ with respect to the splitting \eqref{csp decomposition} of $\mathfrak {csp}(\mathfrak g_{-1})$;
    \item  The subspace $\mathfrak g_0^{\mathrm{mod}}$ is invariant with respect to the involution on $\mathfrak g_{-1}\oplus\mathfrak {csp}(\mathfrak g_{-1})$.
\end{enumerate}
\end{definition}

\begin{definition}\label{regular points in Pzero}
A point $\psi_0\in P^0$ is \emph{regular} (with respect to the Tanaka prolongation) if the maps $\psi\mapsto\dim \mathfrak{g}_k^{\mathrm{mod}}$ are constant in a neighborhood of $\psi_0$ for all $k\geq 0$.
\end{definition}
If we assume that there exists an integer $l\geq 0$ such that the set of points $\psi \in P^0$ with $\mathfrak g_l^{\mathrm{mod}}(\psi)=0$ is generic, then the regularity property in Definition \ref{regular points in Pzero} is also generic.

\begin{remark}
For a 2-nondegenerate CR structure whose Levi form has a rank one kernel such that $\mathfrak{g}_{0,2}$ is generated by an element of rank greater than one, the regularity property in Definition \ref{regular points in Pzero} is generic. 
\end{remark}

Now we formulate the main result of the paper.

\begin{theorem}
\label{maintheorm}
Fix  a CR symbol $\mathfrak g^0$ and a corresponding modified CR symbol $\mathfrak g^{0, \mathrm{mod}}$ so that  its universal Tanaka prolongation $\mathfrak u(\mathfrak g^{0, \mathrm{mod}})$ is finite dimensional.
\begin{enumerate}
\item 
Given a $2$-nondegenerate, hypersurface-type CR structure with symbol $\mathfrak {g}^0$ such that there exists a regular point $\psi_0$ (w.r.t. the Tanaka prolongation) in the bundle  $\Re P^0$ of this structure with $\mathfrak g^{0,\mathrm{mod}}(\psi_0)=\mathfrak g^{0, \mathrm{mod}}$, there exists a bundle over a neighborhood of $\psi_0$ in $\Re P^0$ of dimension equal to $\dim_\C \mathfrak u(\mathfrak g^{0, \mathrm{mod}})$ that admits a canonical absolute parallelism. 

\item
The real dimension of the algebra of infinitesimal symmetries of a $2$-nondegenerate, hypersurface-type CR structure of the previous item is not greater than $\dim_\C \mathfrak u(\mathfrak g^{0, \mathrm{mod}})$.
\end{enumerate}
\end{theorem}

The proof of this theorem is obtained via a modification of the approach of \cite{zelenko2009tanaka} to the original Tanaka prolongation procedure for the construction of absolute parallelisms  developed in \cite{tanaka1}.  A sketch of the proof with emphasis of the main modifications is given in Section \ref{Absolute parallelisms}.

Let us clarify the meaning of the term canonical absolute parallelism in the formulation of item (1) of Theorem \ref{maintheorm}. Similar to the standard Tanaka theory in \cite{tanaka1}, wherein a sequence of affine fiber bundles
 \begin{equation}
  \label{bundles original construction}
  \C M\leftarrow P^0\leftarrow P^1\leftarrow P^2\leftarrow\cdots
 \end{equation}
are constructed, we will recursively construct a sequence of bundles $\{P^i\to \mathcal{O}^{i-1}\}_{1\leq i\leq l+\mu}$ such that each base space $\mathcal{O}^{i}$ is a neighborhood in $P^i$ fitting into the sequence of fiber bundles
\begin{equation}
  \label{bundles}
  \C M\leftarrow \mathcal{O}^0\leftarrow \mathcal{O}^1\leftarrow \mathcal{O}^2\leftarrow\cdots.
\end{equation}
Moreover, each of the spaces in \eqref{bundles} have an involution defined on them and by restricting the maps in \eqref{bundles} to their respective fixed point sets one obtains another sequence of fiber bundles
\begin{equation}
  M\leftarrow \Re\mathcal{O}^0\leftarrow \Re\mathcal{O}^1\leftarrow \Re\mathcal{O}^2\leftarrow\cdots.
\end{equation}

The fibers of each $P^i$ are affine spaces with modeling vector space of dimension equal to $\dim \mathfrak g_i^{\mathrm{mod}}(\psi_0)$ from \eqref{mgk higher}.  If the positively graded part of the universal Tanaka prolongation $\mathfrak u(\psi_0)$ of $\mathfrak g^{0, \mathrm{mod}}(\psi_0)$ consists of $l$ nonzero graded components, then all $\mathcal{O}^i$ with $i\geq l$ are identified with each other by the bundle projections, which are diffeomorphisms in those cases. The bundle $\mathcal{O}^{l+2}$ is an $e$-structure over $\mathcal{O}^{l+1}$, which determines a canonical absolute parallelism on $\mathcal{O}^l$ via aforementioned identification between $\mathcal{O}^{l+1}$ and $\mathcal{O}^{l}$. It is important to note that for any $0<i\leq l$ the recursive construction of the bundle $\mathcal{O}^{i+1}$ over $\mathcal{O}^i$ depends on a choice of normalization conditions, as in the standard Tanaka prolongation theory, and also on a choice of identifying spaces, which is a new feature required for the prolongation procedure to work in the presence of nonconstant (modified) symbols. Algebraically, ``normalization condition" refers to a choice of vector space complement to the image of a certain Lie algebra cohomology differential along with identifying spaces that are the choices of complementary subspaces to $\mathfrak g_k^{\mathrm{mod}}$ in the $i$th algebraic Tanaka prolongation of the algebra $\mathfrak g_-\rtimes \mathfrak{csp}(\mathfrak g_{-1})$.  Therefore, the word ``canonical" in item (1) of Theorem \ref{maintheorm} means that for any CR structure of the type referred to in item (1), the same fixed normalization conditions are applied in each step of the construction of the sequence \eqref{bundles}. This also guaranties the preservation the constructed bundles under the action of the group of symmetries of the underlying structure which essentially implies item (2) of the theorem.

Theorem \ref{maintheorm} and \cite[Theorem 3.2]{porter2017absolute} is used in our other recent paper \cite{Sykes2022} to obtain sharp upper bounds on the algebra of infinitesimal symmetries of homogeneous 2-nondegenerate, hypersurface-type CR structures whose Levi kernel has rank 1.

\section{CR structures with constant modified symbols}\label{CR structures with constant modified symbols}

In this section we study the CR structures with constant modified CR symbol $\mathfrak{g}^{0,\mathrm{mod}}$, meaning that $\mathfrak{g}^{0,\mathrm{mod}}(\psi)=\mathfrak{g}^{0,\mathrm{mod}}$ for every $\psi$ in $P^0$.
We prove that in this case the modified CR symbol must be a Lie algebra (Proposition \ref{constmodprop}) and that the CR structure has a regular symbol at every point in the sense of Definition \ref{definition of regular} (Theorem \ref{constant mod symbols implies regular}).

\begin{proposition}
\label{constmodprop}
If a 2-nondegenerate CR structure of hypersurface type  has a constant modified symbol $\mathfrak{g}^{0,\mathrm{mod}}$ then the degree zero component $\mathfrak{g}^{\mathrm{mod}}_0$ is a subalgebra of $\mathfrak{csp}(\mathfrak{g}_{-})$.
\end{proposition}
\begin{proof}
For a point $p\in M$, let $\gamma=\pi(p)$ and let $\varphi_0\in (P^0)_\gamma$ be fixed. 
The differential of the projection $\pi:\C M\to \mathcal{N}$ induces an isomorphism
\[
\pi_*:\mathfrak{g}_{-1}(p)\oplus \mathfrak{g}_{-2}(p)\to \mathcal{D}_{\gamma}\oplus T_{\gamma}\mathcal{N}/\mathcal{D}_{\gamma}.
\]
The fiber $(P^0)_\gamma$ of $P^0$ can be identified with the submanifold
\[
G:=\left\{(\pi_*\circ \varphi_0)^{-1}\circ\pi_*\circ \psi \,\left|\, \psi\in(P^0)_\lambda\right.\right\}
\]
in the Lie group $CSp(\mathfrak{g}_{-1})$.

If we apply the Maurer--Cartan form $\Omega:T \,CSp(\mathfrak{g}_{-1})\to \mathfrak{csp}(\mathfrak{g}_{-1})$ of $CSp(\mathfrak{g}_{-1})$ to a tangent space of $G$ then its image will be a subspace of $\mathfrak{csp}(\mathfrak{g}_{-1})$, and moreover we can actually show that 
\begin{align}\label{constant Maurer form}
\Omega(T_g G)=\mathfrak{g}_0^{\mathrm{mod}} \quad\quad\forall\, g\in G.
\end{align}

Indeed, let us compute $\Omega(T_g G)$. For a point $g=(\pi_*\circ \varphi_0)^{-1}\circ\pi_*\circ \psi\in G$ and a vector $v\in T_g G$, let $\psi:(-\epsilon,\epsilon)\to (P^0)_\lambda$ be a curve with $\psi(0)=\psi$ such that 
\[
v= (\pi_*\circ \varphi_0)^{-1}\circ\pi_*\left( \psi^\prime(0)\right).
\]
The value of the Maurer--Cartan form $\Omega$ at $v$ is given by
\begin{align}
\Omega(v)&=\left((\pi_*\circ \varphi_0)^{-1}\circ\pi_*\circ \psi\right)^{-1}\circ(\pi_*\circ \varphi_0)^{-1}\circ\pi_*\left( \psi^\prime(0)\right)=\psi^{-1}\psi^{\prime}(0)=\theta_{0}\big(\psi^\prime(0)\big),
\end{align}
so indeed $\Omega(v)$ belongs to $\mathfrak{g}^{\mathrm{mod}}_0$. Equation \eqref{constant Maurer form} now follows because $\dim G=\dim P^0=\dim \mathfrak{g}^{\mathrm{mod}}_0$.

Let us now show that $\mathfrak{g}_{0}^{\mathrm{mod}}$ is a subalgebra of $\mathfrak{csp}(\mathfrak{g}_{-})$. Fix two vectors $v_1,v_2\in \mathfrak{g}_{0}^{\mathrm{mod}}$, and set $V_i=\Omega^{-1}(v_i)$; that is, $V_i$ is the left-invariant vector field on $CSp(\mathfrak{g}_{-1})$ whose value at the identity is $v_i$. Since $V_1$ and $V_2$ are both tangent to $G$ at each point in $G$, the left-invariant vector field $[V_1,V_2]$ is also tangent to $G$ at every point in $G$. In particular, letting $e$ denote the identity element in $CSp(\mathfrak{g}_-)$, we have 
\[
[v_1,v_2]=[V_1,V_2]_e\in T_eG=\mathfrak{g}_{0}^{\mathrm{mod}},
\]
which shows that $\mathfrak{g}_{0}^{\mathrm{mod}}$ is closed under Lie brackets.
\end{proof}

For the remainder of Section \ref{CR structures with constant modified symbols}  our goal is to prove the following theorem.

\begin{theorem}\label{constant mod symbols implies regular}
If a 2-nondegenerate CR structure of hypersurface type   has a constant modified symbol $\mathfrak{g}^{0,\mathrm{mod}}$
then the CR structure has a regular CR symbol in the sense of Definition 
\ref{definition of regular} and the modified symbol equals the CR symbol as described in \eqref{CRsymbol_2}.
\end{theorem}

Before proving this theorem we introduce a matrix representation of $\mathfrak{g}_0^{\mathrm{mod}}$  and describe the conditions for $\mathfrak{g}_0^{\mathrm{mod}}$ to be a Lie subalgebra  of $\mathfrak {csp}(\mathfrak g_{-1})$ in terms of this matrix representation. These conditions will be also used for proving the results on non-existence of certain homogeneous CR structures in Section \ref{reduction_and_nongenericity_section}.

Let $\ell$ be the Hermitian form  on $\mathfrak{g}_{-1,1}$ as in item (1) of  Remark  \ref{symbol_as_pair}.
Still setting $r=\mathrm{rank} K$, set
$m=\mathrm{rank}{H/K}=n-r$,
and let $\{e_i\}_{i=1}^{m}$ be a basis of $\mathfrak{g}_{-1,1}$. Let $\left\{\overline{e_i}\right\}_{i=1}^{m}$ be the vectors obtained via the antilinear involution  on $\mathfrak{g}_{-1}$. Identify $\mathfrak{g}_{-1,1}$ with $\C^{m}$ by identifying  $(e_1,\ldots, e_m)$ with the standard basis of $\C^m $, and let $H_\ell$ be the Hermitian matrix representing $\ell$  with respect to this identification, that is, 
$\ell(e_i,e_j)=e_j^*H_{\ell} e_i$, 
where $(\cdot)^*$ denotes taking the conjugate transpose. We define the basis $(a_1,\ldots, a_{2m}) $ of $\mathfrak{g}_{-1}$  by the rule
\begin{align}\label{canon basis}
a_i=
\begin{cases}
e_i &:\, i\in\{1,\ldots, m\}\\
\overline{e_{i-m+1} }&:\, i\in \{m+1,\ldots,2m\}.
\end{cases}
\end{align}
Let $\omega$ be the symplectic form on $\mathfrak{g}_{-1}$ represented by the matrix
\begin{align}\label{sympl form}
J_\ell:=i\left(
\begin{array}{cc}
0& H_\ell\\
 -H_\ell^T& 0
\end{array}
\right);
\end{align}
that is, identifying $\mathfrak{g}_{-1}$ with $\C^{2m}$  by  identifying  $(a_1,\ldots,a_{2m})$ with the standard basis of $\C^{2m}$, we have
$\omega( a_i, a_j)=  a_j^TJ_\ell a_i$.

This symplectic form $\omega$ represents the conformal symplectic form previously defined on $\mathfrak{g}_{-1}$. By representing operators with respect to the same basis (i.e., as given in  \eqref{canon basis}), we identify the conformal symplectic Lie algebra with a matrix Lie algebra given by
\begin{align}\label{SPinclusion1}
\mathfrak{csp}(\mathfrak{g}_{-1})=
\left\{ 
\left(
\begin{array}{cc}
X_{1,1}& X_{1,2}\\
X_{2,1}& X_{2,2}
\end{array}
\right)+c\mathrm{Id}\,\left|\,
\parbox{7cm}{\flushleft $X_{2,2}=-H_\ell^{-1}X_{1,1}^TH_\ell,\,X_{2,1}=H_\ell^{-1}X_{2,1}^T\overline{H_\ell}$,  $X_{1,2}=\overline{H_\ell}^{-1}X_{1,2}^TH_\ell$, and $c\in\C$}\right.\right\}. 
\end{align}

Let $C_{1},\ldots, C_r$ be matrices such that the spaces $\mathfrak{g}_{0,2}$ and $\mathfrak{g}_{0,-2}$ are spanned by matrices of the form
 \begin{equation}
 \label{matrixg0pm2}
 \left(
\begin{array}{cc}
      0&C_i  \\
      0&0
 \end{array}
 \right)
 \quad\mbox{ and }\quad
 \left(
 \begin{array}{cc}
      0&0  \\
      \overline{C}_i&0
 \end{array}
 \right)\quad\quad\forall\, i\in\{1,\ldots,r\},
 \end{equation}
respectively.
In what follows, we consider the Lie algebras of square matrices $\alpha$ satisfying
\begin{align}\label{firstalgebra}
\alpha C_iH_\ell ^{-1} +   C_iH_\ell ^{-1}\alpha^T \in \mathrm{span}\{C_jH_\ell ^{-1}\}_{j=1}^{r} \quad\forall\,i\in\{1,\ldots,\mathrm r\}
\end{align}
and respectively
\begin{align}\label{secondalgebra}
  \alpha^TH_\ell \overline{C_i}+H_\ell \overline{C_i} \alpha\in \mathrm{span}\{H_\ell \overline{C_j}\}_{j=1}^{r}\quad\forall\,i\in\{1,\ldots,r\},
\end{align}
and we define the algebra $\mathscr{A}$ to be their intersection, that is,
\begin{align}\label{intersection algebra}
\mathscr{A}:=\left\{\alpha\,\left| \,\mbox{$\alpha$ satisfies \eqref{firstalgebra} and \eqref{secondalgebra}}\right.\right\}.
\end{align}

\begin{remark}
\label{regmatrixrem}
Using \eqref{regcriteq}, it can be shown by direct computations that the CR symbol $\mathfrak g^0$ is regular if and only if 
\begin{equation}
\label{regmatrix}
C_i\overline C_j C_k+C_k \overline C_j C_i\in \mathrm{span}_{\C}\,\{C_s\}_{s=1}^r, \quad
\forall i,j,k\in \{1,\ldots, r\},
\end{equation}
where the matrices are as in \eqref{matrixg0pm2}.
\end{remark}

The four properties in Definition \ref{abstractmod} imply that under the identification in \eqref{SPinclusion1}, the space $\mathfrak{g}_0^{\mathrm{mod}}$ has a decomposition $\mathfrak{g}_0^{\mathrm{mod}}=\mathfrak{X}_{0,2}\oplus \mathfrak{g}_{0,0}\oplus \mathfrak{X}_{0,-2}$ such that, for $i\in\{1,\ldots, r\}$ there exist $m\times m$ matrices $\Omega_i$ for which $\mathfrak{X}_{0,2}$ and $\mathfrak{X}_{0,-2}$ are spanned by the matrices
\begin{equation}\label{Xzero2}
\left(
\begin{array}{cc}
 \Omega_{i}& C_{i}\\
 0& -H_\ell^{-1}\Omega_{i}^TH_\ell
\end{array}
\right)
\quad\mbox{ and }\quad
\left(
\begin{array}{cc}
-\overline{H_\ell}^{-1}\Omega_{i}^* \overline{H_\ell}&  0\\
\overline{C_{i}}&  {\overline{\Omega_{i}}}
\end{array}
\right)
\quad\quad\forall\, i\in\{1,\ldots, r\}
\end{equation}
respectively, and, moreover, $\mathfrak{g}_{0,0}$ consists of block diagonal matrices in terms of the block decomposition given in \eqref{SPinclusion1}.
By Proposition \ref{constmodprop}, $\mathfrak{g}_0^{\mathrm{mod}}$ is a Lie subalgebra of $\mathfrak{csp}(\mathfrak{g}_{-1})$, and hence $\mathfrak{X}_0$ is a matrix Lie algebra. In particular,
\begin{align}\label{grading}
[\mathfrak{g}_{0,0},\mathfrak{X}_{0,\pm 2}]\subset  \mathfrak{X}_{0,\pm2}\oplus\mathfrak{g}_{0,0},
\end{align}
and
\begin{align}\label{stronger grading condition}
 [\mathfrak{X}_{0,-2},\mathfrak{X}_{0,2}]\subset \mathfrak{X}_{0,-2}\oplus\mathfrak{X}_{0,2}\oplus\mathfrak{g}_{0,0}.
\end{align}

The following proposition is obtained by straightforward calculation using the identification in \eqref{SPinclusion1} and applying the commutator relations in \eqref{grading} and \eqref{stronger grading condition}.
\begin{proposition}\label{system of conditions for homogeneity lemma}
The modified CR symbol  $\mathfrak g^{0, \mathrm{mod}}$ is a Lie subalgebra of $\mathfrak g_-\rtimes \mathfrak{csp}(\mathfrak{g}_{-1})$ if and only if
there exist coefficients $\eta_{\alpha,i}^s\in\C $ and $\mu_{i,j}^s\in\C$ indexed by $\alpha\in\mathscr{A}$ and $i,j,s\in\{1,\ldots,\mathrm{rank}\,K\}$ such that the  system of relations
\begin{align}\label{system}
\left.\mbox{
\begin{minipage}{.7\textwidth} 
(i)\quad\quad$\displaystyle \,\,\,\, \alpha C_iH_\ell ^{-1} +   C_iH_\ell ^{-1}\alpha^T  = \sum_{s=1}^{r}\eta_{\alpha,i}^s C_sH_\ell^{-1}$\\
(ii)\quad\quad$\displaystyle  [\alpha,\Omega_i]-\sum_{s=1}^{r}\eta_{\alpha,i}^s\Omega_s\in\mathscr{A} $\\
(iii)\quad\quad$\displaystyle  \,\, \, \Omega_j^TH_\ell\overline{C_i}+H_\ell\overline{C_i} \Omega_j=\sum_{s=1}^{r} \mu_{i,j}^s H_\ell\overline{C_s} $\\
(iv)\quad\quad$\displaystyle \left[\overline{H_\ell^{-1}\Omega_i^T H_\ell},\Omega_j \right]+C_j\overline{C_i}-\sum_{s=1}^{r}\left(\overline{\mu_{i,j}^s}\Omega_s +\mu_{j,i}^s \overline{H_\ell^{-1}\Omega_s^T H_\ell}\right) \in\mathscr{A} $\\
\end{minipage}
}
\right\}
\end{align}
holds for all $ \alpha\in\mathscr{A}$ and $i,j\in\{1,\ldots,\mathrm{rank}\,K\}$. Note that condition (i) on its own is satisfied automatically by the definition of $\mathscr{A}$, but satisfying (i) and (ii) simultaneously with the same coefficients $\eta_{\alpha,i}^s$ is not automatic.
\end{proposition}
\begin{remark}
Under the identification in \eqref{SPinclusion1},  
\begin{align}\label{g00 subalg. representation}
\mathfrak{g}_{0,0}=\text{span}_{\C}
\left\{\left.
\left(
\begin{array}{cc}
\alpha & 0\\
0 & -H_\ell^{-1} \alpha^T H_\ell
\end{array}
\right)+c\mathrm{Id}\, \right|\, \alpha\in \mathscr{A}
\mbox{ and }c\in\C\right\}.
\end{align}
\end{remark}

Now we are ready to prove Theorem \ref{constant mod symbols implies regular}.
Since $\mathrm{Id}$ belongs to $\mathscr{A}$, by setting $\alpha=\cfrac{1}{2}\mathrm{Id}$, we get $\eta_{\alpha,i}^s= \delta_{i,s}$  from item (i) of \eqref{system}, where $\delta_{i,s}$ denotes the Kronecker symbol. Substituting this $\alpha$ and the corresponding $\eta_{\alpha,i}^s$ into the equation in item (ii) of \eqref{system}, we get that  $\Omega_i\in\mathscr{A}$ for all $i\in\{1,\ldots,r\}$. Then subtracting the matrix of the form appearing in \eqref{g00 subalg. representation} with $\alpha=\Omega_i$ from the matrices appearing in \eqref{Xzero2} as the generators of $\mathfrak X_{0, 2}$ and using \eqref{matrixg0pm2}, we get that the space $\mathfrak g_{0,2}$ belongs to $\mathfrak g_0^{\mathrm{mod}}$. A similar argument implies that $\mathfrak g_{0,-2}$ belongs to $\mathfrak g_0^{\mathrm{mod}}$, and therefore $\mathfrak g_0=\mathfrak g_0^{\mathrm{mod}}$. Accordingly, by Proposition \ref{constmodprop}, if a CR structure satisfies the hypothesis of Theorem \ref{constant mod symbols implies regular} then $\mathfrak g_0$ is a subalgebra of $\mathfrak{csp}(\mathfrak g_{-1})$ and therefore $\mathfrak g^0$ is a regular symbol. This completes the proof of Theorem \ref{constant mod symbols implies regular}.

\section{Reduction to level sets of modified symbols}
\label{reduction_and_nongenericity_section} 
According to Theorem \ref{constant mod symbols implies regular}, a CR structure with a nonregular symbol cannot have a constant 
modified symbol on $P^0$. Consequently, for such structures the upper bound for the algebra of infinitesimal symmetries  given in  Theorem \ref{maintheorm} is far from being sharp in the case of nonregular $\mathfrak g^0$ and can be improved under appropriate natural assumptions.  The standard way to deal with structures with nonconstant invariants (e.g., the modified CR symbol in our case)  is to make a reduction to the level set of these invariants.

In more detail, given an  abstract modified CR symbol $\mathfrak g^{0, \mathrm {mod}}$ of type $\mathfrak g^0$ the set $P^0(\mathfrak g^{0, \mathrm {mod}})$ consisting of all $\psi\in P^0$ such that $\mathfrak g^{0, \mathrm {mod}}(\psi)=\mathfrak g^{0, \mathrm {mod}}$ is called \emph{the level set of $\mathfrak g^{0, \mathrm {mod}}$} in $P^0$. Assume that  $P^0(\mathfrak g^{0, \mathrm {mod}})$ is a smooth submanifold of $P^0$ such that 
\begin{equation}
\label{onto condition}
\C M=\mathrm{pr}\left(P^0(\mathfrak g^{0, \mathrm {mod}})\right).
\end{equation}

The condition in \eqref{onto condition} is motivated by the study of homogeneous CR manifolds, that is, CR manifolds whose groups of symmetries act transitively. If $\theta_0\left(T_\psi (P^0(\mathfrak g^{0, \mathrm {mod}}))_\gamma\right)$ 
is the same subspace $\widetilde {\mathfrak g^{0, \mathrm {mod}}}$ for all $\psi\in P^0(\mathfrak g^{0, \mathrm {mod}})$, then we say that $P^0(\mathfrak g^{0, \mathrm {mod}})$ is a \emph{reduction of $P^0$ with constant reduced modified symbol $\widetilde {\mathfrak g^{0, \mathrm {mod}}}$.}

If, on the other hand, $\theta_0\left(T_\psi P^0(\mathfrak g^{0, \mathrm {mod}})\right)$ is not constant on $P^0(\mathfrak g^{0, \mathrm {mod}})$, then we can repeat the process of restriction to a level set. If the chosen level set projects onto a set containing $\C M$ and the image of the tangent spaces to it under $\theta_0$ is a fixed subspace of $\mathfrak{csp}(\mathfrak g_{-1})$ independently  of a point of the level set, then we again obtain the reduction of $P^0$  with constant reduced modified symbol (after two steps of reduction), and if not then we can repeat the process again. In this way, at least in the homogeneous case, after a finite number of steps we  will arrive to a submanifold $P^{0,\mathrm{red}}$ of $P^0$ that projects onto $\C M$, and such that the tangent spaces to it are mapped under $\theta_0$ is a fixed subspace $\mathfrak g^{0,\mathrm{red}}$ of $\mathfrak {csp}(\mathfrak g_{-1})$.  Also, at least in the homogeneous case, in every step of this reduction procedure the level set can be chosen so that it has a nonempty intersection with $\Re P^0$ and hence $\mathfrak g^{0,\mathrm{red}}$ will inherit an involution from the involution defined on $\mathfrak g^{0, \mathrm {mod}}(\psi)$ for any $\psi\in \Re P^0\cap P^{0,\mathrm{red}}$. In this case we will say that the bundle  $P^0$ associated with the CR structure admits a reduction with constant reduced modified symbol $\mathfrak g^{0,\mathrm{red}}$.

Note that the  subspace $\mathfrak g^{0,\mathrm{red}}$ must be a Lie subalgebra of  $\mathfrak {csp}(\mathfrak g_{-1})$ by literally the same arguments as in the proof of Proposition \ref{constmodprop}. These constructions motivate the following definition, which generalizes Definition \ref{abstractmod}.
\begin{definition}
\label{abstractred}
Fix a CR symbol $\mathfrak g^0=\mathfrak{g}_{-}\oplus \mathfrak{g}_{0,-2}\oplus \mathfrak{g}_{0,0}\oplus \mathfrak{g}_{0,2}$ of the form in \eqref{CRsymbol_2}. A Lie subalgebra   $\mathfrak g^{0, \mathrm{red}}= \mathfrak g_-\oplus \mathfrak g_0^{ \mathrm {red}}$ of $\mathfrak g_- \rtimes\mathfrak {csp}(\mathfrak g_{-1}) $ is called an abstract reduced modified CR symbol (with involution) of type $\mathfrak g^0$ if the following properties hold:
\begin{enumerate}
    \item $\dim \mathfrak{g}^{\mathrm{red}}_{0}=\dim\left(\mathfrak g_0^{\mathrm{red}}\cap \left(\mathfrak{csp}(\mathfrak{g}_{-1})\right)_{0,0}\right)+2\dim \mathfrak g_{0,2}\leq \dim \mathfrak{g}_{0}$;
    \item $\mathfrak g_0^{\mathrm{red}}\cap \left(\mathfrak{csp}(\mathfrak{g}_{-1})\right)_{0,0}\subset\mathfrak g_{0,0} $;
    \item $\left(\mathfrak g_0^{ \mathrm {red}}\right)_{0,\pm2}=\mathfrak{g}_{0,\pm 2}$, where $\left(\mathfrak g_0^{ \mathrm {red}}\right)_{0,\pm2}$ stands for the image under the projection to the subspace $\left(\mathfrak {csp}(\mathfrak g_{-1})\right)_{0,\pm2}$ with respect to the splitting \eqref{csp decomposition} of $\mathfrak {csp}(\mathfrak g_{-1})$;
    \item  The subspace $\mathfrak g_0^{\mathrm{red}}$ is invariant with respect to the involution on $\mathfrak g_{-1}\oplus\mathfrak {csp}(\mathfrak g_{-1})$.
\end{enumerate}
\end{definition}

To any abstract reduced modified symbol  $\mathfrak g^{0,\mathrm{red}}$ that is a Lie algebra, we construct  corresponding special homogeneous CR structures as follows.  Set $\mathfrak g_{0,0}^{\mathrm{red}}:= \mathfrak g_0^{\mathrm{red}}\cap \left(\mathfrak{csp}(\mathfrak{g}_{-1})\right)_{0,0}$. Denote by $G^{0,\mathrm{red}}$ and $G_{0,0}^{\mathrm{red}}$ connected Lie groups with Lie algebras $\mathfrak  g^{0,\mathrm{red}}$ and $\mathfrak g_{0,0}^{\mathrm{red}}$, respectively, such that $G_{0,0}^{\mathrm{red}}\subset G^{0,\mathrm{red}}$, and denote by  $\Re G^{0,\mathrm{red}}$ and $\Re G_{0,0}^{\mathrm{red}}$ the corresponding real parts with respect to the involution on $\mathfrak g^{0,\mathrm{red}}$, meaning that $\Re G^{0,\mathrm{red}}$ and $\Re G_{0,0}^{\mathrm{red}}$ are the maximal subgroups of $ G^{0,\mathrm{red}}$ and $ G_{0,0}^{\mathrm{red}}$ whose tangent spaces belong to the left translations of the fixed point set of the involution on $\mathfrak g^{0,\mathrm{red}}$ and on $\mathfrak g_{0,0}^{\mathrm{red}}$ respectively.

Let $ M_0^{\C} = G^{0, \mathrm{red}}/ G_{0,0}^{\mathrm{red}}$ and $M_0= \Re G^{0,  \mathrm{red}}/\Re G_{0,0}^{ \mathrm{red}}$. In both cases here we  use left cosets. For every pair $(i,j)$ with $i<0$, let $\widehat D_{i,j}^{\mathrm{flat}}$  be the left-invariant distribution on $G^{0,  \mathrm{red}}$ such that it is equal to $\mathfrak{g}_{i,j}$ at the identity. Also, for $j=\pm 2$, let $\widehat D_{0, \pm 2}^{\mathrm{flat}}$ be the left-invariant distributions equal to 
\begin{align}\label{reduced modified signed components}
\mathfrak  g_0^{\mathrm{red}}\cap \left(\left(\mathfrak{csp}(\mathfrak{g}_{-1})\right)_{0,0}\oplus \left(\mathfrak{csp}(\mathfrak{g}_{-1})\right)_{0,j} \right)
\end{align}
at the identity. Since all $\mathfrak{g}_{-1,j}$ as well as the spaces in \eqref{reduced modified signed components} are invariant under the adjoint action of $G_{0,0}^{\mathrm{red}}$, the push-forward of each $\widehat D_{i,j}^{\mathrm{flat}}$ to $M_0^\mathbb C$  is a well defined distribution, which we denote by $D^{\mathrm{flat}}_{i,j}$. Let $D_{-1}^{\mathrm{flat}}$ be the distribution which is the sum of $D_{i,j}^{\mathrm{flat}}$ with $i=-1$. We restrict all of these distributions to $M_0$, considering them as subbundles of the complexified tangent bundle of $M_0$. The distribution $ H^{\mathrm{flat}}:=D_{-1,1}^{\mathrm{flat}}\oplus D_{0,2}^{\mathrm{flat}}$ defines a CR structure of hypersurface type, and, by construction, the corresponding bundle  $P^0$ associated with this CR structure admits a reduction with constant reduced modified symbol $\mathfrak g^{0,\mathrm{red}}$. The structure $H^{\mathrm{flat}}$ on the constructed homogeneous model $M_0$ is called the \emph{flat CR structure with constant reduced modified symbol $ \mathfrak{g}^{0,\mathrm{red}}$}.
\begin{theorem}
\label{maintheorconst} Assume that for a given a CR symbol $\mathfrak g^0$ there exists a reduced modified symbol $\mathfrak g^{0, \mathrm{red}}$ of type $\mathfrak g^0$ with finite dimensional Tanaka prolongation. Then the following three statements hold.
\begin{enumerate}
\item 
Given a $2$-nondegenerate, hypersurface-type CR structure such that the corresponding bundle $P^0$ contains a subbundle $P^{0,\mathrm{red}}$ with the reduced  modified symbol  $\mathfrak g^{0, \mathrm{red}}$, there exists a bundle over $\Re P^{0,\mathrm{red}}:=P^{0,\mathrm{red}}\cap\Re P^{0}$ of dimension equal to $\dim_\C\mathfrak u(\mathfrak g^{0, \mathrm{red}})$ that admits a canonical absolute parallelism;
\item The dimension of the algebra of infinitesimal symmetries of a $2$-nondegenerate, hypersurface-type CR structure of item (1) is not greater than $\dim_\C \mathfrak u(\mathfrak g^{0, \mathrm{red}})$. Moreover, 
if we assume that the CR symbol $\mathfrak g^0$ is recoverable then the  algebra of infinitesimal symmetries of the flat CR structure with constant reduced modified symbol  $\mathfrak g^{0, \mathrm{red}}$ is isomorphic to the real part $\Re \mathfrak u(\mathfrak g^{0, \mathrm{red}})$ of $\mathfrak u(\mathfrak g^{0, \mathrm{red}})$;
\item
If  the CR symbol $\mathfrak g^0$ is recoverable then any CR structure with the constant reduced modified symbol  $\mathfrak g^{0, \mathrm{red}}$ whose algebra of infinitesimal symmetries has real dimension equal to $\dim_\C \mathfrak u(\mathfrak g^{0, \mathrm{red}})$  is locally to equivalent to the flat CR structure with reduced modified symbol  $\mathfrak g^{0, \mathrm{red}}$.
\end{enumerate}
\end{theorem}

Items (1) and (2) of this theorem follow from the standard Tanaka theory for constant symbols \cite{tanaka1} applied, after the reduction of the bundle $P^0$, to the bundle $P^{0, \mathrm{red}}$.
Item (3) follows from the fact that $\mathfrak{u}(\mathfrak g^{0, \mathrm{red}})$ always contains the grading element $E$; that is, the element  such that $[E, x] =i x$, for $x\in \mathfrak{g}_i$ with $i\in\{-1, -2\}$, which also implies that $[E, x]=0$ for $x\in \mathfrak g_0^{ \mathrm{red}}$. This grading element is the generator of the natural family of dilations on the fibers of $P^0$. Explicitly, given a point $\psi_0\in P^0$, the grading element in $\mathfrak g^{0, \mathrm{red}}$ is the velocity at $t=0$ of the curve $\psi(t)$ such that $\left.\psi(t)\right|_{\mathfrak{g}_{-2}}=\left.\psi(t)\right|_{\mathfrak{g}_{-2}}=e^{2t}\left.\psi_0\right|_{\mathfrak{g}_{-2}}$ and $\left.\psi(t)\right|_{\mathfrak{g}_{-1}}=e^t \left.\psi_0\right|_{\mathfrak{g}_{-1}}$.
The fact that the grading element is tangent to the level sets after reduction follows from the fact that this orbit $\psi(t)$ belongs to the same level set of modified CR symbols. 
Further, the algebra of infinitesimal symmetries of a CR structure has a natural filtration such that, in the case where the dimension of this algebra is equal to $\dim_\C \mathfrak u(\mathfrak g^{0, \mathrm{red}})$, the associated graded algebra is isomorphic to $\mathfrak{u}(\mathfrak g^{0, \mathrm{red}})$. The existence of this grading element in the reduced symbol implies that the filtered algebra of infinitesimal symmetries is isomorphic to its associated graded algebra (considered as filtered algebras) (see \cite[Lemma 3]{doubkom}), that is, to $\mathfrak u(\mathfrak g^{0, \mathrm{red}})$, which implies that the CR structure is locally equivalent to the flat one. 

\begin{remark}
Concerning the base space of the bundle on which the parallelisms of Theorems \ref{maintheorm} and \ref{maintheorconst} are defined, we construct them through application of a Tanaka prolongation procedure to the induced DLC structure on the Levi leaf space $\mathcal{N}$, and to match classical Tanaka theory, one should regard the total space on which the parallelism is defined as bundle with base space $\mathcal{N}$. Its projection to $\mathcal{N}$, however, naturally factors through a projection to $\mathbb{C} M$, so it is also reasonable to alternatively regard it as a bundle over $\mathbb{C} M$. If the considered CR structures are recoverable then their local equivalence problem is reduced to that of their associate DLC structures, and the canonical parallelisms constructed on bundles over $\mathcal{N}$ naturally give canonical parallelisms on bundles over $M$. Regardless of recoverability the CR structure's symmetry group has a naturally induced faithful action on the parallelisms of Theorems \ref{maintheorm} and \ref{maintheorconst}.
\end{remark}

We conclude this section with the following corollary, which relates our Theorem \ref{maintheorconst} with the main theorem of \cite[Theorem 3.2]{porter2017absolute}.
\begin{corollary}
If the CR symbol $\mathfrak g^0$ is regular and recoverable then its usual Tanaka prolongation and the bigraded Tanaka prolongation defined in \cite[section 3]{porter2017absolute} coincide.
\end{corollary}
\begin{proof}
In the case of regular $\mathfrak g^0$ there is a flat CR structure with the constant modified symbol equal to $\mathfrak g^0$ so that there is no reduction of the bundle $P^0$, and, from the assumption of recoverability, item (2) of Theorem \ref{maintheorconst} gives the same algebra of infinitesimal symmetry as item (2) of  \cite[Theorem 3.2]{porter2017absolute}. The former algebra is the usual Tanaka prolongation of $\mathfrak g^0$ whereas the latter is the bigraded one. 
\end{proof}

Note that without the assumption of recoverability the statement of the previous corollary is wrong. For example, for $\mathrm{rank}\, K=1$, if $\mathrm{ad} K$ is generated by a rank 1 operator, then the usual Tanaka prolongation is infinite dimensional and the bigraded Tanaka prolongation is not.  

\section{Generic CR symbols and nonexistence of homogeneneous models}
\label{nongenericity section}

In this section we prove the following theorem.
\begin{theorem}
\label{gen_nonreg}
For any fixed rank $r>1$, in the set of all CR symbols associated with 2-nondegenerate, hypersurface-type CR manifolds of odd dimension greater than $4r+1$ with rank $r$ Levi kernel and with reduced Levi form of arbitrary signature, the CR symbols not associated with any homogeneous model are generic. For $r=1$, the same statement holds if the reduced Levi form is sign-definite, that is, when the CR structure is pseudoconvex.   
\end{theorem}

\begin{remark}
We believe that the pseudoconvexity assumption in the case of $r=1$ is not essential and can be omitted through more subtle analysis of the corresponding modification of system \eqref{system} than the analysis we apply for the pseudoconvex case (see Remark \ref{off-diagonal_rem} below for more detail). For the discussion on sharpness of the lower bounds for the dimension of the ambient manifold  see Remark \ref{dimboundrem} below. The goal here is to exhibit the phenomena of non-existence of homogeneous models for generic basic data such as the CR symbol rather than to get the most general results in this direction.
\end{remark}

\begin{proof}
The proof consists of a series of lemmas:

The following lemma is about the structure of the algebra $\mathscr A$, defined in \eqref{intersection algebra}, for generic CR symbols. Note that the inclusion 
\begin{equation}
\label{linein}
\{s\,\mathrm{Id}\,|\,s\in \mathbb C\}\subset\mathscr A.
\end{equation}
always holds.

\begin{lemma}\label{generically minimal inersection algebra}
For any fixed rank $r$, in the set of all CR symbols associated with 2-nondegenerate, hypersurface-type CR manifolds of odd dimension greater than $4r+1$ with rank $r$ Levi kernel, the subset of CR symbols such that the algebra 
\begin{equation}
\label{minimalA}
\mathscr{A}=\{s\,\mathrm{Id}\,|\, s\in\C\}
\end{equation}
is generic.
\end{lemma}
\begin{proof}
Fix a CR symbol $\mathfrak g^0$, and, still using $m=n-r$, let $H_\ell$ and $\{C_j\}_{j=1}^r$ be the  $m\times m$ matrices associated with $\mathfrak{g}^0$ as in \eqref{matrixg0pm2}, where $H_\ell$ represents the reduced Levi form. Since the system for the algebra $\mathscr{A}$ given by \eqref{firstalgebra}-\eqref{secondalgebra} is overdetermined and linear (in $\alpha$) and the inclusion \eqref{linein} always holds, to prove our lemma it is enough to prove that for fixed signature of the reduced Levi form $\ell$  (or equivalently, signature of the Hermitian matrix $H_\ell$) there exists at least one tuple of matrices $\{C_j\}_{j=1}^r$ for which \eqref{minimalA} holds.

Assume that the matrices $C_j$ are nonsingular for all $1\leq j\leq r$. If 
\begin{equation}
\label{Ai_alpha}
A_i=\alpha C_iH_\ell^{-1},
\end{equation}
with $\alpha$ satisfying \eqref{firstalgebra}, then we have that $A_i+A_i^T\in  \mathrm{span}\{C_jH_\ell^{-1}\}_{j=1}^{r}$. Recalling that the set of  solutions of the matrix equation $A+A^T=S$ (with respect to $A$ for a fixed symmetric matrix $S$) is $\frac{1}{2}S+\mathfrak{so}(m)$ and that $\alpha= A H_\ell C_i^{-1}$ from \eqref{Ai_alpha}, we have that $\alpha$ satisfies system \eqref{firstalgebra} if and only if 
\[
A_i\in\mathrm{span}\{C_jH_\ell^{-1}\}_{j=1}^{r}+\mathfrak{so}(m)\quad\quad\forall\, i\in\{1,\ldots,m\},
\]
which is equivalent to
\begin{equation}
\label{firstalgebra1}
\alpha\in \bigcap_{i=1}^r \left(\mathrm{span}\{C_j C_i^{-1}\}_{j=1}^r+ \mathfrak{so}(m) H_\ell C_i^{-1}\right).
\end{equation}
Similar analysis of \eqref{secondalgebra} implies that  $\alpha$ satisfies system \eqref{secondalgebra} if and only if
\begin{equation}
\label{secondalgebra1}
\alpha\in \bigcap_{i=1}^r \left(\mathrm{span}\left\{\overline C_i^{-1} \overline C_j\right\}_{j=1}^r+\overline C_i^{-1} H^{-1}\mathfrak{so}(m) \right).
\end{equation}
Hence, the algebra $\mathscr{A}$ satisfies
\begin{equation}
\label{intersection for generic result lemma}
\mathscr A=\bigcap_{i=1}^r\left( \left(\mathrm{span}\left\{C_j C_i^{-1}\right\}_{j=1}^r+ \mathfrak{so}(m) H_\ell C_i^{-1}\right)\bigcap \left( \mathrm{span}\left\{\overline C_i^{-1} \overline C_j\right\}_{j=1}^r+  \overline C_i^{-1} H^{-1}\mathfrak{so}(m)\right)\right).
\end{equation}

Choose a basis in $\mathfrak g_{-1,1}$  such that 
\begin{equation}
\label{Hdiag}
H_\ell=\mathrm{diag}(\underbrace{1,\ldots, 1}_{q\, \mathrm {entries}}, -1, \ldots, -1).
\end{equation}

We consider the splitting of $\mathfrak {gl}(m)$ into the space $\mathfrak {gl}^{\mathrm{diag}}(m)$ of diagonal $m\times m$ matrices  and the space $\mathfrak {gl}^{\mathrm{holl}}(m)$ of $m\times m$ matrices with all zeros on the diagonal, sometimes called hollow matrices,
\begin{equation}
\label{splitdiag}
\mathfrak {gl}(m)=\mathfrak {gl}^{\mathrm{diag}}(m)\oplus \mathfrak {gl}^{\mathrm{holl}}(m).
\end{equation}

Since, by our assumptions, $m>r$,
we can take a special tuple $\{C_j\}_{j=1}^r$  such that 
every matrix   $C_j$ is nonsingular and diagonal, that is,
\begin{equation}
\label{Cdiag}
C_j=\mathrm{diag}(\lambda_{j,1}, \ldots\lambda_{j,m}).
\end{equation}
Accordingly, for every  $i\in\{1,\ldots r\}$, we have the following inclusions:
\begin{enumerate}
\item The spaces $\mathrm{span}\left\{C_j C_i^{-1}\right\}_{j=1}^{r}$ and $\mathrm{span}\left\{\overline C_i^{-1} \overline C_j\right\}_{j=1}^{r}$ belong to $\mathfrak {gl}^{\mathrm{diag}}(m)$; 
\item The spaces 
$\mathfrak{so}(m) H_\ell C_i^{-1}$ and 
$\overline C_i^{-1} H^{-1}\mathfrak{so}(m)$ belong to $\mathfrak {gl}^{\mathrm{holl}}(m)$.
\end{enumerate}

Based on the splitting in \eqref{splitdiag}, for $\alpha\in \mathfrak{gl}(m)$,
we let $\alpha^{\mathrm{diag}}\in\mathfrak {gl}^{\mathrm{diag}}(m)$  
and 
$\alpha^{\mathrm{holl}}\in \mathfrak {gl}^{\mathrm{holl}}(m)$
denote the matrices for which
\begin{equation}
\label{alphasplit}
\alpha= \alpha^{\mathrm{diag}}+\alpha^{\mathrm{holl}}.
\end{equation}
From \eqref{intersection for generic result lemma} and the splitting in \eqref{splitdiag} it follows that
\begin{align}\label{hollowpart}
 \alpha^{\mathrm{holl}}\in \bigcap_{i=1}^{r}\left(\left(\mathfrak{so}(m) H_\ell C_i^{-1}\right)\bigcap \left( \overline C_i^{-1} H^{-1}\mathfrak{so}(m)\right)\right) \quad\quad\forall\,\alpha\in\mathscr{A}
\end{align}
and
\begin{align}\label{diagpart}
\alpha^{\mathrm{diag}}\in
\bigcap_{i=1}^{r}\left( \left(\mathrm{span}\left\{C_j C_i^{-1}\right\}_{j=1}^{r}\right)\bigcap \left( \mathrm{span}\left\{\overline C_i^{-1} \overline C_j\right\}_{j=1}^{r}\right)\right) \quad\quad\forall\,\alpha\in\mathscr{A}.
\end{align}
In particular, for a fixed matrix $\alpha\in\mathscr{A}$, by \eqref{firstalgebra1} and \eqref{hollowpart},
there exists $B$ and $\widetilde B$ in $\mathfrak{so}(m)$  such that
\begin{equation}
\label{hollow part of alpha}
\alpha^{\mathrm{holl}}= B H_\ell C_1^{-1}=\overline C_1^{-1} H^{-1} \widetilde B.
\end{equation}

Let 
\begin{equation}
\label{eps}
  \varepsilon_i^q= \begin{cases}
  1, & 1\leq i\leq q,\\
  -1, & q+1\leq i\leq m,
  \end{cases}  
\end{equation}
and in the following calculation we denote the $(i,j)$ entry of a matrix $X$ by $X_{i,j}$.
Using \eqref{Hdiag} and \eqref{Cdiag}, we get that 
$$(B H_\ell C_1^{-1})_{i,j}= \frac{\varepsilon_j^q B_{i, j}}{\lambda_{1,j}} 
\quad\mbox{ and }\quad  
\left(\overline C_1^{-1} H^{-1}B\right)_{i,j}=\frac{\varepsilon_i^q \widetilde B_{i, j}}{\overline{\lambda_{1,i}}}.$$
From this and \eqref{hollow part of alpha} we have
\begin{equation}
\label{bbtilde}
B_{i,j}= \frac{\varepsilon_i^q \varepsilon_j^q\lambda_{1,j}} {\overline{\lambda_{1,i}}} \widetilde B_{i, j},
\end{equation}
and hence, since $B$ and $\widetilde B$ are skew symmetric,
\begin{equation}
\label{bbtildeAlt}
\frac{\varepsilon_i^q \varepsilon_j^q\lambda_{1,j}} {\overline{\lambda_{1,i}}}
\widetilde B_{i, j}
=B_{i,j}
=-B_{j,i}
=-\frac{\varepsilon_{i}^q\varepsilon_{j}^q\lambda_{1,i}}{\overline{\lambda_{1,j}}}\widetilde B_{j,i}
=\frac{\varepsilon_{i}^q\varepsilon_{j}^q\lambda_{1,i}}{\overline{\lambda_{1,j}}}\widetilde B_{i,j}.
\end{equation}
If we now assume that 
\begin{align}\label{generic assumption about C}
\left|\lambda_{1,i}\right|\neq \left|\lambda_{1,j}\right|\quad\quad\forall\, i\neq j
\end{align}
then it follows from \eqref{bbtildeAlt} that $B=0$, and hence $\alpha^{\mathrm{holl}}=B H_\ell C_1^{-1}=0$, that is,
$\alpha=\alpha^{\mathrm{diag}}$.

In other words, \eqref{generic assumption about C} implies that \eqref{intersection for generic result lemma} can be simplified to
\begin{equation}
\label{intersection for generic result lemma Simplified}
\mathscr A=\bigcap_{i=1}^r\left( \left(\mathrm{span}\left\{C_j C_i^{-1}\right\}_{j=1}^r\right)\bigcap \left( \mathrm{span}\left\{\overline C_i^{-1} \overline C_j\right\}_{j=1}^r\right)\right).
\end{equation}
If $r=1$ then \eqref{intersection for generic result lemma Simplified} is equivalent to \eqref{minimalA}, which is what we wanted to show, so let us now assume that $r>1$. 

Now use that if $\alpha\in \mathcal A$, 
then there exist $\{\nu_{i,j}\}_{i, j=1}^r$ such that  
\begin{equation}
\label{alpha_diag_eq}
\alpha=\sum_{j=1}^r \nu_{i,j} C_j C_i^{-1}=\nu_{i,i} \mathrm{Id}+\sum_{j\neq i} \nu_{i,j} C_j C_i^{-1},
\end{equation}
which implies by comparing the diagonal entries that for every  $i> 1$
\[
\sum_{j=1}^r \frac{\lambda_{j,s}}{\lambda_{1,s}}\nu_{1,j}=\sum_{j=1}^r \frac{\lambda_{j,s}}{\lambda_{i,s}}\nu_{i,j}, \quad \forall s=\{1, \ldots, m\}
\]
or, equivalently, 
\begin{equation}
\label{diagsystem}
\sum_{j>1} \lambda_{j,s}\lambda_{i,s}\nu_{1,j}-\sum_{1\leq j\leq r, j\neq i} \lambda_{j,s}\lambda_{1,s}\nu_{i,j}+ \lambda_{1,s}\lambda_{i,s}(\nu_{1, 1}-\nu_{i, i})=0, \quad \forall 1\leq s\leq m , 2\leq i\leq r,
\end{equation}
which is a system of $(r-1)m$ linear homogeneous equations with respect to $r^2-1$ unknowns 
$\{\nu_{i, j}|1\leq i, j\leq r, i\neq j\}$ and $\{\nu_{1, 1}-\nu_{i, i}| 2\leq i\leq r\}$. 

It remains to note that $(r-1)m\geq r^2-1$ if and only if $m>r$  and for generic tuples 
$\{\lambda_{i, s}|1\leq i\leq r, 1\leq s\leq m\}$ the rank of the matrix in the system \eqref{diagsystem} is $r^2-1$. Indeed, this matrix (after appropriate rearrangement of columns)
has a block-triangular form with $r-1$ diagonal blocks. 

The maximal size diagonal block has size $m\times (2r-2)$ and it is obtained from columns that use variables appearing in the equations in system \eqref{diagsystem} with $i=2$ (i.e., $\{\nu_{1, j}|j>1\}$,  $\{\nu_{2, j}|j>1\})$ and $\{\nu_{1, 1}-\nu_{2, 2}| 2\leq i\leq r\}$ and its $s$th row consists of evaluations of quadratic monomials $x_i x_j$, $i\leq j$ such that at least one $i$ or $j$ takes values in $\{1, 2\}$ at the point  $(x_1, \ldots x_r)=(\lambda_{1, s}, \ldots,\lambda_{r, s})$. 
This diagonal block has maximal rank at generic points; the determinant of each of its maximal minors is a nonzero polynomial in the $\lambda$'s, because when calculating  this minor as the alternating (according to the signature of permutations) sum of the corresponding products of the entries of the submatrix corresponding to this minor there cannot be cancellation, as each term of this alternating sum gives a unique monomial. The latter follows from the fact that distinct rows of this diagonal block depend on disjoint sets of variables and that the monomials at different columns are different.  

The other $r-2$ diagonal blocks are of size $m\times r$ and they are  parameterized by $i>2$. The block corresponding to a given $i>2$ is  obtained from columns of the matrix of the system  \eqref{diagsystem} that correspond to the variables $\{\nu_{i, j}| 1\leq j\leq r,  j\neq i\}$  and $\eta_{1,1}-\eta_{i,i}$. Similar to the maximal size diagonal block, the $s$th row of the diagonal block under consideration consists of evaluations of quadratic monomials $x_1 x_j$ , $1\leq j\leq r$ such that at the point  $(x_1, \ldots x_r)=(\lambda_{1, s}, \ldots,\lambda_{r, s})$. This diagonal block has maximal rank at generic points by the same reason as in the previous paragraph.

Therefore, the matrix  corresponding to System \eqref{diagsystem} has maximal rank, which implies that $\nu_{i, j}=0$ for all  $i\neq j$ and $\nu_{i,i}=\nu_{1,1}$ for all $i$. So, by \eqref{alpha_diag_eq} the matrix $\alpha$ must be a multiple of identity, which proves \eqref{minimalA}. 
\end{proof}

\begin{remark}
\label{dimboundrem}
The lower bound $4r+3$ for the dimension of manifold in Lemma \ref{generically minimal inersection algebra} is sharp for $r=1$, as for $5$-dimensional manifold there is only one CR symbol and it is regular and does not satisfy \eqref{minimalA}. However, for $r>1$ 
this bound is strictly greater than the minimal dimension for which nonregular CR symbols exist, so we expect that this bound is not sharp, but our method of  proof using diagonal $C$'s, cannot improve it. 
\end{remark}

\begin{lemma}
\label{gennonreglem}
For fixed $n=\mathrm{rank}\,H$ and $r=\mathrm{rank}\, K$ 
such that the strict inequality in \eqref{estim1} holds,
the nonregular symbols constitute a generic subset in the set of all CR symbols.
\end{lemma}

\begin{proof}
We prove this using the same principle that was applied for the proof of Lemma \eqref{generically minimal inersection algebra}; that is, we will characterize nonregularity as nonsolvability of a certain overdetermined algebraic system and then find one example of a CR symbol for which this system has no solution.

Given a CR symbol $\mathfrak{g}_0$ represented by the matrices $\{C_{i}\}_{i=1}^r$ as in \eqref{matrixg0pm2}, the condition for regularity of $\mathfrak{g}_0$ in Remark \ref{regmatrixrem} is given by the system of equations \eqref{regmatrix}.

Let $m=n-r$ as before. First  consider the case when $n>2r$ or $m>r$.
Working with respect to a basis of $\mathfrak{g}_{-1,1}$ such that $H_\ell$ is as in \eqref{Hdiag}, choose $\{C_i\}_{i=1}^r$ and $\{\lambda_{i,s}\}_{1\leq i\leq r,1\leq s\leq m}^m$ satisfying \eqref{Cdiag}.
The system \eqref{regmatrix} can be rewritten as
\begin{equation}
\label{regmatrix1}
\sum_{l=1}^r\nu_{i,j,k}^l\lambda _{l,s}=2\lambda_{i,s}\overline{\lambda_{j,s}}\lambda_{k,s}\quad\quad\forall\, i,j,k\in\{1,\ldots, r\}, s\in{1, \ldots, m}
\end{equation}
for some unknowns $\{\nu_{i,j,k}^l\}$. Fix the triple $i, j, k$ and consider the $m\times (r+1)$ matrix such that its $s$th row consists of evaluations of monomials $x_l$ with $1\leq l \leq r$ and $2 x_i \overline x_j x_k$  at the point  $(x_1, \ldots x_r)=(\lambda_{1, s}, \ldots,\lambda_{r, s})$. By assumption $m>r$, so the solvability of the linear system \eqref{regmatrix1} with respect to $\{\nu_{i,j,k}^l\}_{ l=1}^m$ is equivalent to the fact that this matrix, which is exactly the augmented matrix of this nonhomogeneous system, has rank not greater than $r$. On the other hand,  by the same arguments applied at the end of the proof of Lemma \ref{generically minimal inersection algebra} this matrix has rank $r+1$ for a generic tuple  of diagonal matrices $\{C_i\}_{i=1}^m$, so the system is not solvable generically.

Now consider the case when $m\leq r$ but the strict equality in \eqref{estim1} holds. Note that this implies $m>1$. Choose the tuple  $\{C_i\}_{i=1}^r$ such that the first $m-1$ elements in it are diagonal as in  \eqref{Cdiag} and the rest have zero on the diagonal and consider the matrix equation \eqref{regmatrix} for example for $i=j=k=1$. By construction, using the splitting  \eqref{splitdiag}, we get that
\[
C_1\overline{C_1} C_1 \in \mathrm{span}\{C_i\}_{i=1}^{m-1},
\]
which yields exactly the same relation as in the case $r=m-1$, and we can repeat the argument of the case $m>r$.
\end{proof}

Now we are ready to prove our theorem. We will show that as a desired generic set of CR symbols in the theorem one can  take the set of nonregular symbols  with the algebra $\mathscr A$ satisfying \eqref{minimalA} and maybe some additional generic conditions.

As is done in the proof of Lemma \ref{generically minimal inersection algebra}, after fixing a symbol $\mathfrak g^0$ with these generic properties, we let $H_\ell$ and $\{C_j\}_{j=1}^r$ be a set of $m\times m$ matrices associated with $\mathfrak{g}^0$. There exists a reduced modified CR symbol of type $\mathfrak g^0$ if and only if there exist $m\times m$ matrices $\{\Omega_j\}_{j=1}^r$ such that the system of relation \eqref{system} can be satisfied after replacing the algebra $\mathscr{A}$ with some subalgebra $\mathscr{A}_0\subset \mathscr{A}$. So, to produce a contradiction, let us assume that there exists a reduced modified CR symbol $\mathfrak{g}^{0,\mathrm{red}}$ of type $\mathfrak g^0$, and fix such corresponding $\{\Omega_j\}_{j=1}^r$ and $\mathscr{A}_0$. In particular, under the identification in \eqref{SPinclusion1}, $\mathfrak{g}^{0,\mathrm{red}}$ is spanned by matrices of the form in \eqref{Xzero2} along with block diagonal matrices of the form in \eqref{g00 subalg. representation} but with $\mathscr{A}$ replaced by $\mathscr{A}_0$.

If $\mathrm{Id}$ belongs to the subalgebra $\mathscr{A}_0$ then, for each $i$, the first two conditions in \eqref{system} imply that $\Omega_i$ belongs to $\mathscr{A}_0$, but this implies that $\mathfrak{g}^{0}$ is regular, contradicting our assumptions. So $\mathrm{Id}$ does not belong to $\mathscr{A}_0$, and hence, by \eqref{minimalA},
\begin{align}\label{Azero is zero condition}
\mathscr{A}_0=0.
\end{align}

Similar to arguments of Lemma \ref{generically minimal inersection algebra}, since the system of relations  (iii) and (iv) from  \eqref{system} with $\mathscr{A}_0=0$ is overdetermined and  algebraic (with respect to the unknown matrices $\Omega_i$), then by the classical elimination theory in order to prove generic nonexistence of solutions of this system it is enough to prove that for fixed signature of the reduced Levi form $\ell$  (or, equivalently, signature of the Hermitian matrix $H_\ell$) there exists at least one tuple of matrices $\{C_j\}_{j=1}^r$ for which this system of equation is incompatible.

First consider the case $r>1$, which is simpler. For each $i$, condition (iii) of \eqref{system} means that $\alpha=\Omega_i$ satisfies \eqref{secondalgebra}, or equivalently, $\alpha=\overline{H_\ell}^{-1}\Omega_i^*\overline{H_\ell}$ satisfies \eqref{firstalgebra} and we can repeat the arguments of the proof of Lemma \ref{generically minimal inersection algebra} after formula \eqref{intersection for generic result lemma Simplified} to show that  for generic tuple of diagonal matrices $\{C_j\}_{j=1}^r$ the matrix $\Omega^{\mathrm{diag}}_i$ is a multiple of the identity matrix, noticing that in that part of the proof of Lemma \ref{generically minimal inersection algebra} we only used that $\alpha$ satisfies \eqref{firstalgebra}.

Now we apply an argument similar to the one in the proof of  Lemma \ref{generically minimal inersection algebra} between \eqref{hollow part of alpha} and  \eqref{intersection for generic result lemma Simplified} to conclude that $\Omega_i^{\mathrm{holl}}=0$. In more detail, by analogy with \eqref{hollowpart} , taking into account that $\alpha=\Omega_i$ satisfies \eqref{secondalgebra} only, we have that
\begin{align}\label{hollowpart omega}
 \Omega_i^{\mathrm{holl}}\in \bigcap_{j=1}^r\left( \overline C_j^{-1} H^{-1}\mathfrak{so}(m)\right) 
\end{align}
and therefore by analogy with with \eqref{hollow part of alpha}, there exist matrices $B$ and $\widetilde B$ in $\mathfrak{so}(m)$ such that
\begin{equation}
\label{hollow part of omega}
\Omega_i^{\mathrm{holl}}= \overline C_1^{-1} H^{-1} B=\overline C_2^{-1} H^{-1} \widetilde B.
\end{equation}
Comparing entries in \eqref{hollow part of omega} and using skew-symmetricity of $B$ and $\widetilde B$, we get that in order to guarantee that $\Omega^{\mathrm{holl}}=0$ we should replace the condition in \eqref{generic assumption about C} by the condition that 
\begin{equation}
\label{generic assumption about C_1}
\begin{vmatrix}
\lambda_{1,i}&\lambda_{1,j}\\
\lambda_{2,i}&\lambda_{2,j}
\end{vmatrix}\neq 0, \quad \forall i\neq j.
\end{equation}
Therefore, for a generic tuple of diagonal matrices $\{C_j\}_{j=1}^r$, we get that $\Omega_i=s\mathrm{Id}\in \mathscr A$ and so the symbol is regular, contradicting our assumptions.

Now consider the remaining case, which is where $r=1$ and $H_\ell$ is positive definite.
Working with respect to a basis of $\mathfrak{g}_{-1,1}$ such that $H_\ell=\mathrm{Id}$, by \eqref{Azero is zero condition} and condition (iv) in \eqref{system}, we have
\begin{equation}
\label{condition 4 set equal to zero}
\sum_{s=1}^{r}\left(\overline{\mu_{1,1}^s}\Omega_s +\mu_{1,1}^s \Omega_1^*\right)= \left[\Omega_1^*,\Omega_1 \right]+C_1\overline{C_1}.
\end{equation}
By analogy with \eqref{diagpart} with $r=1$, taking into account that $\alpha=\Omega_i$ satisfies \eqref{secondalgebra} only, we have that
\begin{align}\label{diagpart omega}
 \Omega_1^{\mathrm{diag}}\in \mathrm{span}\{\mathrm{Id}\}.
\end{align}
Note also that \eqref{hollow part of omega} holds with $r=1$ and $H_\ell=\mathrm{Id}$. Fix $\mu\in\C$ and $B\in\mathfrak{so}(m)$ such that
\[
\Omega_1^{\mathrm{diag}}=\mu \mathrm{Id}
\quad\mbox{ and }\quad
\Omega_1^{\mathrm{holl}}=\overline{C_1}^{-1} B.
\]
Using the notation set in \eqref{Cdiag}, the $(i,j)$ element of $\left[\Omega_1^*,\Omega_1 \right]$ satisfies
\begin{align}
\left(\left[\Omega_1^*,\Omega_1 \right]\right)_{i,j} 
=\left(\left[\overline{B}{C_1^T}^{-1},\overline{C_1}^{-1} B \right]\right)_{i,j}
=\sum_{k=1}^m\left(\frac{1}{|\lambda_{1,k}|^2}-\frac{1}{\overline{\lambda_{1,i}}\lambda_{1,j}}\right)B_{k,i}\overline{B_{k,j}},
\end{align}
and, for $1\leq i\leq m$, by equating the $(i,i)$ elements of the matrices on each side of \eqref{condition 4 set equal to zero} we get
\begin{align}\label{condition 4 diag set equal to zero}
2\Re\mu=\left|\lambda_{1,i}\right|^2+\sum_{k=1}^m\left(\frac{1}{|\lambda_{1,k}|^2}-\frac{1}{\left|\lambda_{1,i}\right|^2}\right)\left|B_{k,i}\right|^2
\quad\quad\forall\, i\in\{1,\ldots, m\}.
\end{align}
Let $i_0$ be the index such that $|\lambda_{1,i_0}|=\max\{|\lambda_{1,1}|,\ldots,|\lambda_{1,m}|\}$. Accordingly, every term on the right side of \eqref{condition 4 diag set equal to zero} is nonnegative, and hence
\begin{align}\label{genericity lemma ineq1}
\left|\lambda_{1,i_0}\right|^2\leq 2\Re\mu.
\end{align}
On the other hand, taking the trace of both sides of \eqref{condition 4 set equal to zero} yields 
\begin{align}\label{genericity lemma ineq2}
2m\Re\mu=\sum_{i=1}^m\left|\lambda_{1,i}\right|^2<m\left|\lambda_{1,i_0}\right|^2,
\end{align}
where the strict inequality is obtained by imposing the assumption in \eqref{generic assumption about C}. Clearly \eqref{genericity lemma ineq1} and \eqref{genericity lemma ineq2} are incompatible, which means that, for the chosen $C_1$, no choice of $\Omega_1$ satisfies \eqref{system}.
\end{proof}

\begin{remark} 
\label{off-diagonal_rem}
The last arguments of the previous proof do not work in the case with $r=1$ and sign-indefinite $H_{\ell}$. There, additional analysis of the equations obtained by comparing  off-diagonal entries in the matrix equation given by condition (iv) of \eqref{system} is needed.  
\end{remark}

\section{Examples with nontrivial modified symbols}
\label{examples} 

In this section we describe three examples of CR structures to which Theorems \ref{maintheorm} and  \ref{maintheorconst} apply. All three examples are actually homogeneous CR manifolds, exhibiting  the maximally symmetric structures as described in Theorem \ref{maintheorconst}. These are instructive examples as they illustrate discrepancies between CR symbols and modified CR symbols, the reduction procedure of Section \ref{reduction_and_nongenericity_section}, and the existence of homogeneous structures with nonregular CR symbols. One may expect this latter phenomenon to be uncommon in light of Theorem \ref{gen_nonreg}, so it is valuable to have examples of it. For further examples, including a full implementation of the parallelism construction referred to in Theorem \ref{maintheorconst}, see Section \ref{Examples of the canonical parallelism construction}.

The first example is nonregular in the sense of Definition \ref{definition of regular}.
This section's second and third examples have the same CR symbol in the sense of \cite[Definition 2.2]{porter2017absolute} but different modified CR symbols, and, as is mentioned in the introduction, while the construction of an absolute parallelism given in \cite{porter2017absolute} is the same for both examples, the construction given here varies, resulting in parallelisms of different dimensions for each example  whose dimension matches that of the underlying CR manifold's symmetry group.

Each example in this section is a maximally symmetric homogeneous structure described in terms of a reduced modified CR symbol, and Theorem \ref{maintheorconst} implies that we can indeed describe such structures up to local equivalence by giving one of their reduced modified symbols as defined in \eqref{abstractred}. From a given reduced modified symbol $\mathfrak g^{0,\mathrm{red}}$, one can construct globally the homogeneous model  $(M_0,H^{\mathrm{flat}})$ exhibiting the flat CR structure with constant reduced modified symbol $\mathfrak g^{0,\mathrm{red}}$ as described in Section \ref{reduction_and_nongenericity_section}.

\begin{example}\label{7d nonregular}
Let $\mathfrak{g}_{-}$ be the five dimensional Heisenberg algebra with a basis $(e_0,\ldots,e_4)$ whose nonzero brackets are given by
\[
[e_1,e_4]=[e_2,e_3]=e_0.
\]
The basis $(e_1,\ldots, e_4)$ spans $\mathfrak{g}_{-1}$, and we define
\[
\mathfrak{g}_{-1,-1}:=\mathrm{span}_{\C}\{e_1, e_2\}
\quad\mbox{ and }\quad
\mathfrak{g}_{-1,1}:=\mathrm{span}_{\C}\{e_3, e_4\}.
\]
Representing elements of $\mathfrak{csp}(\mathfrak{g}_{-1})$ as matrices with respect to $(e_1,\ldots, e_4)$, we define $\mathfrak{g}_{0}^{\mathrm{red}}$ to be the subspace of $\mathfrak{csp}(\mathfrak{g}_{-1})$ spanned by the three matrices
\begin{equation}
\label{example1mat}
\left(
\begin{array}{cc:cc}
  0 & \frac{i}{\sqrt{2}} & 0 & i \\
  \frac{1}{\sqrt{2}} & 0 & 1 & 0 \\[5pt]\hdashline
  0 & 0 & 0 & \frac{-i}{\sqrt{2}} \rule{0pt}{2.6ex}\\
  0 & 0 & \frac{-1}{\sqrt{2}} & 0 \\
\end{array}
\right)
\quad\mbox{ and }\quad
\left(
\begin{array}{cc:cc}
  0 & \frac{i}{\sqrt{2}} & 0 & 0 \\
  \frac{-1}{\sqrt{2}} & 0 & 0 & 0 \\[5pt]\hdashline
  0 & -i & 0 & \frac{-i}{\sqrt{2}} \rule{0pt}{2.6ex}\\
  1 & 0 & \frac{1}{\sqrt{2}} & 0 \\
\end{array}
\right)
\end{equation}
and the $4\times 4$ identity matrix. With these definitions set, 
$\mathfrak{g}^{0,\mathrm{red}}=\mathfrak{g}_{-}\rtimes \mathfrak{g}_{0}^{\mathrm{red}}$ is a Lie algebra and it is an abstract reduced modified symbol of type $\mathfrak g^0$ in the sense of Definition \ref{abstractred}, where $\mathfrak{g}^0$ is the CR symbol with component $\mathfrak{g}_{0,2}$ generated by the matrix obtained  from the first matrix in \eqref{example1mat} by setting diagonal $2\times 2$ blocks equal to zero. In other words, as the matrix $C_1$ in \eqref{matrixg0pm2}, we can take 
\begin{equation}
\label{EG1 C1} C_1=
\left(
\begin{array}{cc}
     0&i  \\
     1&0 
\end{array}
\right).
\end{equation}
We also need to describe the antilinear involution $\sigma:\mathfrak{g}^{0,\mathrm{red}}\to \mathfrak{g}^{0,\mathrm{red}}$ associated with this model's CR structure. On $\mathfrak{g}_{-}$, the map $\sigma$ is defined by
\[
\sigma(e_0)=e_0,\quad \sigma(e_1)=e_3,\quad \sigma(e_2)=e_4,
\]
which uniquely defines an antilinear operator on $\mathfrak{g}_{-}$. We extend $\sigma$ to an antilinear involution on $\mathfrak{csp}(\mathfrak{g}_{-1})$ by the rule
\begin{align}\label{involution extension}
\sigma(\psi)(x)=\sigma\big(\psi\circ \sigma(x)\big).
\end{align}
By Remark \ref{regmatrixrem} the CR symbol $\mathfrak g^0$ is not regular, as $C_1\overline C_1 C_1 \notin \C C_1$.

Consider the flat CR structure $H^{\mathrm{flat}}$ defined on $M_0$ with  constant reduced modified symbol $\mathfrak{g}_{0}^{\mathrm{red}}$ as described in section \ref{reduction_and_nongenericity_section}. Let us now explicitly describe $P^0$, the modified CR symbols, and level sets of the mapping $\psi\mapsto \theta_0(T_\psi P^0)$ associated with this CR structure.

Let $G^{0,\mathrm{red}}$ be as described in Section \ref{reduction_and_nongenericity_section} and let $q:G^{0,\mathrm{red}}\to M_0^\C$ be the natural projection.
Let us first show that $G^{0,\mathrm{red}}$ can be naturally embedded into $P^0$.
The reduced modified symbol $\mathfrak{g}^{0,\mathrm{red}}$ is, 
by construction, canonically identified with the tangent space $T_eG^{0,\mathrm{red}}$ of $G^{0,\mathrm{red}}$ at the identity, and, more generally, via the differential of the left translation $L_h$ by an element 
$h\in G^{0,\mathrm{red}} $, we canonically identify $\mathfrak{g}^{0,\mathrm{red}}$ with the tangent space of $G^{0,\mathrm{red}}$ at $h$. In particular, for each $g\in G^{0,\mathrm{red}}$, these identifications can be restricted to give canonical isomorphisms $\psi_h:\mathfrak{g}_{-}\to {D}^{\mathrm{flat}}_{-}\big(q(h)\big)$, where $\psi_h=q_*\circ \left.(L_h)_*\right|_{\mathfrak{g}_{-}}$ and
\[
{D}^{\mathrm{flat}}_{-}\big(q(h)\big):=\bigoplus_{(i,j),\, i<0}{D}^{\mathrm{flat}}_{i,j}\big(q(h)\big).
\] 
The map $h\mapsto \psi_h$ defines the required embedding of $G^{0,\mathrm{red}}$ into $P^0$. In the sequel, we identify $G^{0,\mathrm{red}}$ with its image under this embedding.

Using \eqref{EG1 C1} together with $H_\ell=
\left(
\begin{array}{cc}
     0&1  \\
     1&0 
\end{array}
\right)$ and the identification in \eqref{SPinclusion1}, it is easy to show that $\mathfrak{g}_{0,0}$ is a $2$-dimensional subspace of $\mathfrak{csp}(\mathfrak{g}_{-})$ spanned by
\begin{align}\label{EG1 g_00 generators}
x_1=
\left(
\begin{array}{cc}
     \mathrm{Id}_2&0  \\
     0&-\mathrm{Id}_2 
\end{array}
\right)
\quad\mbox{ and }\quad 
x_2=\mathrm{Id}_4
\end{align}
where $\mathrm{Id}_k$ denotes the $k\times k$ identity matrix. The element $x_2$ in \eqref{EG1 g_00 generators} belongs to $\mathfrak{g}^{0,\mathrm{red}}$, and is actually the grading element referred to in the paragraph immediately following Theorem \ref{maintheorconst}. Counting dimensions, $G^{0,\mathrm{red}}$ has codimension 1 in $P^0$. 
Recall that $P^0$ is a $G_{0,0}$-bundle over $M_0$. Using the identification in  \eqref{SPinclusion1}, we consider the one-parametric subgroup $\mathrm{exp}(c x_1)\subset G_{0,0}$, so then $P^0$ can described by 
\begin{align}\label{EG1 frame bundle}
P^0= \{\psi\circ \mathrm{exp}(c x_1)\, |\,\psi\in G^{0,\mathrm{red}}, \, c\in\C\}.
\end{align}

For a given $\psi\in P^0$ such that 
\begin{equation}
\label{psi_c}
\psi=\psi_0 \circ \mathrm{exp}( c x_1),  \quad \psi_0\in G^{0, \mathrm{red}}  
\end{equation}
the degree zero component of the modified symbol $\mathfrak{g}_{0}^{\mathrm{mod}}(\psi)$ is spanned by the two matrices in \eqref{EG1 g_00 generators} together with the two matrices
\begin{align}\label{EG1 modified symbols}
\left(
\begin{array}{cccc}
  0 & \frac{i}{\sqrt{2}}e^{2c} & 0 & i \\
  \frac{1}{\sqrt{2}}e^{2c} & 0 & 1 & 0 \\
  0 & 0 & 0 & \frac{-i}{\sqrt{2}}e^{2c} \\
  0 & 0 & \frac{-1}{\sqrt{2}}e^{2c} & 0 \\
\end{array}
\right)
\quad\mbox{ and }\quad
\left(
\begin{array}{cccc}
  0 & \frac{i}{\sqrt{2}}e^{-2c} & 0 & 0 \\
  \frac{-1}{\sqrt{2}}e^{-2c} & 0 & 0 & 0 \\
  0 & -i & 0 & \frac{-i}{\sqrt{2}}e^{-2c} \\
  1 & 0 & \frac{1}{\sqrt{2}}e^{-2c} & 0 \\
\end{array}
\right).
\end{align}
Clearly, the level sets of the mapping $\psi\mapsto \theta_0(T_\psi P^0)$ are parameterized by the value $e^{2c}$ appearing in \eqref{EG1 modified symbols}. The image of each tangent space to one of these level sets under the soldering form $\theta_0$ is the space spanned by the matrices in \eqref{psi_c} together with the matrix $x_2$ in \eqref{EG1 g_00 generators}, which is the reduced modified symbol corresponding to that level set. The space $G^{0,\mathrm{red}}$ is a connected component of the level set corresponding to $e^{2c}=1$, which has two connected components. So, we started with an abstract reduced modified symbol $g^{0, \mathrm{red}}$ and we have shown that it is indeed the reduced modified symbol of  the level set $P^{0,\mathrm{red}}$ corresponding to $e^{2c}=1$. Consequently,  Theorem \ref{maintheorconst} can be applied  to the CR structure $H^{\mathrm{flat}}$ on $M_0$ to obtain that this homogeneous model's symmetry group has dimension equal to $\dim_\C \mathfrak{u}(\mathfrak{g}^{0,\mathrm{red}})=8$, where this formula follows from a direct calculation that $\mathfrak{u}(\mathfrak{g}^{0,\mathrm{red}})=\mathfrak{g}^{0,\mathrm{red}}$. By construction, in fact, $G^{0,\mathrm{red}}$ is the connected component of the symmetry group containing the identity.

Note that in the reduction of $P^0$ we can also use  other level sets of the mapping $\psi\mapsto \theta_0(T_\psi P^0)$ to obtain a different reduced modified symbol isomorphic to $\mathfrak{g}^{0,\mathrm{red}}$ from which we could build this same homogeneous model, but we have to make sure that the chosen level set has a nonempty intersection  with the real part $\Re P^0$ of the bundle $P^0$, which happens if and only if the space  $\theta_0(T_\psi P^0)$ is invariant under involution on $\mathfrak{csp}(\mathfrak{g}_{-1})$. The latter holds  if and only if $\Re c=0$, and hence $\Re P^0$ belongs to the subset of $\mathrm{pr}^{-1}(M_0)\subset P^0$ containing points at which the modified symbol is characterized by \eqref{EG1 modified symbols} with $\Re c=0$.
\end{example}

\begin{example}\label{9d regular not maximal for its symbol}

Let $\mathfrak{g}_{-}$ be the $7$-dimensional Heisenberg algebra with a basis $(e_0,\ldots,e_6)$ whose nonzero brackets are given by
\[
[e_1,e_6]=[e_2,e_5]=[e_3,e_4]=e_0.
\]
The basis $(e_1,\ldots, e_6)$ spans $\mathfrak{g}_{-1}$, and we define
\[
\mathfrak{g}_{-1,-1}:=\mathrm{span}_{\C}\{e_1, e_2, e_3\}
\quad\mbox{ and }\quad
\mathfrak{g}_{-1,1}:=\mathrm{span}_{\C}\{e_4, e_5,e_6\}.
\]
Representing elements of $\mathfrak{csp}(\mathfrak{g}_{-1})$ as matrices with respect to $(e_1,\ldots, e_6)$, we define $\mathfrak{g}_{0}^{\mathrm{red}}$ to be the subspace of $\mathfrak{csp}(\mathfrak{g}_{-1})$ spanned by
\begin{align}\label{EG2 adK and adKbar}
\left(
\begin{array}{ccc:ccc}
  0 & 1 & 0 & 0 & 1 & 0 \\
  0 & 0 & 0 & 0 & 0 & 1 \\
  0 & 0 & 0 & 0 & 0 & 0 \\\hdashline
  0 & 0 & 0 & 0 & 0 & 0 \rule{0pt}{2.2ex}\\
  0 & 0 & 0 & 0 & 0 & -1 \\
  0 & 0 & 0 & 0 & 0 & 0 \\
\end{array}
\right)
\quad \mbox{ and }\quad
\left(
\begin{array}{ccc:ccc}
    0 & 0 & 0 & 0 & 0 & 0  \\
    0 & 0 & -1 & 0 & 0 & 0  \\
    0 & 0 & 0 & 0 & 0 & 0  \\\hdashline
    0 & 1 & 0 & 0 & 1 & 0 \rule{0pt}{2.2ex}\\
    0 & 0 & 1 & 0 & 0 & 0  \\
    0 & 0 & 0 & 0 & 0 & 0  \\
\end{array}
\right),
\end{align}
together with
\begin{align}\label{EG2 zeroZeroReduction}
\mathfrak{g}_{0,0}^{\mathrm{red}}=\left\{
\left.
\left(
\begin{array}{cccccc}
    c_1+c_2 & 0 & c_4 & 0 & 0 & 0  \\
    0 & c_1 & 0 & 0 & 0 & 0  \\
    0 & 0 & c_1+c_3 & 0 & 0 & 0  \\
    0 & 0 & 0 & c_1-c_3 & 0 & -c_4 \\
    0 & 0 & 0 & 0 & c_1 & 0  \\
    0 & 0 & 0 & 0 & 0 & c_1-c_2  \\
\end{array}
\right)
\,\right|\,
c_i\in \C
\right\}.
\end{align}
For this example, we again consider the flat CR structure $H^{\mathrm{flat}}$ defined on $M_0$ with  constant reduced modified symbol $\mathfrak{g}_{0}^{\mathrm{red}}$ and associated Lie group $G^{0,\mathrm{red}}$ as described in section \ref{reduction_and_nongenericity_section}. Similar to the calculations for Example \ref{7d nonregular}, we calculate $\mathfrak{g}_{0,0}$ explicitly using \eqref{intersection algebra} with 
\[
C_1=
\left(
\begin{array}{ccc}
     0&1&0  \\
     0&0&1  \\
     0&0&0
\end{array}
\right)
\quad\mbox{ and }\quad
H_\ell=
\left(
\begin{array}{ccc}
     0&0&1  \\
     0&1&0  \\
     1&0&0
\end{array}
\right)
\]
to obtain that $\mathfrak{g}_{0,0}$ is spanned by matrices of the form in \eqref{EG2 zeroZeroReduction} together with
\begin{align}\label{EG2 g_00 generators}
x_1=\left(
\begin{array}{cc}
     \mathrm{Id}&0  \\
     0&-\mathrm{Id}
\end{array}
\right).
\end{align}
Note that, by Remark \ref{regmatrixrem}, the CR symbol of $H^{\mathrm{flat}}$ is regular at every point because $C_1\overline{C_1}C_1=0$. Now by using \eqref{EG2 g_00 generators} instead of \eqref{EG1 g_00 generators} the formulas in \eqref{EG1 frame bundle} and \eqref{psi_c} apply for our present example. In particular, for a point $\psi\in P^0$ satisfying \eqref{psi_c}, the degree zero component of the modified symbol $\mathfrak{g}_{0}^{\mathrm{mod}}(\psi)$ is spanned by the matrices in \eqref{EG2 adK and adKbar} and \eqref{EG2 g_00 generators} together with the two matrices
\begin{align}\label{EG2 modified generators}
\quad
\left(
\begin{array}{cccccc}
  0 & e^{2c} & 0 & 0 & 1 & 0 \\
  0 & 0 & 0 & 0 & 0 & 1 \\
  0 & 0 & 0 & 0 & 0 & 0  \\
  0 & 0 & 0 & 0 & 0 & 0\\
  0 & 0 & 0 & 0 & 0 & -e^{2c} \\
  0 & 0 & 0 & 0 & 0 & 0 \\
\end{array}
\right)
\quad \mbox{ and }\quad
\left(
\begin{array}{cccccc}
    0 & 0 & 0 & 0 & 0 & 0  \\
    0 & 0 & -e^{-2c} & 0 & 0 & 0  \\
    0 & 0 & 0 & 0 & 0 & 0  \\
    0 & 1 & 0 & 0 & e^{-2c} & 0 \\
    0 & 0 & 1 & 0 & 0 & 0  \\
    0 & 0 & 0 & 0 & 0 & 0  \\
\end{array}
\right).
\end{align}
As in Example \ref{7d nonregular}, the level sets of $P^0$ are parameterized by the values $e^{2c}$ appearing in \eqref{EG2 modified generators}, and the level sets having nontrivial intersection with $\Re P^0$ are those for which the corresponding parameter $e^{2c}$ satisfies $\Re(c)=0$.
The Tanaka prolongation $\mathfrak{u}(\mathfrak{g}^{0,\mathrm{red}})$ is 14-dimensional with a 1-dimensional positively graded component.
\end{example}

\begin{example}\label{9d regular maximal for its symbol}
This third example should be contrasted with Example \ref{9d regular not maximal for its symbol} and compared to the constructions in \cite{porter2017absolute}. For this example, let $\mathfrak{g}^0$ be the CR symbol of the structure in Example \ref{9d regular not maximal for its symbol} as characterized in \eqref{CRsymbol_2}, and set $\mathfrak{g}^{0,\mathrm{red}}=\mathfrak{g}^0$. Consider the flat structure $H^{\mathrm{flat}}$ on the homogeneous model $M_0$  constructed from $\mathfrak{g}^{0,\mathrm{red}}$ in Section \ref{reduction_and_nongenericity_section}. Note that since the CR symbol in Example \ref{9d regular not maximal for its symbol} is regular, $\mathfrak{g}_{0}^{\mathrm{red}}$ defined this way is indeed a Lie algebra. In contrast to Example \ref{9d regular not maximal for its symbol}, the model $(M_0,H^{\mathrm{flat}})$ has constant modified symbol, and, again in contrast to Example \ref{9d regular not maximal for its symbol}, the construction of the absolute parallelism for $(M_0,H^{\mathrm{flat}})$ given in this text is equivalent to the construction given in \cite{porter2017absolute}. The Tanaka prolongation $\mathfrak{u}(\mathfrak{g}^{0,\mathrm{red}})$ for this example is 16-dimensional with a 2-dimensional positively graded component. The prolongation $\mathfrak{u}(\mathfrak{g}^{0,\mathrm{red}})$ turns out to be equivalent to the bigraded prolongation introduced in \cite[Definition 2.2]{porter2017absolute} whose dimension is given in \cite[Theorem 5.3]{porter2017absolute}. This equivalence is not incidental, but rather a consequence of recoverability. That is, a CR structure is recoverable if and only if its associated bigraded prolongation is equivalent to the corresponding usual Tanaka prolongation of its CR symbol. As a consequence, just as happens with this example, the methods for constructing the absolute parallelism given here and in \cite{porter2017absolute} are equivalent for any recoverable structure having a constant modified symbol (which, by Theorem \ref{constant mod symbols implies regular}, implies that the modified symbol equals its CR symbol and its CR symbol is regular).
\end{example}

\section{Proof of Theorem \ref{maintheorm}}\label{Absolute parallelisms}
In this section, 
we modify methods from the theory of Noboru Tanaka's prolongation procedure described in \cite{zelenko2009tanaka} in order to prove Theorem \ref{maintheorm}. In particular, we describe the modifications necessary to obtain the bundles $\{P^{i}\}_{i=1}^\infty$ corresponding to \eqref{bundles}, which are required due to the non-constancy of $\mathfrak g_0^{\mathrm{mod}}(\psi)$. For structures with constant modified CR symbols, however, the standard Tanaka prolongation procedure can be applied directly without modification. The key modifications appear in the constructions of $P^1$ and $P^2$, and we construct these explicitly. Each higher degree prolongation $P^k$ is obtained from $P^{k-1}$ in the same way that $P^2$ is obtained from $P^{1}$.

\subsection{Constructing the first geometric prolongation}\label{A structure function on $P_0$} 

Let $\Pi_0:P^0\to M$ denote the natural projection. The contact structure $\mathcal{D}$ on $\mathcal{N}$ lifts to a filtration $D^0_0\subset D^{-1}_0\subset D^{-2}_0$ of $T P^0$ given by
\[
D^0_0=(\Pi_0)_*^{-1}(0),\quad D^{-1}_0=(\Pi_0)_*^{-1}(\mathcal{D}),\quad\mbox{ and }\quad D^{-2}_0=TP^0,
\]
and we use the notation
\[
D_0^{i}(\psi):=D_0^{i}\cap T_{\psi}P^0.
\]
Informally, our first goal in this section is to define \emph{structure functions} on $P^0$, which are elements in 
\begin{align}
\mathcal{S}_0=\Hom\left(\mathfrak{g}_{-1}\otimes \mathfrak{g}_{-2},\mathfrak{g}_{-2}\right)\oplus \Hom\left(\mathfrak{g}_{-1}\wedge \mathfrak{g}_{-1},\mathfrak{g}_{-1}\right)
\end{align}
associated with  horizontal subspaces in $TP^0$, and we will use this as a tool for finding a family of objects similar to Ehresmann connections on $P^0$ that is naturally associated with the underlying CR structure in a certain sense. But to be more precise, rather than associating horizontal subspaces with a structure function -- as is done in the theory of $G$-structures -- the structure functions we introduce will be associated with graded horizontal subspaces, namely pairs of subspaces of the form $\mathcal{H}_{\psi}=(H_{\psi}^{-2},H_{\psi}^{-1})$ where $\psi$ is a point in $P^0$, 
\begin{align}\label{horizontal subspace pair a}
H_{\psi}^{-2}\subset D_0^{-2}(\psi)/D_0^{0}(\psi),\quad H_{\psi}^{-1}\subset D_0^{-1}(\psi),
\end{align}
\begin{align}\label{horizontal subspace pair b}
D_0^{-2}(\psi)/D_0^{0}(\psi)=H_2\oplus D_0^{-1}(\psi)/D_0^{0}(\psi),\quad\mbox{ and }\quad D_0^{-1}(\psi)=H_{\psi}^{-1}\oplus D_0^0(\psi).
\end{align}

Notice that $D_0^{0}$ is the domain of the map $\theta_0:D_0^0\to \mathfrak{csp}(\mathfrak{g}_{-1})$ introduced in \eqref{degree zero soldering form section 4}. We call $\theta_0$ the degree zero soldering form on $P^0$ and introduce additional soldering forms $\theta_{-1}:D_0^{-1}\to\mathfrak{g}_{-1}$ and $\theta_{-2}:D_0^{-2}\to\mathfrak{g}_{-2}$ as follows. 
The projection $\Pi_0$ naturally induces a linear map from $D_0^{-2}(\psi)/D_0^{-1}(\psi)\oplus D_0^{-1}(\psi)/D_0^{0}(\psi)\oplus D_0^{0}(\psi)\to \mathcal{D}_{\Pi_0(\psi)}\oplus T_{\Pi_0(\psi)}\mathcal{N}/\mathcal{D}_{\Pi_0(\psi)}$. Using these induced maps and letting 
\[
\pi_0^{-2}:D_0^{-2}\to D_0^{-2}/D_0^{-1}\quad\mbox{ and }\quad \pi^{-1}_0:D_0^{-1}\to D_0^{-1}/D_0^{0}
\]
denote the natural projections, for a point  $\psi$ in $P_0$, we define
\[
\theta_{-1}(v):=\psi^{-1}\circ (\Pi_0)_*\circ \pi^{-1}_0 (v)\quad\quad\forall\, v\in D_0^{-1}(\psi),
\]
and
\[
\theta_{-2}(v):=\psi^{-1}\circ (\Pi_0)_*\circ \pi^{-2}_0 (v) \quad\quad\forall\, v\in D_0^{-2}(\psi).
\]
We have the associated maps $\overline{\theta}_{-2}:D_0^{-2}/D_0^{-1}\to\mathfrak{g}_{-2}$, $\overline{\theta}_{-1}:D_0^{-1}/D_0^{0}\to\mathfrak{g}_{-1}$, and $\overline{\theta}_{0}:D_0^{0}\to\mathfrak{csp}(\mathfrak{g}_{-1})$ given by
\[
\overline{\theta}_{i}\big(\pi_0^i(v)\big):=\theta_i(v)
\]
where $\pi_0^0:D_0^0\to D_0^0$ is taken to be the identity map.
\begin{remark}
These soldering forms $\{\theta_i\}$ are similar to the soldering forms introduced in \cite[Section 3]{zelenko2009tanaka}, but a subtle difference is that the space $\theta_0\left(D_0^0(\psi) \right)$ depends on $\psi$. For further reference, the forms $\theta_{-2}$ and $\theta_{-1}$ are close analogues of the forms labeled as $\theta^{(0)}_{-2}$ and $\theta^{(0)}_{-1}$ in \cite{tanaka1967generalized}.
\end{remark}
For a graded horizontal space $\mathcal{H}_\psi$ of the form satisfying \eqref{horizontal subspace pair a} and \eqref{horizontal subspace pair b}, we define  $\text{pr}_{-2}^{\mathcal{H}_\psi}:D_0^{-2}(\psi)/D_0^{0}(\psi)\to D_0^{-1}(\psi)/D_0^{0}(\psi)$ and $\text{pr}_{-1}^{\mathcal{H}_\psi}:D_0^{-1}(\psi)\to D_0^{0}(\psi)$ to be the projections parallel to the subspaces $H^{-2}_\psi$ and $H^{-1}_\psi$ respectively, and define the map
\[
\phi^{\mathcal{H}_\psi}\in \Hom\left(\mathfrak{g}_{-2},D_0^{-2}(\psi)/D_0^{0}(\psi)\right)\oplus \Hom\left(\mathfrak{g}_{-1},D_0^{-1}(\psi)\right)\oplus \Hom\left(\theta_0\left(D_0^{0}(\psi)\right),D_0^{0}(\psi)\right)
\]
by
\[
\phi^{\mathcal{H}_\psi}(v):=
\begin{cases}
\left((\Pi_0)_*\circ \left.\pi_0^{-2}\right|_{H^{-2}_\psi}\right)^{-1}\circ \psi(v) & \mbox{ if }v\in \mathfrak{g}_{-2}\\
\left((\Pi_0)_*\circ \left.\pi_0^{-1}\right|_{H^{-1}_\psi}\right)^{-1}\circ \psi(v) & \mbox{ if }v\in \mathfrak{g}_{-1}\\
\left(\left.\theta_0\right|_{D_0^0(\psi)}\right)^{-1}(v) & \mbox{ if }v\in\theta_0\left(D_0^{0}(\psi)\right).
\end{cases}
\]
We can now define the structure function $S_{\mathcal{H}_\psi}\in \mathcal{S}_0$ associated with ${\mathcal{H}_\psi}$ by the formula
\begin{align}\label{structure function a}
S_{\mathcal{H}_\psi}(v_1,v_2):=
\begin{cases}
\overline{\theta}_{-2}\left([Y_1,Y_2](\psi)+D_{0}^{-1}(\psi)\right)&\mbox{ if }v_2\in \mathfrak{g}_{-2}\\
\overline{\theta}_{-1}\left(\text{pr}_{-2}^{\mathcal{H_\psi}}\left([Y_1,Y_2](\psi)+D_{0}^{0}(\psi)\right)\right)&\mbox{ if }v_2\in \mathfrak{g}_{-1}
\end{cases}
\end{align} 
where $Y_1$ and $Y_2$ are vector fields defined on a neighborhood of $\psi$ in $P^0$ such that, supposing $v_2\in\mathfrak{g}_{i}$ for $i\in\{-1,-2\}$, 
\begin{align}\label{structure function b}
Y_1\in\Gamma(D_0^{-1}), \quad Y_2\in\Gamma(D_0^{i}),\quad \theta_{-1}(Y_1)=v_1,\quad \theta_{i}(Y_2)=v_2,\quad (Y_1)_{\psi}=\phi^{\mathcal{H}_{\psi}}(v_1),
\end{align}
and either
\begin{align}\label{structure function c}
\left\{\parbox{8.5cm}{(a) $\quad i=-1$ and $(Y_2)_{\psi}=\phi^{\mathcal{H}_{\psi}}(v_2)$, or \\
(b) $\quad i=-2$ and $(Y_2)_{\psi}\equiv\phi^{\mathcal{H}_{\psi}}(v_2)\pmod{D_0^0(\psi)}$.}\right.
\end{align}
The definition of $S_{\mathcal{H}_\psi}$ given in \eqref{structure function a}, \eqref{structure function b}, and \eqref{structure function c} coincides with the definition in \cite[(3.5)]{zelenko2009tanaka}, wherein the following lemma is proven.

\begin{lemma}[proven in {\cite[Section 3]{zelenko2009tanaka}}]
The definition of $S_{\mathcal{H}_\psi}$ given in \eqref{structure function a} does not depend on the choice of vector fields $Y_1$ and $Y_2$ satisfying \eqref{structure function b} and \eqref{structure function c}.
\end{lemma}

Considering another graded horizontal space $\widetilde{\mathcal{H}}_\psi$, let us describe the difference between the structure functions $S_{\mathcal{H}_{\psi}}$ and $S_{\widetilde{\mathcal{H}}_{\psi}}$. For this we introduce the function 
\[
f_{\mathcal{H}_{\psi} \widetilde{\mathcal{H}}_{\psi}}\in \Hom(\mathfrak{g}_{-2},\mathfrak{g}_{-1})\oplus\Hom\left(\mathfrak{g}_{-1},\theta_0\left(D_0^{0}(\psi)\right)\right)
\]
defined by
\[
f_{\mathcal{H}_{\psi} \widetilde{\mathcal{H}}_{\psi}}(v)=
\overline{\theta}_{i+1}\left(\phi^{\mathcal{H}_\psi}(v)-\phi^{\widetilde{\mathcal{H}}_\psi}(v)\right)\quad\quad\forall\, v\in \mathfrak{g_i}\mbox{ and } i\in\{-1,-2\},
\]
and introduce the \emph{anti-symmetrization (or generalized Spencer) operator} 
\[
\partial_{\psi}^0:\Hom(\mathfrak{g}_{-2},\mathfrak{g}_{-1})\oplus\Hom\left(\mathfrak{g}_{-1},\theta_0\left(D_0^{0}(\psi)\right)\right)\to \mathcal{S}_0
\]
defined by
\[
\partial_{\psi}^0 f(v_1,v_2):=[f(v_1),v_2]+[v_1,f(v_2)]-f([v_1,v_2]),
\]
where the brackets $[\cdot, \cdot]$ are defined by the Lie algebra structure on $\mathfrak{g}_{-}\rtimes \mathfrak{csp}(\mathfrak{g}_{-1})$. It is shown in \cite[Proposition 3.1]{zelenko2009tanaka} that
\begin{align}\label{structure function transformation rule}
S_{\mathcal{H}_{\psi}}=S_{\widetilde{\mathcal{H}}_{\psi}}+\partial_\psi^0 f_{\mathcal{H}_{\psi} \widetilde{\mathcal{H}}_{\psi}}.
\end{align}

In standard Tanaka theory, one defines the so-called first geometric prolongation of $P^0$, which is a certain fiber bundle $P^1$ defined over the base space $P^0$, but here instead we define an analogous first prolongation as a  bundle $P^1$ over a neighborhood $\mathcal{O}^0$ in $P^0$. For defining this, let $\psi_0\in \Re P^0$ be as in the item (1) of Theorem \ref{maintheorm}. By regularity of $\psi_0$ there exists an open neighborhood $\mathcal{O}^0\subset P^0$  of $\psi_0$ such that there exists a subspace $N_0\subset \mathcal{S}_0$ for which
\begin{align}\label{gzero normalization b}
\mathcal{S}_0=N_0\oplus \text{Im}\partial_{ \psi}^0 \quad\quad\forall\, \psi\in \mathcal{O}^0.
\end{align}
Moreover, the natural involutions on each previously defined $\mathfrak g_i$ induce  the natural involution on the space $\mathcal S_0$ and also $\mathrm{Im}\partial_{ \psi}^0$ is invariant under this involution for  $\psi$ belong to $\Re\mathcal{O}^0:=\Re P^0\cap \mathcal O^0$, based on the rule that the involution of the tensor product of two elements is the tensor product of the involution of these elements. So,  we can take $N_0$ to be invariant with respect to the involution.

The subspace $ N_0$ is called the \emph {normalization condition of the structure function} for the first prolongation and the choice of $N_0$ defines the bundle $P^1$ via the formula
\begin{align}\label{P1 definition via horizontals}
P^1:=
\left\{\mathcal{H}_{\psi}\,\left|\,\parbox{6cm}{$\psi\in\mathcal{O}^0$ and $\mathcal{H}_{\psi}$ is a pair of horizontal spaces in $T_{\psi}\mathcal{O}^0$ as described in \eqref{horizontal subspace pair a} and \eqref{horizontal subspace pair b} such that $S_{\mathcal{H}_{\psi}}\in N_0$}\right.\right\},
\end{align}
or, equivalently,
\begin{align}\label{P1 definition via maps}
P^1:=
\left\{\phi^{\mathcal{H}_{\psi}}\,\left|\,\parbox{6cm}{$\psi\in\mathcal{O}^0$ and $\mathcal{H}_{\psi}$ is a pair of horizontal spaces in $T_{\psi}\mathcal{O}^0$ as described in \eqref{horizontal subspace pair a} and \eqref{horizontal subspace pair b} such that $S_{\mathcal{H}_{\psi}}\in N_0$}\right.\right\}.
\end{align}
Since $N_0$ is invariant with respect to the involution, we have that if $\phi^{\mathcal{H}_{\psi}}\in P_1$ then $\overline{\phi^{\mathcal{H}_{\psi}}}\in P_1 $ for all $\psi\in \Re \mathcal O^0$, where
\[
\overline{\phi^{\mathcal{H}_{\psi}}}(v):=\overline{\phi^{\mathcal{H}_{\psi}}(\overline{v})}
\]
according to the rule of commuting the involution with tensor products. Hence a natural involution is defined on the fibers of $P^1$ over $\Re O^0$, and the fixed point set of this induced involution is a subspace in $P^1$ that we denote by $\Re P^1$  and call the real part of $P^1$.

By \eqref{structure function transformation rule}, if $\mathcal{H}_\psi$ and $\widetilde{\mathcal{H}}_\psi$ are two elements of $P^1$ belonging to the fiber $(P^1)_\psi$ of $P^1$ over the point $\psi\in \mathcal{O}^0$, then $\partial_{\psi}^0f_{H_\psi \widetilde{\mathcal{H}}_\psi}=0$, and hence
\[
f_{H_\psi \widetilde{\mathcal{H}}_\psi}\in\ker\partial_{\psi}^0=\mathfrak{g}_1^{\mathrm{mod}}(\psi).
\]
Conversely, if $f\in\ker \partial_{\psi}^0$ and $ \mathcal{H}_\psi$ is in the fiber $(P^1)_\psi$ of $P^1$ over $\psi$ then the graded horizontal space
\[
\{(v,w)+f(v,w)\,|\,(v,w)\in \mathcal{H}_\psi\}
\]
also belongs to $(P^1)_\psi$. In other words, by the regularity of $\psi_0$, $P^1$ is an affine bundle modeled on $\mathfrak{g}_1^{\mathrm{mod}}(\psi_0)$, and each tangent space $T_{\mathcal{H}_\psi}P^1$ is naturally identified with $\mathfrak{g}_1^{\mathrm{mod}}(\psi)$ by the map $\theta_1^{(1)}:T_{\mathcal{H}_\psi}(P^1)_\psi\to \mathfrak{g}_1^{\mathrm{mod}}(\psi)$ defined by the formula
\begin{align}\label{degree 1,1 soldering form}
X=\left.\frac{d}{dt}\right|_{t=0}\left\{\left.(v,w)+t\theta_1^{(1)}(X)(v,w)\,\right|\,(v,w)\in \mathcal{H}_\psi\right\}.
\end{align}
Similarly, $\Re P^1$ is an affine bundle over $\Re \mathcal{O}^0$ modeled on $\Re \mathfrak{g}_1^{\mathrm{mod}}(\psi_0)$.

The difference between $P^1$ defined here and the first geometric prolongation defined in the standard Tanaka theory as in \cite{zelenko2009tanaka} is that the maps $\phi^{\mathcal{H}_{\psi}}$ appearing in \eqref{P1 definition via maps} have different domains because $\mathfrak g_0^{\mathrm{mod}}(\psi)$ is non-constant. If $\mathfrak{g}_{1}^{\mathrm{mod}}(\psi_0)\neq0$ then in order to continue the prolongation procedure, constructing higher degree geometric prolongations, we need to somehow identify the domains of each $\phi^{\mathcal{H}_{\psi}}$. Moreover, independent of $\mathfrak{g}_1^{\mathrm{mod}}(\psi_0)$, ultimately we still will need to fix an identification of all $\mathfrak{g}_0^{\mathrm{mod}}(\psi)$ in order to construct canonical absolute parallelisms. All of this motivates the following introduction of what we call the \emph{identification space} $\mathcal I_0$. By regularity of $\psi_0$,  we can fix a subspace $\mathcal I_0\subset \mathfrak{csp}(\mathfrak{g}_{-1})$ invariant under the induced involution on $\mathfrak{csp}(\mathfrak{g}_{-1})$, such that, after possibly shrinking the neighborhood $\mathcal{O}^0$, in addition to \eqref{gzero normalization b}, we have
\begin{align}\label{gzero normalization a}
\mathfrak{csp}(\mathfrak{g}_{-1})=\mathcal I_0\oplus \mathfrak{g}_0^{\mathrm{mod}}(\psi) \quad\quad\forall\, \psi\in \mathcal{O}^0.
\end{align}
For all $\psi$ in $\mathcal{O}^0$, each $\mathfrak{g}_0^{\mathrm{mod}}(\psi)$ is identified with $\mathfrak{g}_0^{\mathrm{mod}}(\psi_0)$ via the projection to the latter that is parallel to $\mathcal I_0$. We let 
\[
\mathrm{pr}^{\mathcal{I}_0}:\mathfrak{g}_{-}\oplus \mathfrak{csp}(\mathfrak{g}_{-1})\to\mathfrak{g}^{0,\mathrm{mod}}(\psi_0)
\]
denote the map that is equal to the identity on $\mathfrak g_-$ and equal to the projection parallel to $\mathcal{I}_0$ on $\mathfrak{csp}(\mathfrak{g}_{-1})$.

\subsection{Constructing the second geometric prolongation}
We define the second geometric prolongation as a bundle over a neighborhood $\mathcal{O}^1$ in $P^1$ just as we defined $P^1$ as a bundle over the neighborhood $\mathcal{O}^0$ in $P^0$. For this we now introduce structure functions associated with graded horizontal spaces in $T\mathcal{O}^1$ and define $P^2$ to be the bundle of these graded horizontal spaces whose structure functions satisfy a certain normalization condition.

The filtration $D_0^0\subset D_0^{-1}\subset D_0^{-2}$ of $TP^0$ lifts to a filtration $D_1^1\subset D_1^0\subset D_1^{-1}\subset D_1^{-2}$ of $TP^1$, where, for $i\in\{0,-1,-2\}$, $D_{1}^{i}=(\Pi_1)_*^{-1}D_0^{i}$, and $D_1^1=(\Pi_1)_*^{-1}(0)$. We set $D_1^i(\mathcal{H}_\psi)=D_1^i\cap T_{\mathcal{H}_\psi}P^1$. Using the definition of $P^1$ given in \eqref{P1 definition via horizontals}, for each $\mathcal{H}_{\psi}\in P^1$ and $i\in\{-2,-1,0,1\}$, we also have soldering forms 
\[
\theta_{i}^{(1)}:D_1^i(\mathcal{H}_\psi)\to\mathfrak{g}^{0,\mathrm{mod}}(\psi)\oplus \mathfrak{g}_1^{\mathrm{mod}}(\psi)
\]
 where $\theta_{1}^{(1)}$ is as given in \eqref{degree 1,1 soldering form}, and, for $i<1$, $\theta_{i}^{(1)}(v)=\theta_{i}\circ (\Pi_1)_*(v)$. Each $\theta_{i}^{(1)}$ has the corresponding map $\overline{\theta}_{i}^{(1)}$ with domain $D_1^i(\mathcal{H}_\psi)/D_1^{i+1}(\mathcal{H}_\psi)$ defined by $\overline{\theta}_{1}^{(1)}=\theta_{1}^{(1)}$ and $\overline{\theta}_{i}^{(1)}=\theta_{i}\circ (\Pi_1)_*$.

Similar to the definition of graded horizonal subspaces in $T P^0$, for $p\in P^1$, we define graded horizontal subspaces in $T_p P^1$ as tuples of subspaces $\mathcal{H}_{p}=(H^{0}_{p},H^{-1}_{p},H^{-2}_{p})$ such that $H^{i}_p\subset D_1^{i}(p)$, $H^{0}_p\oplus D_1^1(p)= D_1^{0}(p)$,  $H^{-1}_p\oplus D_1^0(p) = D_1^{-1}(p)$, and $H^{-2}_p/ D_1^1(p)\oplus  D_1^{-1}(p)/ D_1^1(p)= D_1^{-2}(p)/ D_1^1(p)$. For each of these graded horizontal subspaces $\mathcal{H}_{p}$ in $T_pP^1$, we define  $\text{pr}_{-2}^{\mathcal{H}_p}:D_1^{-2}(p)/D_1^{1}(p)\to D_1^{0}(p)/D_1^{1}(p)$ and $\text{pr}_{-1}^{\mathcal{H}_p}:D_1^{-1}(p)\to D_1^{1}(p)$ to be the projections parallel to the subspaces $H^{-2}_p$ and  $H^{-1}_p$ respectively. Analogous to the map defined in \eqref{P1 definition via horizontals}, each graded horizontal subspace $\mathcal{H}_{p}$ in $T_pP^1$ uniquely determines an isomorphism 
\begin{align}\label{horizontal inverse of soldering forms}
    \phi^{\mathcal{H}_{p}}:\mathfrak{g}_{-}\oplus \mathfrak{g}_{0}^{\mathrm{mod}}\big(\Pi_1(p)\big)\oplus \mathfrak{g}_{1}^{\mathrm{mod}}\big(\Pi_1(p)\big)\to H^{-2}_p\oplus H^{-1}_p\oplus H^{0}_p\oplus D_{1}^1(p)
\end{align}
such that, for $i\in\{-1,-2\}$, $\phi^{\mathcal{H}_{p}}(\mathfrak{g}_{i})=H^{i}_p$, $\phi^{\mathcal{H}_{p}}\big(\mathfrak{g}_{0}^{\mathrm{mod}}(p)\big)=H^{0}_p$, and $\phi^{\mathcal{H}_{p}}\big(\mathfrak{g}_{1}^{\mathrm{mod}}(p)\big)=D_{1}^1(p)$. For a graded horizontal subspace $\mathcal{H}_{p}\subset T_pP^1$ we define its structure function $S_{\mathcal{H}_{p}}$ to be the element of
\[
\mathcal{S}_1:=\Hom\left(\mathfrak{g}_{-1}\otimes \mathfrak{g}_{-2},\mathfrak{g}_{-1}\right)\oplus \Hom\left(\mathfrak{g}_{-1}\wedge \mathfrak{g}_{-1},\mathfrak{g}_{0}^{\mathrm{mod}}(\psi_0)\right)\oplus \Hom\left(\mathfrak{g}_{-1}\otimes \mathfrak{g}_{0}^{\mathrm{mod}}(\psi_0),\mathfrak{g}_{0}^{\mathrm{mod}}(\psi_0)\right)
\]
defined by
\begin{align}\label{structure function degree 1 a}
S_{\mathcal{H}_p}(v_1,v_2):=
\begin{cases}
\overline{\theta}_{-1}^{(1)}\left(\text{pr}_{-2}^{\mathcal{H}_p}[Y_1,Y_2](p)+D_{1}^{0}(p)\right)&\mbox{ if }v_2\in \mathfrak{g}_{-2}\\
\text{pr}^{\mathcal{I}_0}\circ\overline{\theta}_{0}^{(1)}\left(\text{pr}_{-1}^{\mathcal{H}_p}\left([Y_1,Y_2](p)+D_{1}^{1}(p)\right)\right)&\mbox{ if }v_2\in \mathfrak{g}_{-1}\\
\text{pr}^{\mathcal{I}_0}\circ\overline{\theta}_{0}^{(1)}\left([Y_1,Y_2](p)\right)&\mbox{ if }v_2\in \mathfrak{g}_{0}^{\mathrm{mod}}(\psi_0)
\end{cases}
\end{align} 
where $Y_1$ and $Y_2$ are vector fields defined on a neighborhood of $p$ in $P^1$ such that, supposing $v_2\in\mathfrak{g}_{i}(\psi_0)$ for $i\in\{0,-1,-2\}$, 
\begin{align}\label{structure function degree 1 b}
Y_1\in\Gamma(D_1^{-1}), \quad Y_2\in\Gamma(D_1^{i}),\quad \theta_{-1}^{(1)}(Y_1)=v_1,\quad \theta_{i}^{(1)}(Y_2)=\left(\left.\text{pr}^{\mathcal{I}_0}\right|_{\mathfrak{g}^{0,\mathrm{mod}}(\Pi_1(p))}\right)^{-1} (v_2),
\end{align}
\begin{align}
(Y_1)_{p}=\phi^{\mathcal{H}_{p}}(v_1),
\end{align}
and either
\begin{align}\label{structure function degree 1 c}
\left\{\parbox{9.8cm}{(a) $\quad i=0$ and $(Y_2)_{p}=\phi^{\mathcal{H}_{p}}\circ \left(\left.\text{pr}^{\mathcal{I}_0}\right|_{\mathfrak{g}^{0,\mathrm{mod}}(\Pi_1(p))}\right)^{-1} (v_2)$, \\
(b) $\quad i=-1$ and $(Y_2)_{p}\equiv\phi^{\mathcal{H}_{p}}(v_2)\pmod{D_1^1(\psi)}$, or\\
(c) $\quad i=-2$ and $(Y_2)_{p}\equiv\phi^{\mathcal{H}_{p}}(v_2)\pmod{D_1^0(\psi)}$.}\right.
\end{align}
Comparing this formula for $\mathcal{S}_{\mathcal{H}_p}$ to the structure functions defined (for geometric prolongations of arbitrary degree) in \cite[Formula (4.16)]{zelenko2009tanaka}, the only difference is that we include the projection $\mathrm{pr}^{\mathcal{I}_0}$ in multiple places, and this modification is necessary because the symbols $\mathfrak{g}_{0}^{\mathrm{mod}}(\psi)$ are non-constant. Notice that if the structure's modified CR symbols are constant on $\mathcal{O}^{0}$  then the formulas in \eqref{structure function degree 1 a} and \eqref{structure function degree 1 c} would be unaffected by the removal of  $\mathrm{pr}^{\mathcal{I}_0}$.

For a point $p\in P^1$, we introduce another anti-symmetrization operator 
\[
\partial^1_{p}:\mathrm{Hom}\big(\mathfrak{g}_{-2},\mathfrak{g}_{0}^{\mathrm{mod}}(\Pi_1(p))\big)\oplus \mathrm{Hom}\big(\mathfrak{g}_{-1},\mathfrak{g}_{1}^{\mathrm{mod}}(\Pi_1(p))\big)
\oplus\mathrm{Hom}\big(\mathfrak{g}_{0}^{\mathrm{mod}}(\Pi_1(p)),\mathfrak{g}_{1}^{\mathrm{mod}}(\Pi_1(p))\big)
\to \mathcal{S}_1
\]
defined by
\begin{align}\label{anti-symmetrization for P1}
\quad\quad
\partial^1_{p}f(v_1,v_2)=
\begin{cases}
\mathrm{pr}^{\mathcal{I}_0}\left(\left[f\circ \iota \left(v_1\right),v_2\right]+\left[v_1,f\circ \iota\left(v_2\right)\right]-f\circ \iota\left(\left[v_1,v_2\right]\right)\right) & \mbox{ if }v_2\in \mathfrak{g}_{-2},\mathfrak{g}_{-1} \\
\mathrm{pr}^{\mathcal{I}_0}\left(\left[v_1,f\circ \iota \left(v_2\right)\right]\right)& \mbox{ if }v_2\in \mathfrak{g}_{0}^{\mathrm{mod}}(\psi_0)
\end{cases}
\end{align}
using the identification $\iota=\left(\left.\text{pr}^{\mathcal{I}_0}\right|_{\mathfrak{g}^{0,\mathrm{mod}}(\Pi_1(p))}\right)^{-1}$  for brevity. Note that this definition of $\partial_p^1$ is similar to the generalized Spencer operator defined for the second geometric prolongation in \cite{zelenko2009tanaka}, and the key difference is that our definition of $\partial_p^1$ here includes intertwining with the projection $\mathrm{pr}^{\mathcal{I}_0}$.

Similar to the construction of the first geometric prolongation, by regularity of $\psi_0$ there exists an open neighborhood $\mathcal{O}^1\subset P^1$  with $\psi_0\in \Pi_1(\mathcal{O}^1)$ such that there exists a subspace $N_1\subset \mathcal{S}_1$ for which
\begin{align}\label{gOne normalization b}
\mathcal{S}_1=N_1\oplus \text{Im}\partial_{ \psi}^0 \quad\quad\forall\, \psi\in \mathcal{O}^0.
\end{align}
We can take it to be invariant with respect to natural involution induced on $\mathcal S_1$.
We call $ N_1$ the \emph{normalization condition of the structure function} for the first prolongation and the choice of $N_1$ defines a second geometric prolongation $P^2$ via the formula
\begin{align}\label{P2 definition via horizontals}
P^2:=
\left\{\mathcal{H}_{p}\,\left|\,\parbox{6cm}{$p\in\mathcal{O}^1$ and $\mathcal{H}_{p}$ is a pair of horizontal spaces in $T_{p}\mathcal{O}^1$ as described in \eqref{horizontal subspace pair a} and \eqref{horizontal subspace pair b} such that $S_{\mathcal{H}_{\psi}}\in N_1$}\right.\right\}.
\end{align}
Just as $P^1$ has the structure of an affine bundle modeled on the vector space $\mathfrak{g}_1^{\mathrm{mod}}(\psi_0)$, the bundle $P^2$ has the structure of an affine bundle over $\mathcal{O}^1$ modeled on $\mathfrak{g}_2^{\mathrm{mod}}(\psi_0)$.

Finally, by complete analogy with the first prolongation,  we can define the real part $\Re P^2$ of $P^2$ as an affine bundle over $\Re \mathcal O^1:=\mathcal{O}^1\cap \Re P^1$ modeled on the space $\Re \mathfrak{g}_2^{\mathrm{mod}}(\psi_0)$.

\subsection{Higher geometric prolongations}
To summarize how the above constructions of $P^1$ and $P^2$ differ from the geometric prolongations in \cite{zelenko2009tanaka}, each bundle $P^i$ is defined over a neighborhood in $P^{i-1}$ and the maps defined in \eqref{structure function degree 1 a} and \eqref{anti-symmetrization for P1} differ from their analogs in \cite{zelenko2009tanaka} in that they are intertwined with the projection $\mathrm{pr}^{\mathcal{I}_0}$. This exact pattern continues for the construction of each higher geometric prolongation $P^i$ with $i>2$. In particular, for example, letting $\mathfrak{u}_{1}$ denote the standard first Tanaka prolongation of the graded Lie algebra $\mathfrak{g}_{-}\oplus \mathfrak{csp}(\mathfrak{g}_{-1})$ (defined using the same formula given in \eqref{mgk 1} with $\mathfrak{g}_{0}^{\mathrm{mod}}(\psi)$ replaced by $\mathfrak{csp}(\mathfrak{g}_{-1})$) and again using the regularity of $\psi_0$, we can shrink the neighborhood $\mathcal{O}^1$ so that, in addition to \eqref{gOne normalization b}, there exists a subspace $\mathcal I_1$ in $\mathfrak{u}_{1}$ that is invariant under the induced involution and satisfies
\begin{align}\label{gOne normalization a}
\mathfrak{u}_{1}=\mathcal{I}_1\oplus\mathfrak{g}_{1}^{\mathrm{mod}}(\psi) \quad\quad\forall\, \psi\in \mathcal{O}^0.
\end{align}
Each $\mathfrak{g}^{0,\mathrm{mod}}(\psi)\oplus \mathfrak{g}_1^{\mathrm{mod}}(\psi)$ is identified with $\mathfrak{g}^{0,\mathrm{mod}}(\psi_0)\oplus \mathfrak{g}_1^{\mathrm{mod}}(\psi_0)$ via the projection $\mathrm{pr}^{\mathcal{I}_1}$ to the latter that is defined on $\mathfrak{g}^{0,\mathrm{mod}}(\psi)$ as $\mathrm{pr}^{\mathcal{I}_{0}}$ and is defined on $\mathfrak{u}_{1}$ as the projection onto $\mathfrak{g}_1^{\mathrm{mod}}(\psi_0)$ parallel to $\mathcal I_1$. Next, the third geometric prolongation $P^3$ can be constructed over a neighborhood $\mathcal{O}^2$ in $P^2$ using the construction in \cite{zelenko2009tanaka} with modification that the structure functions and generalized Spencer operators must be intertwined with the projection $\mathrm{pr}^{\mathcal{I}_1}$ just as the maps in \eqref{structure function degree 1 a} and \eqref{anti-symmetrization for P1} are intertwined with $\mathrm{pr}^{\mathcal{I}_0}$. Repeating the process with these modifications give a microlocal version of the Tanaka prolongation procedure, as it were.

The properties of the geometric prolongations defined in classical Tanaka theory that we are interested in remain unaffected by the modifications to the prolongation procedure made above. In particular,
\begin{enumerate}[label=(\alph*)]
    \item for each $i>0$, the space $P^i$ has the structure of an affine bundle over $\mathcal{O}^{i-1}$ modeled on the vector space $\mathfrak{g}_i^{\mathrm{mod}}(\psi_0)$,
    \item for each $i>0$, $P^i$ and $\mathcal{O}^{i-1}$ has a natural induced involution defined on it, and by restricting to the fixed point sets of these involutions one obtains the space $\Re P^i$ defined as a fiber bundle over $\Re \mathcal{O}^{i-1}$ modeled on the vector space $\Re  \mathfrak{g}_i^{\mathrm{mod}}(\psi_0)$,
    \item if $l+1$ is the smallest number for which $\mathfrak{g}_{l+1}^{\mathrm{mod}}(\psi_0)=0$ then the $l+2$ normalization conditions $N_{0},\ldots, N_{l+1}$ chosen in the first $l+2$ steps of the prolongation procedure determine a canonical absolute parallelism both on $\mathcal{O}^l$ and $\Re\mathcal{O}^l$, 
    \item the pseudogroup of local symmetries of the underlying CR manifold has a naturally induced partial action on each $\mathcal{O}^i$ and $\Re \mathcal{O}^i$ under which the parallelism mentioned in the last item is invariant.
\end{enumerate}

Item (c) completes the proof of item (1) in Theorem \ref{maintheorm}. Since the canonical frame on $\Re\mathcal{O}^l$ referred to in item (c) is invariant under the action of the pseudogroup of local symmetries, each symmetry is uniquely determined locally near a point by its value at that point, and therefore the dimension of this pseudogroup is not greater than the real dimension of the bundle $\Re \mathcal{O}^l$, which establishes item (2) in Theorem \ref{maintheorm}. By items (a) and (b)  above, this dimension is equal to $\dim_\C \mathfrak u(\mathfrak g^{0, \mathrm{mod}})$ as desired.

\section{Examples of the canonical parallelism construction}\label{Examples of the canonical parallelism construction}

Here we apply this paper's main results to solve the local equivalence problem for a family of hypersurfaces in $\mathbb{C}^6$, showing in particular that the family contains infinitely many non-equivalent structures sharing the same constant CR symbol (see Theorem \ref{eg8.4 equivalence solution 2} and Remark \ref{eg8.4 equivalence solution 2 remark}). This family of examples thus shows that there is indeed rich variety among the geometric structures in the local equivalence problems to which this paper's main results apply. Through these examples we demonstrate a full implementation of the canonical parallelism construction detailed in Section \ref{Absolute parallelisms}.

We will use standard real and complex coordinates $z=(z_0,\ldots, z_5)=(x_0+iy_0,\ldots, x_5+iy_5)$ of $\mathbb{C}^6$. For a real analytic function $f:U\to\R$ defined on some open subset $U\subset \mathbb{R}$, set
\begin{align}\label{C6 example defining function}
F_f(z):=-x_0+f(x_1)+2x_1x_4+2x_2x_3+(2x_1x_3+x_2^2)x_5+2x_1x_2x_5^2+x_1^2x_5^3,
\end{align}
and let $M_f$ denote the hypersurface in $\mathbb{C}^6$ defined by
\begin{align}\label{C6 example family}
M_f:=\{z\in \C^6\,|\, F_f(z)=0\}.
\end{align}
Such structures are special cases from a more general class of $2$-nondegenerate hypersurfaces studied in the forthcoming text \cite{kolar2022NewExamples}.

We will now apply this paper's main results to determine necessary and sufficient conditions for two hypersurfaces of the form in \eqref{C6 example family} to be locally equivalent near given points in the respective manifolds. For notational convenience, let us fix some arbitrary $f$ (with $f(0)=0$) and for the remainder of this example let $M$  and $F$ denote the corresponding objects $M_f$ and $F_f$.

The Levi form of $M$ at a point $z\in M$ is represented by the complex Hessian of $F$, (i.e., the matrix whose $(j,k)$ entry is $\tfrac{\partial^2F}{\partial\overline{ z_j} \partial z_k}$), which can be used to easily verify that $M$ has a rank $1$ Levi kernel. The vector fields 
\[
X_j:=\frac{\partial}{\partial z_j} +\frac{\partial F}{\partial z_j}\frac{\partial}{\partial z_0} \quad\quad\forall \,1\leq j\leq 5
\]
span the CR distribution $H$ on $M$. 

The structure on $M$ turns out to be $2$-nondegenerate, and to verify this while calculating its CR symbol it is more convenient to use a different basis of $H$, namely the basis $(Y_1,\ldots, Y_4,Y_9)$ given by
\[
Y_1:=X_1-\frac{3x_5}{2}X_2+\frac{3x_5^2}{8}X_3+\left(\frac{x_5^3}{16}-\frac{f^{\prime\prime}(x_1)}{4}\right)X_4, 
\quad
Y_2:=X_2-\frac{x_5}{2}X_3-\frac{x_5^2}{8}X_4,
\]
\[
Y_3:=X_3+\frac{x_5}{2}X_4,
\,\,
Y_4:=X_4,
\,\,\mbox{ and }\,\,
Y_9:=X_5-x_1X_2-(x_1x_5+x_2)X_3-(x_1x_5^2+x_2x_5+x_3)X_4.
\]
Note, motivation for the indexing choice in the label $Y_9$ will become clear soon.
It is straightforward to check that the Levi kernel $K$ is the rank 1 distribution given by 
\[
K:=\mathrm{span}_{\mathbb{C}}\{Y_9\},
\]
so to describe the CR symbols of $M$  in matrix representations with respect to $(Y_1,\ldots, Y_4,Y_9)$ we need to calculate the Lie brackets of vector fields in $\left(Y_1,\ldots, Y_4,\overline{Y_1},\ldots, \overline{Y_4},Y_9, \overline{Y_9}\right).$ For notational convenience, let us also label
\[
Y_0:=4i\frac{\partial}{\partial y_0},
\quad
Y_{10}:=\overline{Y_9},
\quad\mbox{ and }\quad Y_{j+4}:=\overline{Y_j}\quad\forall\, 1\leq j\leq 4.
\]
Considering the Lie brackets between vector fields among $Y_0,\ldots, Y_{10}$, all of the nonzero brackets are determined by
\begin{align}\label{eg8.4 base manifold brackets a}
[Y_1,Y_5]=\frac{f^{(3)}(x_1)}{4}Y_4-\frac{f^{(3)}(x_1)}{4}Y_8,
\end{align}
\begin{align}\label{eg8.4 base manifold brackets b}
[Y_j,Y_{9-j}]=Y_0\quad\forall\, 1\leq j\leq 4,
\end{align}
\begin{align}\label{eg8.4 base manifold brackets c}
[Y_9,Y_1]=-\frac{1}{2}Y_2,
\quad
[Y_9,Y_2]=\frac{1}{2}Y_3,
\quad
[Y_9,Y_3]=\frac{3}{2}Y_4,
\end{align}
\begin{align}\label{eg8.4 base manifold brackets d}
[Y_9,Y_5]=Y_2-\frac{3}{2}Y_6,
\quad
[Y_9,Y_6]=Y_3-\frac{1}{2}Y_7,
\quad\mbox{ and }\quad
[Y_9,Y_6]=Y_4+\frac{1}{2}Y_8,
\end{align}
together with the Lie bracket's antisymmetry and compatibility with complex conjugation.

It follows from \eqref{eg8.4 base manifold brackets d} that the structure is indeed $2$-nondegenerate. It follows furthermore from \eqref{eg8.4 base manifold brackets a}, \eqref{eg8.4 base manifold brackets b}, \eqref{eg8.4 base manifold brackets c} and \eqref{eg8.4 base manifold brackets d} that, for every $z\in M$,  the matrices $H_\ell$ and $C$ representing the CR symbol on $M$ at $z$ with respect to the basis $\big(\left(Y_0\right)_z,\ldots,\left(Y_{10}\right)_z\big)$ (in the notation of Section \ref{CR structures with constant modified symbols}) are
\begin{equation}\label{eg8.4 symbol}
H_\ell=
\left(
\begin{array}{cccc}
    0 & 0 & 0 & 1  \\
    0 & 0 & 1 & 0  \\
    0 & 1 & 0 & 0  \\
    1 & 0 & 0 & 0 
\end{array}
\right)
\quad\mbox{ and }\quad
C=
\left(
\begin{array}{cccc}
    0 & 0 & 0 & 0  \\
    1 & 0 & 0 & 0  \\
    0 & 1 & 0 & 0  \\
    0 & 0 & 1 & 0 
\end{array}
\right).
\end{equation}

Regarding the modified symbols, for each $z\in M$, consider the modified symbol $\mathfrak{g}^{0,\mathrm{mod}}(\varphi)$ at the point $\varphi:\mathfrak{g}_{-1}\to \mathfrak{g}_{-1}(z)$ in the fiber $P^0_z$ of $P^0$ above $z$ defined by $\varphi(a_j):=(Y_j)_z$ where $(a_1,\ldots, a_8)$ is the basis of $\mathfrak{g}_{-1}$ compatible with $H_\ell$ as in \eqref{canon basis}. The modified symbol is determined by a non-canonical splitting $\mathfrak{g}_0^{\mathrm{mod}}(\varphi)=\mathfrak{X}_{0,2}\oplus \mathfrak{g}_{0,0}(\varphi)\oplus \mathfrak{X}_{0,-2}$ of the form in Section \ref{CR structures with constant modified symbols} determined by the formula \eqref{Xzero2} together with 
\begin{equation}\label{eg8.4 Xzero2}
\left(
\begin{array}{cc}
 \Omega& C\\
 0& -H_\ell^{-1}\Omega^TH_\ell
\end{array}
\right)
=
\left(
\begin{array}{cccccccc}
    0 & 0 & 0 & 0 & 0 & 0 & 0 & 0 \\
    -\frac{1}{2} & 0 & 0 & 0 & 1 & 0 & 0 & 0 \\
    0 & \frac{1}{2} & 0 & 0 & 0 & 1 & 0 & 0 \\
    0 & 0 & \frac{3}{2} & 0 & 0 & 0 & 1 & 0 \\
    0 & 0 & 0 & 0 & 0 & 0 & 0 & 0 \\
    0 & 0 & 0 & 0 & -\frac{3}{2} & 0 & 0 & 0 \\
    0 & 0 & 0 & 0 & 0 & -\frac{1}{2} & 0 & 0 \\
    0 & 0 & 0 & 0 & 0 & 0 & \frac{1}{2} & 0
\end{array}
\right)\in \mathfrak{csp}\big(\mathfrak{g}_{-1}\big),
\end{equation}
a formula that is derived directly from \eqref{eg8.4 base manifold brackets c} and \eqref{eg8.4 base manifold brackets d}. Note that $\eqref{eg8.4 Xzero2}$ represents an endomorphism with respect to the basis $(a_1,\ldots, a_8)$. Together \eqref{eg8.4 symbol} and \eqref{eg8.4 Xzero2} determine the modified symbol of $M$ at each point $z\in M$ and $\varphi\in P^0_z$, namely $\mathfrak{g}_0^{\mathrm{mod}}(\varphi)$ is spanned by
\begin{equation}\label{eg8.4 red symbol}
\left\{\left(
\begin{array}{cccccccc}
    t_4-3t_5 & 0 & 0 & 0 & 0 & 0 & 0 & 0 \\
    -\frac{t_1+3t_2}{2} & t_4-t_5 & 0 & 0 & t_1 & 0 & 0 & 0 \\
    0 & \frac{t_1-t_2}{2} & t_4+t_5 & 0 & 0 & t_1 & 0 & 0 \\
    t_3 & 0 & \frac{3t_1+t_2}{2} & t_4+3t_5 & 0 & 0 & t_1 & 0 \\
    0 & 0 & 0 & 0 & t_4-3t_5 & 0 & 0 & 0 \\
    t_2 & 0 & 0 & 0 & -\frac{3t_1+t_2}{2} & t_4-t_5 & 0 & 0 \\
    0 & t_2 & 0 & 0 & 0 & -\frac{t_1-t_2}{2} & t_4+t_5 & 0 \\
    0 & 0 & t_2 & 0 & -t_3 & 0 & \frac{t_1+3t_2}{2} & t_4+3t_5
\end{array}
\right)
\right\},
\end{equation}
and
\begin{equation}\label{eg8.4 mod symbol}
\left\{\left(
\begin{array}{cccccccc}
    t_1 & 0 & 0 & 0 & 0 & 0 & 0 & 0 \\
    t_2  & t_1 & 0 & 0 & 0 & 0 & 0 & 0 \\
    0 & -t_2 & t_1 & 0 & 0 & 0 & 0 & 0 \\
    0 & 0 & t_2 & t_1 & 0 & 0 & 0 & 0 \\
    0 & 0 & 0 & 0 & -t_1 & 0 & 0 & 0 \\
    0 & 0 & 0 & 0 & -t_2 & -t_1  & 0 & 0 \\
    0 & 0 & 0 & 0 & 0 & t_2 & -t_1  & 0 \\
    0 & 0 & 0 & 0 & 0 & 0 & -t_2 & -t_1 
\end{array}
\right)
\right\}.
\end{equation}
The modified symbol at $\varphi\in P^0_z$ and the natural action of the $P^0$  structure group on the set of modified symbols determine all other modified symbols at points on the fiber $ P^0_z$, and hence we have shown that the modified symbols of $M$ have no dependence on $z$ or $f$! Since a structure's set of modified symbols is a basic local invariant, it is thus natural to study the local equivalence problem for the family of structures of the form in \eqref{C6 example family}.

The bundle $P^0$ admits a reduction $P^{0,\mathrm{red}}$ (as described in Section \ref{reduction_and_nongenericity_section}) with a constant reduced modified symbol $\mathfrak{g}^{0,\mathrm{red}}$ whose degree zero component $\mathfrak{g}_0^{\mathrm{red}}$ is represented by the space of matrices in \eqref{eg8.4 red symbol}, so we can apply Theorem \ref{maintheorconst}. To apply Theorem \ref{maintheorconst}, we need to calculate the universal algebraic Tanaka prolongation $\mathfrak{u}\left(\mathfrak{g}^{0,\mathrm{red}}\right)$ of $\mathfrak{g}^{0,\mathrm{red}}$, and indeed this is straightforward to calculate using \eqref{eg8.4 red symbol} to obtain
\begin{align}\label{eg68.4 prolongation}
\mathfrak{u}\left(\mathfrak{g}^{0,\mathrm{red}}\right)=\mathfrak{g}^{0,\mathrm{red}}.
\end{align}
Note, while it is feasible to derive \eqref{eg68.4 prolongation} by hand, it can be reproduced efficiently using the Maple mathematical software's Differential Geometry package, and it even follows immediately with no calculation from \cite[Theorem 3.8]{Sykes2022}. Let us label the basis $(a_0,\ldots,a_{13})$ of $\mathfrak{g}^{0,\mathrm{red}}$ that extends the previously introduced basis $(a_1,\ldots,a_8)$ satisfying $a_0=[a_1,a_8]$ and with $a_{8+j}$ represented by the matrix in \eqref{eg8.4 red symbol} with $t_k=\delta_{j,k}$.

By  \eqref{eg68.4 prolongation} and Theorem \ref{maintheorconst}, the geometric prolongation procedure (detailed in Section \ref{Absolute parallelisms}) will yield a canonical absolute parallelism on a bundle diffeomorphic to $P^{0,\mathrm{red}}$. We begin this construction by defining a pre-normalized frame $\widetilde{E}_0,\ldots,\widetilde{E}_{13}$ on $P^{0,\mathrm{red}}$ as follows. 

For each point $z\in M$, let $\varphi_z\in P^{0\mathrm{red}}$ be the isomorphism $\varphi_z(a_j):=(Y_j)_z$, that is, $z\mapsto \varphi_z$ defines a section of $\mathrm{pr}:P^{0,\mathrm{red}}\to M$. Notice that $\pi:M\cap \{(z_0,\ldots, z_5)\,|\,z_5=0\}\to \pi(M)\cap\mathcal{N}$ is a diffeomorphism, so $\sigma:\mathcal{N}\cap\pi(M)\to P^{0,\mathrm{red}}$ given by
\[
\sigma(q):=\varphi_{\widetilde{\pi}^{-1}(q)}
\quad\mbox{ where }\quad
\widetilde{\pi}:=\left.\pi\right|_{M\cap \{(z_0,\ldots, z_5)\,|\,z_5=0\}}
\]
defines a section of $\pi\circ\mathrm{pr}:P^{0,\mathrm{red}}\to\mathcal{N}\cap \pi(M)$. It will be convenient to regard $\mathcal{N}\cap \pi(M)$ as a submanifold of $M$ embedded via the section $\widetilde{\pi}^{-1}$ of $\pi$, so in what follows we make this identification and, in particular, let the map $\Pi:P^{0,\mathrm{red}}\to\mathcal{N}\cap \pi(M)\subset M$ denote the projection
\[
\Pi:=\left(\left.\pi\right|_{M\cap \{(z_0,\ldots, z_5)\,|\,z_5=0\}}\right)^{-1}\circ\pi\circ\mathrm{pr}.
\]
Regarding the structure group $G_{0}^{\mathrm{red}}$ of $\Pi:P^{0,\mathrm{red}}\to\mathcal{N}\cap \pi(M)$ as a subgroup of $GL(\mathfrak{g}_{-1})$, it has the natural right action on $P^{0,\mathrm{red}}$ given by
\[
R_g(\varphi):=\varphi\circ g
\quad\quad\forall\,g\in G_{0}^{\mathrm{red}},\,\varphi\in P^{0,\mathrm{red}},
\]
and for each $g\in G_{0}^{\mathrm{red}}$ we define the submanifold $S_g$ of $P^{0,\mathrm{red}}$ using $\sigma$ by
\[
S_g:=\{R_g(\varphi)\,|\, \varphi\in \mathrm{Im}(\sigma)\}.
\]
Notice that the submanifolds $S_g$ foliate $P^{0,\mathrm{red}}$.

Now define the pre-normalized frame $\widetilde{E}_0,\ldots,\widetilde{E}_{13}$ on $P^{0,\mathrm{red}}$ by requiring
\begin{align}\label{cooridinate-free pre-normalized frame a}
\Pi_*\left(\widetilde{E}_{j}\right)_{R_g(\varphi)}:=\varphi\circ g (a_j)
\quad\quad\forall\, g\in G^{0,\mathrm{red}},\, \varphi\in \mathrm{Im}(\sigma),\,j\in\{0,\ldots, 8\},
\end{align}
where $\varphi\circ g (a_0)$ is defined via the natural extension of $\varphi\circ g$ to $\mathfrak{g}_{-}$ with $\varphi(a_0):=(Y_0)_{\mathrm{pr}(\varphi)}$ (i.e., defined such that the extension is a Lie algebra isomorphism),
\begin{align}\label{cooridinate-free pre-normalized frame b}
\left(\widetilde{E}_{j}\right)_{R_g(\varphi)}\in T_{R_g(\varphi)}S_g
\quad\quad\forall\, g\in G^{0,\mathrm{red}},\, \varphi\in \mathrm{Im}(\sigma),\,j\in\{0,\ldots, 8\},
\end{align}
and
\begin{align}\label{cooridinate-free pre-normalized frame c}
\left(\widetilde{E}_{j}\right)_{R_g(\varphi)}:=\left.\frac{d}{dt}\right|_{t=0} R_{e^{ta_j}}\circ R_g(\varphi)
\quad\quad\forall \varphi\in P^{0,\mathrm{red}},\, j\in\{9,\ldots, 13\}.
\end{align}
Indeed \eqref{cooridinate-free pre-normalized frame a}, \eqref{cooridinate-free pre-normalized frame b}, and \eqref{cooridinate-free pre-normalized frame c} determine $\widetilde{E}_0,\ldots,\widetilde{E}_{13}$ uniquely.
\begin{remark}
The frame $\widetilde{E}_0,\ldots,\widetilde{E}_{13}$ determines a Cartan connection $\widetilde{\omega}:TP^{0,\mathrm{red}}\to\mathfrak{g}^{0,\mathrm{red}}$ on the principal bundle $\pi\circ\mathrm{pr}:TP^{0,\mathrm{red}}\to \mathcal{N}\cap \pi(M)$ by the rule $\widetilde{\omega}(E_j)_{\varphi}=a_j$ for all $\varphi\in P^{0,\mathrm{red}}$. The connection is not canonical, however, because of its dependence on both the ad hoc choice of $\sigma$ and the $Y_j$ fields.
\end{remark}

Following the first geometric prolongation construction in Section \ref{Absolute parallelisms}, we need to calculate now the structure function $S_{\widetilde{\mathcal{H}}_{(z,t)}}$ associated with this frame $(\widetilde{E}_1,\ldots,\widetilde{E}_{13})$ via the identifications
\begin{align}\label{frame h-space correspondance a}
\widetilde{\mathcal{H}}_\psi=(\widetilde{{H}}_{(z,t)}^{-2},\widetilde{{H}}_{(z,t)}^{-1}),
\end{align}
\begin{align}\label{frame h-space correspondance aa}
\widetilde{{H}}_{(z,t)}^{-2}:=\mathrm{span}\left\{\left(\widetilde{E_0}\right)_{(z,t)}\right\}\pmod{\mathrm{span}\left\{\left(\widetilde{E_{9}}\right)_{(z,t)},\ldots,\left( \widetilde{E}_{13}\right)_{(z,t)}\right\}},
\end{align}
and
\begin{align}\label{frame h-space correspondance b}
\widetilde{{H}}_{(z,t)}^{-1}:\mathrm{span}\left\{\left(\widetilde{E}_{1}\right)_{(z,t)},\ldots,\left( \widetilde{E}_{8}\right)_{(z,t)}\right\},
\end{align}
and potentially then modify the frame so that the modified frame's structure function is suitably normalized. Notice that the vector fields $\widetilde{E}_0,\ldots,\widetilde{E}_{13}$ satisfy the conditions in \eqref{structure function b} and \eqref{structure function c} of the vector fields used in the definition \eqref{structure function a} of $S_{\widetilde{\mathcal{H}}_{(z,t)}}$. Accordingly, we can obtain $S_{\widetilde{\mathcal{H}}_{(z,t)}}$ by calculating Lie brackets of $\widetilde{E}_0,\ldots,\widetilde{E}_{13}$, that is we need to calculate the structure functions $\widetilde{T}_{j,k}^l$ of the parallelism $(\widetilde{E}_0,\ldots,\widetilde{E}_{13})$, which are the functions satisfying
\[
\left[\widetilde{E}_j,\widetilde{E}_k\right]:=\sum_{l=0}^{13}\widetilde{T}_{j,k}^l\widetilde{E}_l.
\]

We calculate these using \eqref{eg8.4 base manifold brackets a}-\eqref{eg8.4 base manifold brackets d}, \eqref{cooridinate-free pre-normalized frame a}-\eqref{cooridinate-free pre-normalized frame c}, and the matrix representation of $G_0^{\mathrm{red}}$ with respect to the basis $(a_1,\ldots, a_8)$ having general elements of the form
\begin{align}\label{eg8.4 structureGroup}
g=\mathrm{Exp}\big([\mbox{the matrix in \eqref{eg8.4 red symbol}}]\big).
\end{align}
Using \eqref{eg8.4 base manifold brackets a}-\eqref{eg8.4 base manifold brackets b}, \eqref{cooridinate-free pre-normalized frame a}-\eqref{cooridinate-free pre-normalized frame b} and \eqref{eg8.4 structureGroup}, we get
\begin{align}\label{pre-normalized structure coefficients a}
\left[\widetilde{E}_{0},\widetilde{E}_{j}\right]=0
\quad\quad\forall\,j\in\{1,\ldots, 8\},
\end{align}
and
\begin{align}\label{pre-normalized structure coefficients b}
\Pi_{*}\left(\left[\widetilde{E}_{j},\widetilde{E}_{k}\right]_{R_g(\varphi)}\right)&=\left[\sum_{r=1}^8 g_{r,j}Y_{r},\sum_{s=1}^8g_{s,k}Y_{s}\right]_{\mathrm{pr}(\varphi)}\nonumber\\
&=\sum_{r,s=1}^8g_{r,j}g_{s,k}\left[ Y_{r},Y_{s}\right]_{\mathrm{pr}(\varphi)}
\quad\forall\,\varphi\in\mathrm{Im}(\sigma),\,j,k\in\{1,\ldots, 8\},
\end{align}
which uniquely determines $\left[\widetilde{E}_{j},\widetilde{E}_{k}\right]$ because it is also horizontal (i.e., tangent to the $\{S_g\,|\,g\in G_0^{\mathrm{red}}\}$ foliation).
The remaining structure coefficients (i.e., $T_{j,k}^l$ with $9\leq j$) are, by construction, the same as the Lie algebra structure coefficients $a_{j,k}^l$ defined by  
\[
[a_j,a_k]:=\sum_{l=0}^{13}a_{j,k}^la_l,
\]
that is,
\begin{align}\label{pre-normalized structure coefficients c}
\widetilde{T}_{j,k}^l=a_{j,k}^l
\quad\quad\forall\, j\in\{9,\ldots, 13\}.
\end{align}

\begin{lemma}
At a point $\varphi\circ g\in P^{0,\mathrm{red}}$ with $\varphi\in \mathrm{Im}(\sigma)$ and $g$ as in \eqref{eg8.4 structureGroup}, the structure coefficients $\widetilde{T}_{j,k}^l(\varphi\circ g)$ of $(\widetilde{E}_{0},\ldots, \widetilde{E}_{13})$ are
\begin{align}\label{pre-normalized torsion a}
\widetilde{T}_{1,5}^4(\varphi\circ g)=\frac{f^{(3)}(x_1)g_{1,1}g_{5,5}}{4g_{4,4}}=\frac{f^{(3)}(x_1)e^{t_4-9t_5}}{4},
\quad
\widetilde{T}_{1,5}^8(\varphi\circ g)=-\widetilde{T}_{1,5}^4(\varphi\circ g),
\end{align} 
and
\begin{align}\label{pre-normalized torsion b}
\widetilde{T}_{j,k}^l(p)=a_{j,k}^l
\quad\quad\forall\,\widetilde{T}_{j,k}^l\mbox{ not determined by \eqref{pre-normalized torsion a}},
\end{align}
where $x_1$ in \eqref{pre-normalized torsion a} is the $x_1$ coordinate of $\mathrm{pr}(\varphi\circ g)\in M\subset \mathbb{C}^6$.
\end{lemma}
\begin{proof}

We have
\begin{align}\label{ad g hom}
\varphi\left([a_j,a_k]\right)&=\varphi\left([a_j,a_k]\right)\nonumber\\
&\equiv[Y_j,Y_k]_{\mathrm{pr}(\varphi)}\pmod{H\oplus \overline{H}=\langle Y_1,\ldots, Y_{10}\rangle}
\,\,\forall\, j,k\in\{1,\ldots, 8\}
\end{align}
because $\varphi(a_j)=(Y_j)_{\mathrm{pr}(\varphi)}$ and $\varphi:\mathfrak{g}_{-1}\to\mathfrak{g}_{-1}\big({\mathrm{pr}(\varphi)}\big)$ is the restriction of a Lie algebra isomorphism between $\mathfrak{g}_{-}$ and $\mathfrak{g}_{-}\big({\mathrm{pr}(\varphi)}\big)$.

By \eqref{pre-normalized structure coefficients b} and noting that all nonzero brackets between $Y_1,\ldots, Y_8$ appear in \eqref{eg8.4 base manifold brackets a}-\eqref{eg8.4 base manifold brackets b}, for each $\varphi\in \mathrm{Im}(\sigma)$ and $g\in G_0^{\mathrm{red}}$,
\begin{align}\label{1 5 bracket calculation a}
\Pi_{*}\left(\left[\widetilde{E}_{1},\widetilde{E}_{5}\right]_{R_g(\varphi)}\right)&=\sum_{r,s=1}^8g_{r,1}g_{s,5}\left[ Y_{r},Y_{s}\right]_{\pi(\varphi)}\\
&=\left(g_{1,1}g_{5,5}-g_{5,1}g_{1,5}\right)\left[ Y_{1},Y_{5}\right]_{\pi(\varphi)}+ c \left(Y_0\right)_{\pi(\varphi)}\\
&=\left(g_{1,1}g_{5,5}-g_{5,1}g_{1,5}\right)\left[ Y_{1},Y_{5}\right]_{\pi(\varphi)},
\end{align}
where $c$ represents some coefficient, and  the last equality asserting that in fact this $c$ is zero indeed holds by \eqref{ad g hom} because, since $[a_1,a_5]=0$,
\begin{align}
0=\varphi\circ g\left([a_1,a_5]\right)&=\varphi\left([g(a_1),_g(a_5)]\right)\\
&\equiv[\varphi\circ g(a_1),\varphi\circ g(a_5)]\pmod{H\oplus \overline{H}=\langle Y_1,\ldots, Y_{10}\rangle},
\end{align}
which together with \eqref{cooridinate-free pre-normalized frame a} shows that $\mathrm{pr}_{*}\left(\left[\widetilde{E}_{1},\widetilde{E}_{5}\right]_{R_g(\varphi)}\right)\in\mathrm{span}_{\mathbb{C}}\{Y_1,\ldots, Y_{10}\}$.

Exponentiating \eqref{eg8.4 red symbol} to calculate $g$ as in \eqref{eg8.4 structureGroup}, we get
\begin{align}\label{g i j values}
\quad\quad g_{1,j}=\delta_{1,j}e^{t_4-3t_5},
\,
g_{5,j}=\delta_{5,k}e^{t_4-3t_5},
\mbox{ and }
g_{j,k}=\delta_{j,k}e^{t_4+3t_5}
\,\,\forall\, j\in \{1,\ldots, 8\},\,k\in\{4,8\}.
\end{align}
Accordingly, by \eqref{cooridinate-free pre-normalized frame a}, for $j\in\{4,8\}$ and $\varphi\in \mathrm{Im}(\sigma)$, $\Pi_*\left(\widetilde{E}_j\right)_{R_g(\varphi)}=g_{4,4}\left(Y_j\right)_{\mathrm{pr}(\varphi)}$, and hence  applying \eqref{eg8.4 base manifold brackets a}  to \eqref{1 5 bracket calculation a} yields
\begin{align}\label{1 5 bracket calculation b}
\Pi_{*}\left(\left[\widetilde{E}_{1},\widetilde{E}_{5}\right]_{R_g(\varphi)}\right)&=\left(g_{1,1}g_{5,5}-g_{5,1}g_{1,5}\right)\left[ Y_{1},Y_{5}\right]_{\pi(\varphi)}\\
&=\frac{f^{(3)}(x_1)g_{1,1}g_{5,5}}{4g_{4,4}}\left(g_{4,4}\left(Y_4\right)_{\pi(\varphi)}-g_{4,4}\left(Y_8\right)_{\pi(\varphi)}\right) \\
&=\frac{f^{(3)}(x_1)g_{1,1}g_{5,5}}{4g_{4,4}}\Pi_*\left(\left(\widetilde{E}_{4}\right)_{R_g(\varphi)}-\left(\widetilde{E}_{8}\right)_{R_g(\varphi)}\right)\\
&=\frac{f^{(3)}(x_1)e^{t_4-9t_5}}{4}\Pi_*\left(\left(\widetilde{E}_{4}\right)_{R_g(\varphi)}-\left(\widetilde{E}_{8}\right)_{R_g(\varphi)}\right),
\end{align}
which, in particular, implies \eqref{pre-normalized torsion a}.

For \eqref{pre-normalized torsion b}, notice from \eqref{eg8.4 base manifold brackets a} and \eqref{eg8.4 base manifold brackets b} that
\begin{align}\label{j k bracket calculation a}
[Y_j,Y_k]=\sum_{l=0}^{13}a_{j,k}^lY_l=a_{j,k}^0Y_0,
\quad\quad\forall j\leq k \in\{1,\ldots,8\},\,(j,k)\neq(1,5)
\end{align}
so we actually have direct equality in \eqref{ad g hom} whenever $(j,k)\neq(1,5),(5,1)$ rather than just equivalence modulo $H\oplus \overline{H}$. By \eqref{g i j values}, if $(j,k)\neq(1,5),(5,1)$ then the right side of \eqref{pre-normalized structure coefficients b} does not have a nonzero multiple of $[Y_1,Y_5]$, and hence applying \eqref{cooridinate-free pre-normalized frame a}, \eqref{ad g hom}, and  \eqref{j k bracket calculation a} yields
\begin{align}\label{j k bracket calculation b}
\Pi_{*}\left(\left[\widetilde{E}_{j},\widetilde{E}_{k}\right]_{R_g(\varphi)}\right) &=\varphi\circ g\left(\left[a_j, a_k\right]\right)\\
&=a_{j,k}^0 \big(\varphi\circ g\left(a_0\right)\big)\\
&=\sum_{l=0}^{13}a_{j,k}^l\Pi_{*}\left(\widetilde{E}_{l}\right)_{R_g(\varphi)}
\quad\quad\forall j\leq k \in\{1,\ldots,8\},\,(j,k)\neq(1,5).
\end{align}
Together \eqref{pre-normalized structure coefficients a}, \eqref{pre-normalized structure coefficients c}, \eqref{1 5 bracket calculation b}, and \eqref{j k bracket calculation b} establish \eqref{pre-normalized torsion b}.
\end{proof}

Noticing that $\widetilde{T}_{1,11}^4=-1$ and $\widetilde{T}_{5,11}^8=1$, there is an apparent way to modify the frame $\widetilde{E}_0,\ldots,\widetilde{E}_{13}$ such that the structure coefficients in \eqref{pre-normalized torsion a} become zero. Specifically, consider the new frame $({E}_0,\ldots,{E}_{13})$ given by
\[
{E}_j:=
\begin{cases}
\widetilde{E}_{1}-\frac{f^{(3)}(x_1)e^{t_4-9t_5}}{4}\widetilde{E}_{11} &\mbox{ if }j=1\\
\widetilde{E}_{5}+\frac{f^{(3)}(x_1)e^{t_4-9t_5}}{4}\widetilde{E}_{11} &\mbox{ if }j=5\\
\widetilde{E}_{j} &\mbox{ otherwise},
\end{cases}
\]
with corresponding structure functions $T_{j,k}^l$ satisfying
\[
\left[{E}_j,{E}_k\right]:=\sum_{l=0}^{13}{T}_{j,k}^l{E}_l.
\]
Through direct calculation, one finds
\begin{align}\label{normalized torsion}
{T}_{j,k}^l=
\begin{cases}
\frac{f^{(4)}(x_1)e^{2t_4-12t_5}}{2}& \mbox{ if }(j,k,l)=(1,5,11)\\
a_{j,k}^l&\mbox{ otherwise, assuming }j\leq k.
\end{cases}
\end{align}

Letting $\mathcal{H}_{(z,t)}$ be defined as in \eqref{frame h-space correspondance a}, \eqref{frame h-space correspondance aa}, and \eqref{frame h-space correspondance b} (just without tildes decorating every object), we can calculate $S_{\mathcal{H}_{(z,t)}}$ using \eqref{normalized torsion}  and \eqref{structure function a} because ${E}_0,\ldots,{E}_{13}$ satisfy the conditions in \eqref{structure function b} and \eqref{structure function c}. In particular, it follows from \eqref{normalized torsion} that
\[
S_{\mathcal{H}_{(z,t)}}=0
\]
and hence ${H}_{(z,t)}$ is normalized (in the sense of Section \ref{Absolute parallelisms}) relative to any choice of normalization condition chosen to construct the first geometric prolongation. 

Since, by \eqref{eg68.4 prolongation}, $\mathfrak{g}_{1}=0$, the first geometric prolongation is the bundle $P^1\to P^{0,\mathrm{red}}$ whose fiber over a point $(z,t)\in P^{0,\mathrm{red}}$ consists of the single element ${H}_{(z,t)}$. Since  $P^1$ and $P^{0,\mathrm{red}}$ are diffeomorphic, we identify them as we build the second geometric prolongation, and, in particular, regard $({E}_0,\ldots,{E}_{13})$ as a frame on $P^1$.

For the second prolongation, we define the graded horizontal spaces ${\mathcal{H}}_{(z,t)}=(\mathcal{H}_{(z,t)}^{-2},\mathcal{H}_{(z,t)}^{-1},\mathcal{H}_{(z,t)}^{0})$ on $P^1$ by
\[
\mathcal{H}_{(z,t)}^{-2}=\mathrm{span}\left\{\left(E_0\right)_{(z,t)}\right\},
\quad
\mathcal{H}_{(z,t)}^{-1}=\mathrm{span}\left\{\left(E_1\right)_{(z,t)},\ldots, \left(E_8\right)_{(z,t)}\right\},
\]
and
\[
\mathcal{H}_{(z,t)}^{0}=\mathrm{span}\left\{\left(E_9\right)_{(z,t)},\ldots, \left(E_{13}\right)_{(z,t)}\right\}.
\]
Similar to our calculations for the first prolongation, the structure functions $S_{\mathcal{H}_{(z,t)}}\in\mathcal{S}_1$  (defined in \eqref{structure function degree 1 a}) for the second prolongation can easily calculated using \eqref{normalized torsion}, namely
\begin{align}\label{second prol str fun}
S_{\mathcal{H}_{(z,t)}}(a_j,a_k)=
\begin{cases}
\frac{f^{(4)}(x_1)e^{2t_4-12t_5}}{2}a_{11} &\mbox{ if }(j,k=0)\\
0&\mbox{ otherwise, assuming }j\leq k.
\end{cases}
\end{align}
Since $\mathfrak{g}_{1}=0$, it is easily seen that \eqref{second prol str fun} does not lie in the image of the antisymetrization operator $\partial^1_{(z,t)}$ (defined in \eqref{anti-symmetrization for P1}). Hence, for the second prolongation construction, we should choose the normalization condition $N_1$ to be any subspace in $\mathcal{S}_1$ satisfying  \eqref{gOne normalization b} and containing the structure function in \eqref{second prol str fun}. Relative to this normalization condition, $\mathcal{H}$ is, of course normalized. Notice also, that if any such normalization condition were fixed then  the analogously constructed $\mathcal{H}$ built for a different $f$ would also be normalized, that is, our normalization condition does not depend on $f$. 

Since $\mathfrak{g}_{1}$ vanishes, the prolongation procedure completes on the second step, and the frame $(E_0,\ldots,E_{13})$ that we constructed is therefore parallel to a canonical absolute parallelism on $P^2\cong P^{0,\mathrm{red}}$, and the frame $(E_0,\ldots,E_{13})$ itself is determined canonically by the parallelism and the choice of basis $(a_0,\ldots, a_{13})$ of $\mathfrak{g}^{0,\mathrm{red}}$. From the classical theory of absolute parallelisms (e.g., as in \cite{AS04,olver1995equivalence,stern}),
fundamental invariants of the resulting parallelism consist of its structure functions  given in \eqref{normalized torsion} and their directional derivatives in the direction of the vector fields of the frame  also called covariant or 
coframe derivatives in the terminology of \cite [chapter 7, \S 4]{stern} and
\cite[chapter 8]{olver1995equivalence}.

The only nonconstant set of functions in this set of invariants are
\[
I_1:=T_{1,5}^{11}=\frac{f^{(4)}(x_1)e^{2t_4-12t_5}}{2},
\]
and
\begin{align}\label{second invariant}
I_{j}:=\overbrace{E_1\circ\cdots\circ  E_1}^{\mbox{$j-1$ copies}}(I_1)=\frac{f^{(j+3)}(x_1)e^{(j+1)t_4-(3j+9)t_5}}{2}
\quad\quad\forall\, j>1,
\end{align}
as
$
E_1\equiv e^{t_4-3t_5}\frac{\partial}{\partial x_1}\pmod{\mathrm{span}\left\{\frac{\partial}{\partial x_0},\frac{\partial}{\partial x_2},\ldots,\frac{\partial}{\partial x_{5}},\frac{\partial}{\partial y_0},\ldots,\frac{\partial}{\partial y_5}\right\}}$, $E_{12}=\frac{\partial}{\partial t_4}$, 
\[
E_{13}\equiv\frac{\partial}{\partial t_5} \pmod{\mathrm{span}\left\{\frac{\partial}{\partial x_0},\frac{\partial}{\partial x_2},\ldots,\frac{\partial}{\partial x_{5}},\frac{\partial}{\partial y_0},\ldots,\frac{\partial}{\partial y_5},\frac{\partial}{\partial t_{1}},\frac{\partial}{\partial t_{2}}, \frac{\partial}{\partial t_{3}}\right\}}
\]
and all of the other $E_j$ have no $\frac{\partial}{\partial x_{1}},\frac{\partial}{\partial t_{4}},\frac{\partial}{\partial t_{5}}$ components.

\begin{remark}
\label{invariants_on _M}
First, note that 
each $I_j$ is a relative invariant, i.e., its vanishing/non-vanishing is an invariant property. Second, 
taking appropriate ratios of invariants $I_j$, one can easily  obtain invariants on the original hypersurface. In more detail given any two nonincreasing tuples (multi-indices)  of positive integers 
$K=(k_1,\ldots, k_s)$ and $L=(l_1,\ldots, l_s)$ with $|K|=|L|$, where as usual $|K|$ denotes the sum of the integers in the tuple $K$, and assuming that $I_{l_t}(x_1)\neq 0$ for all $1\leq t\leq s$, the following function on the original hypersurface is an invariant:
\begin{equation}
 J_{K, L}:=\cfrac{
 \prod_{t=1}^s I_{k_t}} {\prod_{t=1}^s I_{l_t}}=\cfrac{
 \prod_{t=1}^sf^{(k_t+3)}(x_1)} {\prod_{t=1}^s f^{(l_t+3)}(x_1)}
 \end{equation}
 In particular, if  $f^{(5)}(x_1)\neq 0$, 
 then 
\begin{equation}
\label{456}
J_{(1,3), (2,2)}=\cfrac{f^{(4)}(x_1) f^{(6)}(x_1)}{\bigl(f^{(5)}(x_1)\bigr)^2}
\end{equation}
is a well-defined invariant. 
\end{remark}

We can derive from the set $\{I_j\}$ a criterion for local equivalence between structures $M_f$ and $M_{\widehat{f}}$, and an efficient way to complete this is through an application of \cite[Theorem 5.5]{AS04} summarized in the proof of the following lemma.

\begin{theorem}\label{eg8.4 equivalence solution 2}
For two real-analytic functions $f,\widehat{f}:\R\to\R$, the germs of the hypersurfaces $M_{f}$ and $M_{\widehat{f}}$ (defined by \eqref{C6 example family}) at points
\[
p^\prime=({x}^{\prime}_0,\ldots,{x}^{\prime}_5,{y}^{\prime}_0,\ldots,{y}^{\prime}_5)\in M_f
\]
and 
\[
p^{\prime\prime}=({x}^{\prime\prime}_0,\ldots,{x}^{\prime\prime}_5,{y}^{\prime\prime}_0,\ldots,{y}^{\prime\prime}_5)\in M_{\widehat{f}}
\]
respectively are equivalent if and only if $\widehat{f}^{(4)}(x)=af^{(4)}(bx+c)$ for all $x$ near $x_1^\prime$ and some nonzero numbers $a$, $b$ and $c$ such that $bx_1^\prime+c=x_1^{\prime\prime}$.
\end{theorem}

\begin{proof}
For $f$ and $\widehat{f}$, let 
\[
\{P^2,E_0,\ldots, E_{13},I_1,I_2,\ldots\}
\quad\mbox{ and }\quad
\{\widehat{P}^2,\widehat{E}_0,\ldots, \widehat{E}_{13},\widehat{I}_1,\widehat{I}_2,\ldots\}
\]
be the objects constructed above for the respective hypersurfaces $M_{f}$ and $M_{\widehat{f}}$.

By  \cite[Theorem 5.5]{AS04} (or, its alternative formulation as the Nagano principle in \cite[Chapter 5.7]{AS04}), there exists a diffeomorphism $\Phi:P^{2}\to \widehat{P}^2$ for which $\Phi_* E_j=\widehat E_j$ over a neighborhood of some point $p\in P^{2}$ for all $j$ if and only if
\begin{align}\label{eg8.4 eq condition 1}
f^{(k+3)}(x_1)e^{(k+1)t_4-(3k+9)t_5}=2I_k(\psi)=2\widehat{I}_k\big(\widehat{\psi}\big)=\widehat{f}^{(k+3)}(\widehat{x}_1)e^{(k+1)\widehat{t}_4-(3k+9)\widehat{t}_5}
\end{align}
for all $k\in\mathbb{N}$, where each where $\widehat{\psi}$  in \eqref{eg8.4 eq condition 1} denotes $\Phi(\psi)$ and \eqref{eg8.4 eq condition 1} refers to components of $\psi$ and $\widehat{\psi}$ in the local coordinates
\[
\psi=(x_0,\ldots,x_5,y_0,\ldots,y_5,t_1,\ldots,t_5)\in P^2
\]
and 
\[
\Phi(\psi)=(\Phi_1,\ldots,\Phi_{17})=(\widehat{x}_0,\ldots,\widehat{x}_5,\widehat{y}_0,\ldots,\widehat{y}_5,\widehat{t}_1,\ldots,\widehat{t}_5)\in \widehat{P}^2
\]
of the ambient spaces in which we constructed $P^2$ and $\widehat{P}^2$ previously in this section.
Therefore, $M_f$ and $M_{\widehat{f}}$ are equivalent near $\psi$ and $\widehat{\psi}$ if and only if \eqref{eg8.4 eq condition 1} holds for all $k$.

Accordingly, condition \eqref{eg8.4 eq condition 1} simplifies to
\[
f^{(k+4)}(x_1)=\widehat{f}^{(k+4)}\big(\widehat{x}_1\big)e^{(k+2)(\widehat{t}_4-t_4)-(3k+12)(\widehat{t}_5-t_5)}=ab^{k}\widehat{f}^{(k+4)}\big(\widehat{x}_1\big)
\quad\quad\forall \,k\in\mathbb{N}\cup\{0\}
\]
with setting $a=e^{2(\widehat{t}_4-t_4)-12(\widehat{t}_5-t_5)}$ and $b=e^{(\widehat{t}_4-t_4)-3(\widehat{t}_5-t_5)}$, and hence
\begin{align}\label{equivalence solution relations}
\left.\frac{\partial^k}{\partial x^k}\right|_{x=x_1}\left( f^{(4)}(x)\right)=\left.\frac{\partial^k}{\partial x^k}\right|_{x=x_1}\left( a\widehat{f}^{(4)}(bx+c)\right),
\quad\quad\forall \,k\in\mathbb{N}\cup\{0\}
\end{align}
with $c=\widehat{x}_1-bp_2$. Since the prolongation procedure translates local equivalence on the base space to local equivalence in the bundle on which the parallelism is constructed, germs of $M_{f}$ and $M_{\widehat{f}}$ at $p^\prime$ and $p^{\prime\prime}$ are equivalent if and only if  germs of $(E_1,\ldots, E_n)$ and $(\widetilde{E}_1,\ldots, \widetilde{E}_n)$  are equivalent at some points in the respective fibers over $p^\prime$ and $p^{\prime\prime}$. Letting $\psi$ and $\widehat{\psi}$ be points over $p^\prime$ and $p^{\prime\prime}$, we get that germs of $M_{f}$ and $M_{\widehat{f}}$ at $p^\prime$ and $p^{\prime\prime}$ are equivalent if and only if \eqref{equivalence solution relations} holds with $x_1^\prime$ and $x_1^{\prime\prime}$ replacing $x_1$ and $\widehat{x}_1$ respectively, which completes the proof because $f$ and $\widehat{f}$ are real-analytic.
\end{proof}

For Theorem \ref{eg8.4 equivalence solution 2} we applied \cite[Theorem 5.5]{AS04}, which gives local equivalence between sets of vector fields in the real analytic category. One could alternatively apply directly \cite[Theorem 4.1]{stern} on local equivalence of frames to characterize local equivalence in terms of a finite set of conditions rather than the infinite set in \eqref{eg8.4 eq condition 1}. For this, following \cite [chapter 7, \S 4]{stern}, given a nonnegative integer $s$ let $\mathcal F_s:=\{I_j\}_{j=1}^{s+1}$ be the family of nontrivial structure functions of our parallelism together with the covariant derivatives up to the order $s$ and 
$k_{s}(\varphi)=\dim \mathrm{span} \{d\, I_j(\varphi)\}_{j=1}^{s+1}$ be the rank of the family $\mathcal F_s$. From \eqref{second invariant} it is easy to see that $k_{s}(\varphi)$ is equal to the rank of the $(s+1)\times 3$ matrix 
\begin{equation}
    \label{matrix_for rank}
    \mathcal M_s(\varphi):=\begin{pmatrix} 
    f^{(5)}(x_1)& 2 f^{(4)} (x_1)& -12f^{(4)} (x_1)\\
    f^{(6)}(x_1)& 3 f^{(5)} (x_1)& -15f^{(5)} (x_1)\\
     f^{(7)}(x_1)& 4 f^{(6)} (x_1)& -16f^{(6)} (x_1)\\
     \vdots&\vdots&\vdots\\
     f^{(s+5)}(x_1)& (s+2) f^{(5)} (x_1)& -(3s+12)f^{(5)} (x_1)
    \end{pmatrix}.
\end{equation}
Recall that  the parallelism is called regular at $\varphi_0$ if there exists $s$ such that $k_{s+1}$ is constant in a neighborhood of $\varphi_0$ and $k_s(\varphi_0)=k_{s+1}(\varphi_0)$. The smallest such $s$, denoted by  $\mathfrak{o}$, is called the order of parallelism at $\varphi_0$ and in this case one says that  the parallelism has rank $k:=k_s(\varphi_0)$ at $\varphi_0$. Clearly, in the considered case, $k\leq 3$ and $k\leq \mathfrak{o}+1$. Also, since $k_s(\varphi)$ equals the rank of $\mathcal{M}_s(\varphi)$ in  \eqref{matrix_for rank}, it follows that the notions of regularity, the order, and  the rank of our parallelism can be related to a point on the original hypersurface.

By \cite[Chapter 7, Theorem 4.1]{stern}, if our parallelism has order $\mathfrak{0}$ and rank $k$ at $\varphi_0$, then there exist indices  $1\leq j_1<\ldots <j_k\leq \mathfrak{o}+1$ such that 
\begin{equation}
\label{derived_invariants}
I_j=G_j(I_{j_1},\ldots I_{j_k})
\end{equation}
for some smooth functions $G_j$ defined in a neighborhood of $(I_{j_1}(\varphi_0),\ldots I_{j_k}(\varphi_0))\in \mathbb R^k$.

\begin{theorem}[immediate corollary of {\cite[Chapter 7, Theorem 4.1]{stern}}]\label{derived_invariants dependence theorem}
A germ of another hypersurface $M_{\widehat{f}}$ at a point $\widehat{p}_0$ is equivalent to the germ of the original hypersurface $M_f$  at a point $p_0$, for which the canonical parallelism is regular  of order $\mathfrak o$  and rank $r$ at some (and therefore any)  $\varphi_0$ in the fiber of $p_0$, if and only if the canonical parallelism for the hypersurface $M_{\widehat{f}}$ is regular of order $\mathfrak o$  and rank $r$ at some (and therefore any) point  $\widehat{\varphi}_0$ in the fiber of $\widehat{p}_0$ and there exist $\varphi_0$ and $\widehat{\varphi}_0$ on the fibers of $p_0$ and $\widehat{p}_0$, respectively, such that the  functions $G_j$ and $\widehat G_j$  with $j\leq \mathfrak o +1$ corresponding to
the canonical parallelisms of $M_f$ and $M_{\widehat f}$, respectively, coincide.
\end{theorem}

Below we list examples that will be relevant for the analysis of locally homogeneous cases (in the proof of Theorem \ref{section 10 homogeneity classification}). In  these examples we  assume that $x_1^0$ is the $x_1$-component in the coordinate representation of $\varphi_0$:
\begin{enumerate}
\item If $f^{(4)} \equiv 0$ then the parallelism at $\varphi_0$ is regular  of  order $0$ and rank $0$;
\item  If $f^{(4)}(x_1^0)\neq 0$ but  $f^{(5)} \equiv 0$, then the parallelism is regular of  order $0$ and rank $1$
\item  If $f^{(4)}(x_1^0)\neq 0$, $f^{(5)} (x_1^0)\neq 0$, and $\det \mathcal{M}_2\equiv 0$, then the parallelism is of order $1$ and of rank $2$. Note that if  $f^{(5)} (x_1^0)\neq 0$ then the condition $\det M_2\equiv 0$ is equivalent to the condition that the invariant $J_{(1,3),(2,2)}=\cfrac{f^{(4)} f^{(6)}}{\bigl(f^{(5)}\bigr)^2}$
is constant  in a neighborhood of $\varphi_0$, because 
\[
\det M_2=6\left(f^{(5)}\right)^3\frac{d}{dx_1}\left(J_{(1,3),(2,2)}\right).
\]
\end{enumerate}

{\renewcommand{\arraystretch}{1.4}
\begin{table}
\begin{tabular}{|l|c|}
    \hline
    formula of $f$ for which $M_f$ is homogeneous & $M_f$ symmetry group dimension \\\hline\hline
    $f(x)=0$ & 14 \\\hline
    $f(x)=x^4$ & 13 \\\hline
    $f(x)=x^{a}$ with $x>0$ for some real number $a\not\in\{0,1,2,3,4\}$& 12\\\hline
    $f(x)=e^x$& 12 \\\hline
    $f(x)=\ln(x)$ with $x>0$& 12 \\\hline
    $f(x)= x \ln(x)-x$ with $x>0$& 12 \\\hline
    $f(x)=2 x^2\ln(x)-3x^2$ with $x>0$& 12 \\\hline
    $f(x)=6 x^3\ln(x)-11x^3$ with $x>0$& 12 \\\hline
\end{tabular}
\caption{Homogeneous structures of Theorem \ref{section 10 homogeneity classification}}\label{homogeneous structures table}
\end{table}
}

\begin{theorem}\label{section 10 homogeneity classification}
The hypersurface $M_{\widehat{f}}$ has a locally homogeneous structure on one of its open subsets $ U\subset M_{\widehat{f}}$ if and only if $u(x)=\widehat{f}^{(4)}(x)$ is a solution to a differential equation of the form $u\cdot u^{\prime\prime}=c(u^\prime)^2$ for some $c\in \mathbb{R}$, which in turn happens if and only if the structure on $U$ is locally equivalent to the structure on $M_f$ defined by one of the formulas Table \ref{homogeneous structures table}, also displaying the dimension of each structure's symmetry group.

Each formula in Table \ref{homogeneous structures table} defines a different homogeneous structure. In particular, we have a $1$-parametric family of non-equivalent homogeneous structures parameterized by $a$.
\end{theorem}

\begin{proof}
If $M_f$ is homogeneous, then by Remark \ref{invariants_on _M} either 
$f^{(5)}\equiv 0$ or 
\begin{equation}
\label{characterizing ODE}
J_{(1,3), (2,2)}=\cfrac{f^{(4)} f^{(6)}}{\bigl(f^{(5)}\bigr)^2}\equiv c
\end{equation}
for some constant $c\in \mathbb{R}$.

In the former case, by Theorem \ref{eg8.4 equivalence solution 2}, $M_f$ is equivalent to one of the models in the first two rows of Table \ref{homogeneous structures table}.
In the latter case a general solution to \eqref{characterizing ODE} has the form
\begin{align}\label{characterizing ODE solution}
    \widehat{f}^{(4)}(x)=
    \begin{cases}
    c_1\left((1-c)x+c_2\right)^{\tfrac{1}{1-c}} &\mbox{ if } c\neq 1\\
    c_1e^{c_2x}
    \end{cases},
\end{align}
and if $\widehat{f}$ satisfies \eqref{characterizing ODE solution} for some $(c_1,c_2,c)$ then, by Theorem \ref{eg8.4 equivalence solution 2}, $M_{\widehat{f}}$ is equivalent to $M_f$ for some $f$ given by one of the formulas in Table \ref{homogeneous structures table}, where the last four lines in the table corresponds to $c$ from the first line of \eqref{characterizing ODE solution} equal to $\tfrac{5}{4}$, $\tfrac{4}{3}$, $\tfrac{3}{2}$, and  $2$ respectively.

To prove the converse (i.e., $M_f$ is homogeneous if $f$ is given by one of the formulas in Table \ref{homogeneous structures table}),  we note that they correspond to one of the items (1)-(3) in the list of examples given above immediately preceding Theorem \ref{section 10 homogeneity classification}.
\begin{enumerate}
\item Item (1) corresponds to the first line of Table \ref{homogeneous structures table} and in this case all invariants $I_j$ are zero which implies that the model is homogeneous with a $14$-dimensional local group of symmetries, as $\dim  P^{0,\mathrm{red}}=14$.

\item Item (2) corresponds to the second line of  Table \ref{homogeneous structures table}. Since the order is $0$ and the rank is $1$, by \eqref{derived_invariants}
all invariants $I_j$ are functions of $I_1$, so they are constant on any level set of $I_1$. Since these level sets is project surjectively to $M_f$ we get the homogeneity, and since they have codimension 1 in $P^2$ we get the  conclusion on the dimension of the local symmetry group.

\item Item (3) corresponds to the remaining lines of  Table \ref{homogeneous structures table}. Since the order is $1$ and the rank is $2$ , by \eqref{derived_invariants},
all invariants $I_j$ are functions of $I_1$ and $I_2$, so they are constant on any common level set of $I_1$ and $I_2$. Since those common level sets  project surjectively to $M_f$ we get homogeneity and since they have codimension $2$ we get the  conclusion on the dimension of the local symmetry group.
\end{enumerate}
\end{proof}

\begin{remark}\label{eg8.4 equivalence solution 2 remark}
Theorem \ref{eg8.4 equivalence solution 2} shows that, despite having the same constant modified symbols, the moduli space of structures of the form in \eqref{C6 example family} is quite large, specifically parameterized by equivalence classes of real-analytic functions on $\R$ under the equivalence relation $f^{(4)}(x)\sim af^{(4)}(bx+c)$
(i.e., an infinite dimensional vector space modulo the action of a $7$-dimensional Lie group, where this group's $7$ generators come from the three parameters $a$, $b$, and $c$ in Theorem \ref{eg8.4 equivalence solution 2} and a $4$-dimensional group that normalizes the third jet of $f$).

This paper's main results apply in the category of constant symbol structures. One may wonder if there is sufficient variety among geometries in this category to warrant treating its local equivalence problems, and this example shows that there is indeed such rich variety. 
Theorem \ref{section 10 homogeneity classification} shows that infinitely many locally non-equivalent structures of the form in \eqref{C6 example family} are locally homogeneous, so even when restricting further to considering homogeneous structures with given symbols (or even modified symbols) one encounters an appreciable local equivalence problem.
\end{remark}

\begin{appendix}
\section{Non-constant symbol structures}\label{Non-constant symbol structures}

Starting from dimension $7$ the moduli space of CR symbols is not discrete. Hence $2$-nondegenerate hypersurface-type CR manifolds with non-constant CR symbol are generic. The absolute parallelisms of Theorems \ref{maintheorm} and \ref{maintheorconst} are, however, constructed for structures whose CR symbol is constant. There is a natural development of these constructions to build absolute parallelisms for generic structures with non-constant symbol, which we outline in this appendix. 
A notable application of this generalized construction is that it yields parallelisms whose infinitesimal symmetries contain the  CR manifold's infinitesimal symmetries, and thus the construction yields upper bounds on the dimension of CR manifold's automorphism group. These upper bounds can moreover be calculated through a completely algorithmic Tanaka prolongation procedure.

For the non-constant symbol setting, we will generalize the definition of $P^0$ appropriately, after which an absolute parallelism construction follows via the exact same procedure developed above for the constant symbol setting.

Let $\mathfrak{g}^0$ be a CR symbol with a decomposition as in \eqref{CRsymbol_2}. The group $\mathrm{Aut}(\mathfrak{g}_{-})\cong CSp(\mathfrak{g}_{-1})$ of automorphisms of $\mathfrak{g}_{-}$ that preserve the decomposition $\mathfrak{g}_{-}=\mathfrak{g}_{-2}\oplus \mathfrak{g}_{-1}$ has a natural action on $\mathfrak{g}_{-}\rtimes \mathfrak{der}(\mathfrak{g}_{-})$ given by
\[
g.(v,w)=\left(g(v),gwg^{-1}\right)\quad\quad\forall\, g\in \mathrm{Aut}(\mathfrak{g}_{-}),\,(v,w)\in\mathfrak{g}_{-2}\oplus \mathfrak{g}_{-1}.
\]
Define $S(\mathfrak{g}^0)$ to be the orbit of the symbol $\mathfrak{g}^0$ in the appropriate Grassmannian of $\mathfrak{g}_{-}\rtimes \mathfrak{der}(\mathfrak{g}_{-})$ under the action of $\mathrm{Aut}(\mathfrak{g}_{-})$, that is,
\[
S(\mathfrak{g}^0):=\left\{V\subset \mathfrak{g}_{-}\rtimes \mathfrak{der}(\mathfrak{g}_{-})\,\left|\,V=g.\mathfrak{g}^0\mbox{ for some }g\in \mathrm{Aut}(\mathfrak{g}_{-})\right.\right\},
\]
so $S(\mathfrak{g}^0)$ is a natural equivalence class of abstract CR symbols represented by $\mathfrak{g}^0$. Each element $g.\mathfrak{g}^0$ in $S(\mathfrak{g}^0)$ has a decomposition as in \eqref{CRsymbol_2} whose $(i,j)$ component is given by applying $g$ to the $(i,j)$ component of $\mathfrak{g}^0$.

Recall that our definition of $P^0$ for a structure with constant symbol involves fixing some representative abstract symbol $\mathfrak{g}^0$ that is equivalent to $\mathfrak{g}^0(p)$ for all $p\in M$. In other words we chose an element in $S\big(\mathfrak{g}^0(p)\big)$ from which we define $P^0$ and build the subsequent absolute parallelism. Choosing a different element in $S\big(\mathfrak{g}^0(p)\big)$ would yield a formally different absolute parallelism,  but all parallelisms obtained in this way are identified under the action of $\mathrm{Aut}(\mathfrak{g}_{-})\cong CSp(\mathfrak{g}_{-1})$. Once the choice is made one can use this same choice to every structure with given constant symbol to build the canonical parallelism and to some the equivalence problem. 

This perspective suggests a natural generalization of $P^0$ for non-constant symbol structures. Let $\pi_S:S\to  M$ be the fiber bundle over  $M$ whose fiber $S_p$ over a point $p\in M$ is
\[
S_p:=S\big(\mathfrak{g}^0(p)\big)\quad\quad\forall\, p\in  M,
\]
and let $s$ be a smooth section of $S$. Let $\sigma_{i,j}(p)$ (respectively, $\sigma_{-}(p)$) denote the $(i,j)$ component (respectively, negatively graded part) of abstract CR symbol $\sigma(p)$ for all $p\in M$. Each $\sigma_{-}(p)$ is isomorphic to a Heisenberg algebra $\mathfrak{g}_{-}$ of appropriate dimension, so let us fix the identifications
\[
\mathfrak{g}_{-}=\sigma_{-}(p)\quad\quad \forall p\in M.
\]

Given such a section $\sigma : M\to S$, we will say that a CR structure is \emph{of type $\sigma$} if for every point $p\in M$ the CR  symbol of the structure at $p$ is isomorphic to $\sigma(p)$. From this point of view, the case of structures with constant symbol corresponds to constant sections $\sigma$. Obviously, if a CR structures of type $\sigma$ is equivalent to another CR structure via a diffeomprphism $\Phi$, then the latter CR structure is of type $\sigma \circ \Phi$. In many cases already this necessary condition give very strong restrictions on $\Phi$ and may prove that such $\Phi$ does not exists.
However, if one assumes that for two CR structures of type $\sigma$ and $\widetilde \sigma$ on $M$ there exists a diffeomorphism $\Phi$ of $M$ such that $\widetilde \sigma=\sigma\circ \Phi$, then the construction of the canonical absolute parallelism for CR structures of given type $\sigma$ sketched below will in principle give necessary and sufficient conditions for this $\Phi$to be an isomorphism between these structures.

Define $\mathrm{pr}:P^0_\sigma\to  M$ to be the fiber bundle whose fiber $\mathrm{pr}^{-1}(p)$ over a point $p\in  M$ is defined by 
\begin{align}
\mathrm{pr}^{-1}(p)=\left\{
\varphi:\mathfrak{g}_{-}=\sigma_{-}(p)\to \mathfrak g_-(p) \left|\ 
\parbox{8.6cm}{
$\varphi(\mathfrak \sigma_{i,j}(p))=\mathfrak g_{i,j}(p)\quad\forall\, (i,j)\in \{(-1, \pm 1),(-2, 0)\}$,\\
$\varphi^{-1}\circ \mathfrak g_{0,\pm2}(p)\circ\varphi=\sigma_{0,\pm 2}(p)$, and \\
$\varphi([ y_1, y_2])=[\varphi( y_1),\varphi( y_2)] \quad\forall\, y_1, y_2\in\sigma_{-}(p)$
}
\right.\right\},
\end{align}
and suppose $\mathrm{pr}^{-1}(p)$ has constant dimension, which is not a very restrictive assumption because it will automatically have constant dimension near generic points in $M$.

By using $P^0_\sigma$ in place of $P^0$, the definitions of $\Re(P^0_\sigma)$ and of a modified symbol $\mathfrak{g}^{0,\mathrm{mod}}(\psi)$ associated with each point in $\psi\in P^0_\sigma$ remain well defined, but we should make the minor adjustment of defining $\mathfrak{g}^{0,\mathrm{mod}}(\psi):=\mathrm{span}_{\mathbb{C}}\left(\mathrm{Im}\left.\theta_0\right|_{T_\psi \left(P^0_\sigma\right)_{\Pi(\psi)}}\right)$ where this additional complex span is needed because, for simplicity in this section, we only defined $P^0_\sigma$ as a bundle over $M$ rather than $\mathbb{C} M$. The universal Tanaka prolongation of $\mathfrak{g}^{0,\mathrm{mod}}(\psi)$ also remains well defined, and for a regular point $\psi\in P^0_\sigma$, as defined in Definition \ref{regular points in Pzero}, an absolute parallelism on a fiber bundle over a neighborhood of $\psi$ in $P^0$ can be obtained using exactly the procedure outlined in Section \ref{Absolute parallelisms}. By construction, the (local) symmetry group of $M$ acts (locally) faithfully on the total space $\Re(P^0_\sigma)$ of this fiber bundle by symmetries of the parallelism. The dimension of this total space matches the dimension of the universal Tanaka prolongation of $\mathfrak{g}^{0,\mathrm{mod}}(\psi)$, which is consequently an upper bound for the dimension of CR manifold's algebra of infinitesimal symmetries.

\end{appendix}

\end{document}